\documentclass[reqno,english,11pt]{amsart}

\usepackage{amsmath,amsfonts,amssymb,graphicx,amsthm,url}
\usepackage[noadjust]{cite}
\usepackage{stmaryrd}
\usepackage{mathrsfs,booktabs,tabularx}
\usepackage{xifthen,xcolor,tikz,setspace}
\usetikzlibrary{decorations.pathmorphing,patterns,shapes,calc,decorations}
\usetikzlibrary{decorations.pathreplacing}
\usepackage{mathtools,upgreek,paralist}
\usepackage[final]{showkeys} 

\definecolor{refkey}{gray}{.75}
\definecolor{labelkey}{gray}{.5}
\usepackage[algo2e,boxed,vlined,algoruled]{algorithm2e}
\usepackage{comment}
\usepackage[shortlabels]{enumitem}

\usepackage[letterpaper]{geometry}
\usepackage[colorinlistoftodos]{todonotes}
\presetkeys{todonotes}{inline, color=green}{}

\usepackage{accents}
\usepackage[algo2e,boxed,vlined,algoruled]{algorithm2e}

\usepackage[small]{caption}
\usepackage[colorlinks=true]{hyperref}
\colorlet{DarkGreen}{green!50!black}
\colorlet{DarkGray}{gray!60!black}

\numberwithin{equation}{section}

\renewcommand{\restriction}{\mathord{\upharpoonright}}
\renewcommand{\epsilon}{\varepsilon}
\newcommand{\given}{\;\big|\;}
\newcommand{\one}{\mathbf{1}}

\usepackage{color}
 \definecolor{refkey}{gray}{.5}
 \definecolor{labelkey}{gray}{.5}
\definecolor{light}{gray}{.9}

\usepackage{soul}

\usepackage[capitalise]{cleveref}

\newtheorem{theorem}{Theorem}[section]
\newtheorem*{theorem*}{Theorem}
\newtheorem{lemma}[theorem]{Lemma}

\newtheorem{claim}[theorem]{Claim}
\crefname{claim}{Claim}{Claims}
\newtheorem{proposition}[theorem]{Proposition}

\newtheorem{corollary}[theorem]{Corollary}

\theoremstyle{definition}{

\newtheorem{definition}[theorem]{Definition}

\newtheorem*{definition*}{Definition}

\newtheorem{remark}[theorem]{Remark}
\newtheorem*{remark*}{Remark}

}

\newcommand{\E}{\mathbb E}

\newcommand{\R}{\mathbb R}
\newcommand{\Z}{\mathbb Z}

\newcommand{\cA}{\ensuremath{\mathcal A}}
\newcommand{\cB}{\ensuremath{\mathcal B}}
\newcommand{\cC}{\ensuremath{\mathcal C}}

\newcommand{\cE}{\ensuremath{\mathcal E}}
\newcommand{\cF}{\ensuremath{\mathcal F}}

\newcommand{\cI}{\ensuremath{\mathcal I}}
\newcommand{\cJ}{\ensuremath{\mathcal J}}

\newcommand{\cL}{\ensuremath{\mathcal L}}

\newcommand{\cP}{\ensuremath{\mathcal P}}

\newcommand{\cU}{\ensuremath{\mathcal U}}
\newcommand{\cV}{\ensuremath{\mathcal V}}

\newcommand{\sfh}{\ensuremath{\mathsf{h}}}

\newcommand{\fm}{\mathfrak{m}}

\newcommand{\fW}{\mathfrak{W}}

\newcommand{\fD}{\mathfrak{D}}

\newcommand{\g}{{\ensuremath{\mathbf g}}}
\newcommand{\f}{{\ensuremath{\mathbf f}}}

\newcommand{\bP}{{\ensuremath{\mathbf P}}}

 \renewcommand{\epsilon}{\varepsilon}

\DeclareMathOperator{\diam}{diam}

\newcommand{\Iso}{{\mathsf{Iso}}}

\newcommand{\Itop}{\cI_{\mathsf{top}}}
\newcommand{\Ibot}{\cI_{\mathsf{bot}}}
\newcommand{\Vtop}{\cV_\mathsf{top}}
\newcommand{\Vbot}{\cV_\mathsf{bot}}

\newcommand{\Atop}{\widehat{\cV}_{\mathsf{top}}}
\newcommand{\Abot}{\widehat{\cV}_{\mathsf{bot}}}

\newcommand{\Top}{\mathsf{top}}
\newcommand{\Bot}{\mathsf{bot}}
\newcommand{\Vred}{\cV_\mathsf{red}}
\newcommand{\Vblue}{\cV_\mathsf{blue}}
\newcommand{\Ared}{\widehat{\cV}_{\mathsf{red}}}

\newcommand{\Ablue}{\widehat{\cV}_{\mathsf{blue}}}

\newcommand{\Red}{\ensuremath{\mathsf{red}}}
\newcommand{\Blue}{\ensuremath{\mathsf{blue}}}
\newcommand{\noR}{\ensuremath{\mathsf{nonR}}}
\newcommand{\noB}{\ensuremath{\mathsf{nonB}}}
\newcommand{\Full}
{\ensuremath{\mathsf{full}}}
\newcommand{\Fuzzy}{\ensuremath{\mathsf{fuzzy}}}

\newcommand{\sep}{{\fD}}

\newcommand{\ez}{{\mathfrak{e}_3}}

\renewcommand{\d}{\mathrm{d}}
\newcommand{\dob}{\ensuremath{\mathsf{dob}}}
\newcommand{\fl}{\ensuremath{\mathsf{fl}}}

\newcommand{\musoft}{{\widehat\upmu}}

\DeclareMathOperator{\hgt}{ht}
\newcommand{\tv}{{\textsc{tv}}}

\newcommand{\hull}[1]{{\mathbullet{#1}}}

\makeatletter
\crefname{step}{Step}{Steps}
\crefname{case}{Case}{Cases}
\newcommand{\superimpose}[2]{%
  {\ooalign{$#1\@firstoftwo#2$\cr\hfil$#1\@secondoftwo#2$\hfil\cr}}}
  
\newcommand{\sbullet}{%
  \hbox{\fontfamily{lmr}\fontsize{.4\dimexpr(\f@size pt)}{0}\selectfont\textbullet}}
\DeclareRobustCommand{\mathbullet}{\accentset{\sbullet}}

\makeatother

\title[3D Ising and Potts interfaces above a hard floor]{Logarithmic delocalization of low temperature \\ 3D Ising and Potts interfaces above a hard floor}

\author{Joseph Chen}
\address{J.\ Chen\hfill\break
Courant Institute\\ New York University\\
251 Mercer Street\\ New York, NY 10012, USA.}
\email{jlc871@courant.nyu.edu}

\author{Reza Gheissari}
\address{R.\ Gheissari\hfill\break
Department of Mathematics \\ Northwestern University }
\email{gheissari@northwestern.edu}

\author{Eyal Lubetzky}
\address{E.\ Lubetzky\hfill\break
Courant Institute\\ New York University\\
251 Mercer Street\\ New York, NY 10012, USA.}
\email{eyal@courant.nyu.edu}

\begin{document}

\begin{abstract}
    We study the entropic repulsion of the low temperature 3D Ising and Potts interface in an $n\times n \times n$ box with blue boundary conditions on its bottom face (the hard floor), and red boundary conditions on its other five faces. For Ising, 
    Fr\"ohlich and Pfister proved in 1987 that the typical interface height above the origin diverges (non-quantitatively), via correlation inequalities special to the Ising model; no such result was known for Potts.
    We show for both the Ising and Potts models that the entropic repulsion fully overcomes the potentially attractive interaction with the floor, and obtain a logarithmically diverging lower bound on the typical interface height. 
    This is complemented by a conjecturally sharp upper bound of $\lfloor \xi^{-1}\log n\rfloor$ where $\xi$ is the rate function for a point-to-plane non-red connection under the infinite volume red measure.    
    The proof goes through a coupled random-cluster interface to overcome the potentially attractive interaction with the boundary, and a coupled fuzzy Potts model to reduce the upper bound to a simpler setting where the repulsion is attained by conditioning a no-floor interface to lie in the upper half-space. 
\end{abstract}

\maketitle

\vspace{-.8cm}

\section{Introduction}\label{sec:intro}
The Potts model on a finite graph $G=(V,E)$
is the following probability distribution over $\sigma \in \{1,\ldots,q\}^V$ (coloring each site in $V$ by one of $q$ colors),
\begin{align}\label{eq:Potts-measure}
    \mu_{G,\beta,q}(\sigma) \propto \exp \Big( - \beta \sum_{(u,v)\in E} \one_{\{\sigma_v\ne \sigma_w\}}\Big)\,,
\end{align}
with $\beta>0$ the inverse-temperature (in this paper, a large constant). The case $q=2$ is the Ising model. 
A boundary condition for the model is an assignment $\eta$ of specific values to a subset of the sites $S\subset V$, whence one considers \cref{eq:Potts-measure} restricted to configurations agreeing with $\eta$ on $S$ (equivalently, one looks at $\mu_{G,\beta,q}$ conditional on $\sigma$ agreeing with $\eta$ on $S$).
We will be interested in the $q$-state Potts model on the $n\times n \times n$ cube $\Lambda_{n}$ with boundary conditions, denoted by $\fl$, that are $\Blue$ boundary conditions on its bottom face (the \emph{hard floor}), and $\Red$ boundary conditions on the other five faces. 
We denote this distribution by $\mu_{\Lambda_n}^\fl$.

These boundary conditions induce a $\Blue$-to-non-$\Blue$ interface, denoted here by $\cI_\Blue$, consisting of the set of plaquettes separating vertices in the $\Blue$ phase (either in the $\Blue$ component of the boundary, or in non-$\Blue$ bubbles encapsulated by that component), from the rest (see \cref{def:FK-Potts-interfaces} for the formal definition).  
(The plaquettes of the interface are unit squares bounded by vertices of~$\Z^3$, and two plaquettes are considered adjacent if they share an edge.)
Regard the height of each plaquette $f\in\cI_\Blue$ as the difference of the $\ez$-coordinate of its midpoint and the $\ez$-coordinate of the floor, which we view as having height zero, and denote it by $\hgt(f)$. 
In this paper, we study the typical height of plaquettes in the bulk of the interface $\cI_\Blue$, as it balances competing forces of rigidity (which would hold at large $\beta$ if instead the boundary conditions changed color in the middle of the box), entropic repulsion away from the hard floor, and a complicated interaction between the interface and the floor through the potential finite bubbles of non-$\Blue$ below the interface that the hard floor precludes. 

The analysis of low-temperature 3D Ising interfaces has a long history, dating fifty years to the classical work of Dobrushin~\cite{Dobrushin73} who established rigidity of the interface when there is no nearby floor; namely, one imposes on the $n\times n\times n$ cube $\Lambda_n$
\emph{Dobrushin boundary conditions}, denoted here by $\dob$, that are $\Blue$ on its lower half and $\Red$ on its top half.
This spurned several other works on rigidity (localization) of integer-valued surfaces at low-temperatures, including the order-order interface of the super-critical random-cluster model and interfaces of low-temperature Potts model~\cite{GielisGrimmett02} under $\mu_{\Lambda_n}^\dob$. The behavior of localized interfaces in the presence of a nearby floor (sometimes called a wall) introduces a competition between the forces preferring rigidity and \emph{entropic repulsion}, a desire to delocalize from the floor to make room for thermal fluctuations that would otherwise be precluded. The work of Fr\"ohlich and Pfister~\cite{FP87a,FP87b} studied this and related surface transitions like wetting and layering rigorously for the Ising model. Using correlation inequalities special to the Ising model, they established non-quantitative delocalization of the Ising interface in $\mu_{\Lambda_n}^\fl$; i.e., the statement that the typical height of the interface above the origin diverges with $n$ (at an unknown rate). To our knowledge, this remains the only rigorous bound on the Ising interface height above a hard floor (see \cref{sec:related-work} for logarithmic bounds in simpler settings), and not even a non-quantitative delocalization was known for the Potts interface.

Rigorous quantitative results on entropic repulsion have generally been restricted to models only on interfaces (as opposed to entire spin configurations) and that are height functions (no overhangs). The model closest to the Ising/Potts interfaces is the---well-studied in its own right---Solid-On-Solid (SOS) model. 
There, the picture of entropic repulsion at low-temperatures is now precisely understood; the SOS interface in $\Lambda_n$ with $0$ boundary conditions and nonnegative heights is such that all but $\epsilon_\beta n^2$ many plaquettes of the interface are at height $\lfloor \frac{1}{4\beta} \log n\rfloor+O(1)$~\cite{CLMST14} with more precise quantities like the shape and fluctuations of the level curve for this highest height themselves being the subject of much rich investigation~\cite{CLMST16,Caddeo-SOS-level-lines}. Entropic repulsion, and the ensuing shape of the interface above a hard floor, have also been studied extensively for related height function models like~\cite{LMS-Harmonic-pinnacles}; see the survey~\cite{IV18} for more rich interface phenomena in SOS and related height function models. One element of all such analyses is a good understanding of the exponential rates for typical and maximal height deviations of the rigid surface before placing it near a floor.

In recent years, the precise law of height deviations, and maximum height, in the 3D Ising interface itself were derived in~\cite{GL_max,GL_tightness}, and for the random-cluster and Potts interfaces in~\cite{ChLu24}. In the Ising case, using this, a phenomenon of delocalization due to entropic repulsion was sharply studied when the repulsion is generated by what we will refer to as a \emph{soft floor}, where the model has no boundary condition floor, but the interface is conditioned to stay above height zero~\cite{GL-entropic-repulsion}. This may feel qualitatively the same, and in the case of the height function models like SOS it is indeed the same. However, the boundary conditions at height zero are an $\exp(c n^2)$ tilt of the measure without them, and the interactions of the interface to the boundary conditions on the floor are non-explicit and quite complicated. Moreover, since the interface is uniformly within height $O(\log n)$ of the floor, these interactions are on the same order as the entropy-induced repulsion effects and compete. In particular, these interactions could pin the surface to the floor and prevent it rising to logarithmic heights. 
We describe this and how we overcome it in more detail in Section~\ref{sec:related-work}. 

Our main results concern the low-temperature Ising and Potts interfaces in $\mu_{\Lambda_n}^\fl$, i.e., with a ``hard floor" formed by boundary conditions at height zero. We show a logarithmically growing lower bound on the typical height of $\cI$, together with a conjecturally sharp upper bound. Moreover, as $\beta \to\infty$, the ratio of the lower and upper bounds goes to $1$. Our analyses require moving back and forth between Ising and Potts, a coupled random-cluster model, and the fuzzy Potts model~\cite{Maes-fuzzy-Potts}, the latter two giving certain advantages in dealing with the possible (attractive and repulsive, respectively) interactions of the interface with the hard floor at height zero. 

\subsection{Main result} 
Under the infinite-volume measure $\mu_{\mathbb Z^3}^\Red$, at large enough $\beta$, there is almost surely an infinite connected component of $\Red$ vertices, call it $\Vred$, and all connected components of $\Vred^c= \mathbb Z^3 \setminus \Vred$ are finite. See, e.g.,~\cite[Chapter 7]{Grimmett_RC} for more on the $\beta$ large regime of the Potts model in $d\ge 3$.  
Define the finite point-to-plane connectivity rate by  
\begin{align}\label{eq:point-to-plane-connectivity}
    \xi_{h} = \xi_{h}(q,\beta) := - \log \mu_{\mathbb Z^3}^{\Red}\big((0,0,0) \longleftrightarrow \R^2\times\{h\} \text{ in $\Vred^c$} \big)\,.
\end{align}
As we will see in Section~\ref{sec:identification-of-rates}, a sub-additivity argument using Fekete's Lemma 
can show that \begin{align}\label{eq:xi}
\lim_{h\to\infty} \frac{\xi_{h}}{h} = :\xi = \xi(q,\beta)
\end{align}
exists (in fact, we show  in Lemma~\ref{lem:xi-tilde-additive} that $\xi_{h}$ is approximately additive, so that $\xi_{h} = \xi h +O(1)$ for all $h$). A simple Peierls argument (and the fact that such a non-red connection implies a finite random-cluster component of size $h$)---see, e.g.,~\cite[Thm~7.3.2]{Grimmett_RC}---implies that $\xi \in [4\beta -C,4\beta +C]$ for all $h\ge 1$, for an absolute constant $C>0$.  With the rate quantity $\xi$ in hand, we define 
\begin{align}\label{eq:h_n^*}
    h_{n}^* = 
   h_n^*(q,\beta): = \lfloor \xi^{-1} \log n  \rfloor\,. 
\end{align}
Our main result is the following.

\begin{theorem}\label{thm:main}
    Fix $q\ge 2$ and $\beta$ large enough. Consider $\mu_{\Lambda_n}^\fl$, the $q$-state Potts model on the $n\times n \times n$ box $\Lambda_n$ with boundary conditions $\Blue$ on its bottom face and $\Red$ elsewhere, and let $\cI_\Blue$ be the interface separating the $\Blue$ and non-$\Blue$ phases as defined above (formally defined in Definition~\ref{def:FK-Potts-interfaces}).  There exists a sequence $\epsilon_\beta\downarrow 0$ as $\beta \uparrow\infty$ such that with probability $1-o(1)$, 
    \begin{align*}
        |\{f\in\cI_\Blue: \hgt(f) \notin  [(1-\epsilon_\beta) h_{n}^* , h_{n}^*]\} | \le \epsilon_\beta n^2\,. 
    \end{align*}
In addition, the same bound holds for the interface $\cI_\Red$ separating the $\Red$ and non-$\Red$ phases.
\end{theorem}

As mentioned, in terms of lower bounds on the interface height with a hard floor, to the best of our knowledge in the $q=2$ case there was only the non-quantitative bound of~\cite{FP87a}, and in Potts for $q\ge 3$, there was no proof of delocalization due to the entropic repulsion, even non-quantitatively. 

The upper bound of Theorem~\ref{thm:main} is conjecturally sharp.  
Indeed, for general $q\ge 2$, the height $h_n^*$ is exactly half of the maximum height oscillation $M_n = \max\{|\hgt(f)|: f\in \cI_{\Blue}\}$ under the no-floor Dobrushin boundary condition measure $\mu_{\Lambda_n}^\dob$ as identified in~\cite{GL_tightness,ChLu24}. This matches the scaling relation between the repelled typical height, and no-floor maximum height in the SOS approximation~\cite{CLMST14}. To be precise,~\cite[Theorem~1.3]{ChLu24} identified $M_n$ less explicitly in terms of a different rate, which they called the pillar rate. Part of our work in this paper (Theorem~\ref{thm:identifying-rate-with-infinite-limit}) identifies that pillar rate with~$\xi_h$. This implies that under $\mu_{\Lambda_n}^{\dob}$, one has $M_n = 2 h_n^* + O(1)$ with probability $1-o(1)$.

\begin{remark}\label{rem:other-interfaces}
Our proofs go through an analysis of the random-cluster model, in order to handle the interaction with the hard floor---this is even needed in the special case $q=2$ (the Ising model) where the model has FKG. Specifically, the interaction between the random-cluster interface $\Itop$ (the boundary of the connected random-cluster component of the $\Red$ boundary sites; see \cref{def:FK-Potts-interfaces}) and the hard floor is amenable to rigorous analysis in a way Ising and Potts interactions are not.

As a consequence of \cref{lem:Ifull-wall-faces}, with high probability the typical height of $\Itop$ (and the analogously defined random-cluster interface $\Ibot$), and the typical height of the Potts interface $\cI_{\Blue}$ are the same. 
As a byproduct, \cref{thm:main} also holds (with the same $h_n^*$) for the interfaces $\Itop,\Ibot$.
At the same time, the intuition behind the entropic repulsion phenomenon would suggest that the interface $\Ibot$ would rise to height $\xi_\Bot^{-1} \log n$, where $\xi_\Bot$ is the relevant random-cluster connection rate.  
Perhaps surprisingly, the upper bound of $h_n^*$ proves that the typical height of $\Ibot$ is $\Omega(\log n)$ smaller than this prediction (as $\xi_\Bot^{-1}$ is known to be strictly smaller than $\xi$ by~\cite{ChLu24}).  
\end{remark}

\begin{remark}
The rate $\xi$ governing \cref{eq:h_n^*} was the large deviation of a point-to-plane connection probability; for the Ising model, it can be shown (albeit for a slightly different notion of adjacency) that it is equivalent to a point-to-point rate, e.g., using the CLT in \cite[Cor.~3]{GL_max}. We expect the same to hold for the notion of adjacency used here, as well as for the Potts model for any $q\geq 3$.
\end{remark}

\begin{remark}
    Our lower bound of $(1-\epsilon_\beta)h_n^*$ in \cref{thm:main} is of the form $(1-C/\beta)h_n^*$.
    See \cref{rem:lower-bound-rate} for a discussion on a sketch of how this could be improved to the point-to-plane random-cluster connection rate in the complement of the (a.s.\ unique) infinite component. This would asymptotically be $(1-(C/\beta)e^{-\beta})h_n^*$, but that would be the limit of our approach.
\end{remark}
\subsection{Further prior work and proof ideas}\label{sec:related-work}    
Let us begin with a history of quantitative lower bounds for entropic repulsion of Ising-style interfaces that simplify some aspect to not have the delicate potentially attractive interaction to the boundary of the hard-floor case.
In~1991, 
Melicher{\v{c}}{\'\i}k~\cite{melichervcik1991entropic} studied a distribution over Ising-type interfaces (allowing overhangs), but with the distribution only being over the interface, not having any finite bubbles above and below the interface, and established that the interface rises to height order $\log n$ when conditioned to be non-negative. The presence of the full spin configurations (including the possibility of finite bubbles below the interface) causes a well-known difficulty for studying Ising interfaces in the presence of nearby boundary conditions due to complicated, possibly attractive, interactions between the interface and the boundary floor. (In $d=2$ where the phenomenology is very different, the works~\cite{IST15,IOVW} required significant ideas to deal with the potential for these interactions to attract the interface to the boundary floor. The latter of these works showed first that the interface is typically at least distance $n^{\epsilon}$ away from the floor, whence the interactions become negligible. However, in the $d =3$ case where the repulsion is only of order $\log n$, there is no hope for such an approach to controlling the interactions.) 
 Holick\'y and Zahradn\'{i}k~\cite{HZ93} sketched in 1993 how some of the steps in~\cite{melichervcik1991entropic} can be extended to incorporate the interactions in the true 3D Ising measure, postponing the details to a future version (but to our knowledge there was still no rigorous proof of this even in this Ising case).
 The recent work~\cite{GL-entropic-repulsion} proved an entropic repulsion transition for the interface in the true 3D Ising measure (including the bubbles) at height $h_n^*$ when the floor is soft, i.e., where the interface is conditioned to be above height zero. Effectively, this incorporated self-interactions of the interface through the bubbles, but still did not have to handle interactions with a nearby hard floor.

Dealing with the potentially attractive interactions between the interface and the hard floor is a significant part of establishing our lower bound of $(1-\epsilon_\beta) h_n^*$ and is achieved by instead working with a certain ``top" subset $\Itop$ of a coupled random-cluster model's order-order interface (see \cref{def:FK-Potts-interfaces}). For this specific marginal of the random-cluster interface, the interaction with the boundary is approximately in the right direction, in that it is less costly to apply maps that lift the interface in the presence of the hard floor, than in its soft-floor analogue. The reason this lower bound does not match $h_n^*$ is because the rates for downwards oscillations of $\Itop$ are the source of our lower bound, but are distinct from those of $\cI_{\Blue}$ that we expect to govern the repulsion in $\mu_{\Lambda_n}^\fl$. This lower bound (and a logarithmic order upper bound) are proven in Section~\ref{sec:logarithmic-lower-upper}. Notably, these proofs of logarithmic order delocalization are essentially self-contained and do  not use any of the more combinatorial walls, ceilings, and pillars machinery core to recent work in the subject.

The proof of the conjecturally sharp upper bound of $h_n^*$ goes via a different comparison to the soft-floor model, directly in the spin space. This comparison shows that the complicated interactions through finite bubbles between the interface $\cI_{\Blue}$ and the hard floor do not induce any additional repulsion beyond that visible in the soft-floor model. In the Ising case, this follows from a sequence of delicate 
monotonicity relations between the soft-floor and hard-floor interfaces. In the Potts case, we first build on~\cite{ChLu24} to get sharp large deviation bounds on maximal oscillations of a rigid Potts interface, and get the $h_n^*$ upper bound in the soft-floor model. Notably, we then replace the Ising model's monotonicity step with a monotonicity step in the coupled \emph{fuzzy Potts model} where all non-$\Blue$ vertices are relabeled as $\Fuzzy$. Since the interface $\cI_{\Blue}$ is measurable with respect to the coupled fuzzy Potts model and monotone in the set of $\Blue$ vertices, this allows for an application of the FKG inequalities developed for the fuzzy Potts model in~\cite{Chayes-fuzzy-Potts}. The comparison between the soft and hard floor measures is in \cref{subsec:log-upper-bound} and the proof of the sharp upper bound is in \cref{sec:sharp-upper}. 

The above arguments would establish Theorem~\ref{thm:main} with $h_n^*$ defined less explicitly in terms of a rate $\alpha_h$~\eqref{eq:alpha-h} for a certain \emph{pillar} of the Dobrushin interface to exceed some height, in place of $\xi_h$. 
\cref{sec:identification-of-rates} establishes equivalence between $\alpha_h$ and the finite point-to-plane connectivity rate of $\xi_h$, and furthermore shows that the latter is additive so that $h_n^*$ has the nice characterization of~\eqref{eq:h_n^*} with the limiting $\xi$. This goes via the random-cluster representation of the interface and certain swapping operations to show the equivalences of rates on the appropriate high probability events for the connections. As mentioned, this equivalence has the added benefit of refining the previous tightness results of the no-floor maximum oscillation $M_n$ of $\cI_{\Blue}$ from~\cite{ChLu24} to become tightness around the explicit value $2\lfloor \xi^{-1} \log n\rfloor = 2h_n^*$.

\section{Logarithmic upper and lower bounds}\label{sec:logarithmic-lower-upper}

The main goal of this section is to establish the lower bound of $(1-\epsilon_\beta)h_n^*$ for the typical height of the plaquettes of $\cI_\Blue$. As mentioned in Section~\ref{sec:related-work}, we approach this via a coupled random-cluster model where the bottom boundary face of $\Lambda_n$ is conditioned not to be connected to the five other boundary faces, inducing an interface $\cI_{\Full}$ of plaquettes dual to closed edges separating the two components. (By a simple energy-entropy tradeoff, when $\beta$ is large, the number of plaquettes in interfaces will not exceed $(1+\epsilon_\beta) n^2$, and so most of the plaquettes of $\cI_{\Blue}$ coincide with the plaquettes of $\cI_{\Top}$ (which bounds the open connected component of the five $\Red$ boundary faces). Thus, lower and upper bounds on the typical height of one of these interfaces also bound the other.) 

Since this section contains many of the new ideas in the paper to control the interactions of interfaces with the hard floor in the two possible directions (attractive and repulsive), we sketch the logic of the argument below.
\begin{enumerate}
    \item \emph{Control of attractive interaction with hard floor via $\cI_{\Top}$.} When seeking a lower bound on the interface height, we are concerned about potentially attractive interactions between the interface and the hard floor. In the random-cluster world, these are expressed in terms of differences in probabilities, under a soft-floor interface, of the bottom face boundary edges all being open (to force a hard floor): see \cref{lem:compare-fl-dob-integral}. If we \emph{only} reveal the $\cI_{\Top}$ portion of the random-cluster interface, it induces free boundary conditions on the vertices immediately below it. Pushing these free boundary conditions up, as would happen under a map that lifts $\cI_{\Top}$, heuristically moves the free part of the boundary away from height zero, increasing the marginals at height zero, and hence interactions with the boundary only favor the higher realization of $\cI_{\Top}$, at which point the new entropy we can inject via downwards $\cI_{\Top}$ spikes beats the energy cost paid for lifting the interface up to $(1-\epsilon) h_n^*$.   
    \item \emph{Necessity for a crude upper bound.} The above argument, however, has problems due to the potential overhangs in $\cI_{\Top}$ resulting in its vertical shift not being entirely above it. Our solution is to lift $\cI_{\Top}$ by an exaggerated amount, more than its maximum height oscillation, and as long as this is $n^{o(1)}$, the extra cost this induces only changes the $\epsilon_\beta$ in the $(1-\epsilon_\beta) h_n^*$ we can ensure the interface reaches up to. 
    \item \emph{Control on repulsive interactions with hard floor via fuzzy Potts.} Step (2) necessitates some non-trivial upper bound on the maximum height of the interface in the hard-floor setting. In \cref{prop:Potts-monotonicity}, we use a delicate revealing and FKG argument with the fuzzy Potts representation of the model to show a monotonicity in the opposite direction of step (1) for the $\cI_{\Blue}$ interface: the interactions of $\cI_{\Blue}$ with the hard floor are only more attractive than in the soft-floor case---actually with boundary conditions flipping color at height $\sfh=C \log n$ instead of $0$. In this latter setting, an $O(\log n)$ upper bound is easy to deliver. This same monotonicity is used in Section~\ref{sec:sharp-upper} to give the sharp $h_n^*$ upper bound on the height of $\cI_{\Blue}$.   
\end{enumerate}
These results are presented in a different order in the section than the above logic, starting with preliminaries introducing the coupled random-cluster, Potts, and fuzzy Potts interfaces in Section~\ref{subsec:random-cluster-fuzzy-Potts}, then the monotonicity relation of item (3) is established in Section~\ref{subsec:log-upper-bound}, and finally items (1)--(2) to give the logarithmic lower bound using maps on $\cI_{\Top}$ are developed in Sections~\ref{subsec:Itop-monotonicity}--\ref{subsec:log-lower-bound}.

\subsection{Random-cluster, FK--Potts and fuzzy Potts models}\label{subsec:random-cluster-fuzzy-Potts}
Let $G=(V,E)$ be a finite graph. The \emph{random-cluster} model on $G$, also known as the \emph{Fortuin--Kasteleyn} (FK) model, is a probability distribution over subsets $\omega$ of the edge set $E$, with parameters $0<p<1$ and $q>0$, given by
\begin{equation}\label{eq:fk-measure}\pi_{G}(\omega) \propto \Big(\frac{p}{1-p}\Big)^{|\omega|} q^{\kappa(\omega)}\,,\end{equation}
where $\kappa(\omega)$ is the number of connected components of the graph $(V, \omega)$. One views $\omega$ as a Boolean function on $E$, assigning $e\in E$ the value $\omega_e=1$ if it is present (``\emph{open}'') and $\omega_e=0$ if it is missing (``\emph{closed}''); we refer to the $\kappa(\omega)$ connected components of the graph of open edges as \emph{open clusters}. Every edge is dual to the plaquette it is normal to and bisected by, and we say that plaquette is dual-to-open or dual-to-closed if said edge is open or closed, respectively. 

Consider the $q$-state Potts model on a finite graph $G$ at inverse-temperature $\beta$, and the FK model on the same graph with $p = 1 - e^{-\beta}$ and the same (integer) $q$. These two models can be coupled via the FK--Potts coupling, given by the following joint measure on edge subsets $\omega\subset  E$ and vertex colorings $\sigma:V\to\{1,\ldots,q\}$:
\begin{equation*}
    \bP_G(\omega,\sigma) \propto \Big(\frac{p}{1-p}\Big)^{|\omega|}
    \prod_{(u,v)\in \omega}\one_{\{\sigma_u = \sigma_v\}}\,.
\end{equation*}
We will refer to the joint distribution $\bP_G$ as the FK--Potts law.
In order to sample $(\omega,\sigma)\sim\bP_G$ conditional on a given edge configuration $\omega\sim\pi_G$, one  independently and uniformly at random draws a random color $\sigma_{\cC}$ for each open cluster $\cC$ of $\omega$, and assigns every $v\in \cC$ the color $\sigma_{\cC}$. For the converse, to sample $(\omega,\sigma)\sim\bP_G$ conditional on a given coloring $\sigma\sim\mu_G$, one closes all edges, then independently samples every monochromatic edge as a Bernoulli($p$) random variable. We will often use $\Blue$ and $\Red$ for the colors $1$ and $2$, respectively, as well as $\noB$ (resp., $\noR$) to denote all colors other than $\Blue$ (resp., $\Red$).

The $q$-state \emph{fuzzy Potts} model\footnote{The definition here is the one from \cite{Chayes-fuzzy-Potts}, and the special case of the $(1,q-1)$-state fuzzy Potts model from \cite{Haggstrom-Fuzzy-Potts}.}  is the pushforward measure 
\[ \phi_G = \mu_G \circ f_{\textsf{bf}}^{-1}\]
for $f_{\textsf{bf}}:\{1,\ldots,q\}^V\to\{\Blue,\Fuzzy\}^V$ that replaces every color other than $\Blue$ by the same $\Fuzzy$ color.
Equivalently, one can sample $\sigma\sim\phi_G$ by first sampling $\omega\sim\pi_G$,  then coloring the vertices of each open cluster of $\omega$ independently, in $\Blue$ with probability $1/q$ or in the color $\Fuzzy$ otherwise.
It was shown in~\cite{Chayes-fuzzy-Potts} (see also~\cite{Haggstrom-Fuzzy-Potts}) that this model satisfies the FKG lattice condition. In particular, for any arbitrary fixed boundary condition $\eta$ (taking values in $\Blue,\Fuzzy$), the fuzzy Potts model~$\phi_G^\eta$ (that is, $\phi_G$ conditional on the boundary vertices of $G$ agreeing with the assignment $\eta$) has for any pair of events $A,B$ that are increasing in the set of $\Blue$ vertices, that 
\begin{align}\label{eq:fuzzy-Potts-FKG}
    \phi_G^\eta(A \cap B) \ge \phi_G^\eta(A) \phi_G^\eta(B)\,.
\end{align}

In the context of \cref{thm:main}, it will be useful to view the $n\times n \times m$ box $\Lambda_{n,m}$ as a subgraph of $(\Z+\frac12)^3$, so that the Potts model assigns colors to the midpoints of unit cubes in $\Z^3$, whence the interface between, say, $\Blue$ and $\noB$ vertices, is a set of plaquettes whose corners lie in $\Z^3$:
\begin{align}\label{eq:floor-domain}
    \Lambda_{n,m} = \big( [-\tfrac{n}{2},\tfrac{n}{2}]^2 \times [0,m]\big) \cap (\Z+\tfrac12)^3\,. 
\end{align}

To impose a boundary condition $\eta$ on $\Lambda_{n,m}$, one adds to the graph its outer vertex boundary in $(\Z+\frac12)^3$ to get the outer vertex boundary set $\partial \Lambda_{n,m} = \{v\in (\Z + \frac12)^3 \setminus \Lambda_{n,m}: v\sim \Lambda_{n,m}\}$, and forces those vertices to take certain colors $\eta$, e.g., $\Red$ on boundary vertices in the upper half-space and $\Blue$ on the lower half-space to obtain~$\mu_{\Lambda_{n,m}}^\fl$. We write $\mu_n^\fl=\mu_{\Lambda_{n,n}}^\fl$ for brevity, denoting the analogous fuzzy Potts model by $\phi_n^\fl$. When a boundary condition $\eta$ is understood from context, e.g., when considering $\mu_n^\fl$, we use $\partial_\Red \Lambda_{n,m}$ (respectively, $\partial_{\Blue} \Lambda_{n,m}$) to denote the subsets of $\partial \Lambda_{n,m}$ getting color $\Red$ (resp., $\Blue$) under $\eta$. 

The FK model corresponding to $\mu_{\Lambda_{n,m}}^\fl$, under the FK--Potts coupling, would have us sample $\pi_{\Lambda_{n,m}}^{\mathsf{w},\mathsf{w}}$ with boundary conditions that are \emph{wired} on $\partial_\Red \Lambda_{n,m}$ (all vertices of $\partial_\Red \Lambda_{n,m}$ are viewed as being in the same connected component when counting $\kappa(\omega)$ in~\eqref{eq:fk-measure}) and wired on $\partial_{\Blue}\Lambda_{n,m}$, conditional on these not being in the same connected component of $\omega$, i.e.,  \begin{align}\label{eq:disconnection-event}
\sep_{n,m}:= \{\omega: \partial_{\Red} \Lambda_{n,m} \nleftrightarrow \partial_\Blue \Lambda_{n,m} \text{ in $\omega$}\}
\end{align} 
One then recovers $\sigma\sim\mu_{\Lambda_{n,m}}^\fl$ from $\omega\sim \pi^{\mathsf{w},\mathsf{w}}_{\Lambda_{n,m}}(\cdot\mid\sep_{n,m})$ via giving the cluster of $\partial_\Blue \Lambda_{n,m}$ the color $\Blue$, giving the cluster of $\partial_\Red \Lambda_{n,m}$ the color $\Red$, and then coloring every remaining open cluster of $\omega$ uniformly and independently (cf., e.g., \cite[\S2.2]{GheissariLubetzky18},\cite[Fact~3.4 and Cor.~3.5]{LubetzkySly12} where this coupling with disconnection events was previously used).
We will use the abbreviated notation $\bar \pi_n^\fl := \pi_{\Lambda_{n,n}}^{\mathsf{w},\mathsf{w}}(\cdot\mid\sep_{n,n})$, and $\bP_n^\fl$ for the joint FK--Potts distribution whose marginals are $\mu_n^\fl$ and $\bar\pi_n^\fl$.

We will also need to consider the Potts model on the enlarged $n\times n \times 2m$ box 
\begin{align}\label{eq:dob-domain}
    \Lambda'_{n,m} = \big( [-\tfrac{n}{2},\tfrac{n}{2}]^2 \times [-m,m]\big) \cap (\Z+\tfrac12)^3\,,
\end{align}
with boundary conditions that, for an integer $\sfh$, are $\Red$ at height at least $\sfh$ and $\Blue$ at height at most $\sfh$ (recall our sites are at half-integer heights). We will denote this distribution by $\mu_{\Lambda'_{n,m}}^\sfh$, and write $\mu_n^\sfh = \mu_{\Lambda'_{n,n}}^\sfh$ for brevity. Similarly, we denote by $\bar\pi_{\Lambda'_{n,m}}^\sfh$ the FK model on $\Lambda_{n,m}'$ with two wired boundary components, above and below height $\sfh$, conditioned on no path connecting them in~$\omega$, and put $\bar\pi_n^\sfh = \bar\pi_{\Lambda'_{n,n}}^\sfh$; as before, the FK--Potts measure with marginals $(\mu_n^\sfh,\bar\pi_n^\sfh)$ is denoted by $\bP_n^\sfh$. 
In the special case $\sfh=0$---the setup in Dobrushin's pioneering works---we may also write $\mu_n^\dob:=\mu_n^0$. In this context, $\partial_\Red \Lambda_{n,m}'$ and $\partial_\Blue \Lambda_{n,m}'$ refer to the portions of the outer vertex boundary of $\Lambda_{n,m}'$ above and below height $\sfh$ respectively. 

\begin{figure}
\vspace{-0.1in}
    \begin{tikzpicture}
   \node (fig1) at (0.5,0) {
    	\includegraphics[width=0.65\textwidth]{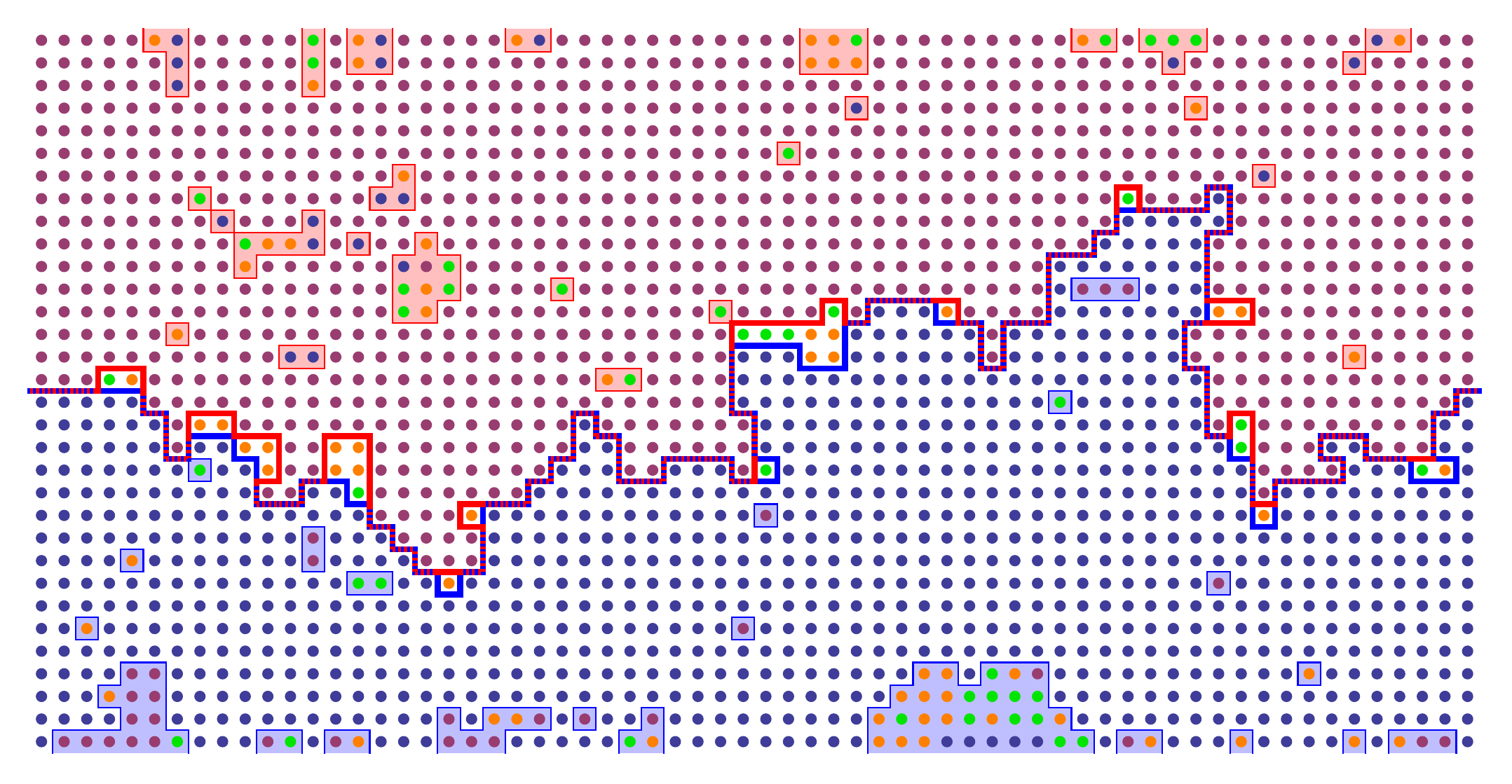}};
   \node (fig2) at (8.6,-1.35) {
   	\includegraphics[width=0.3\textwidth]{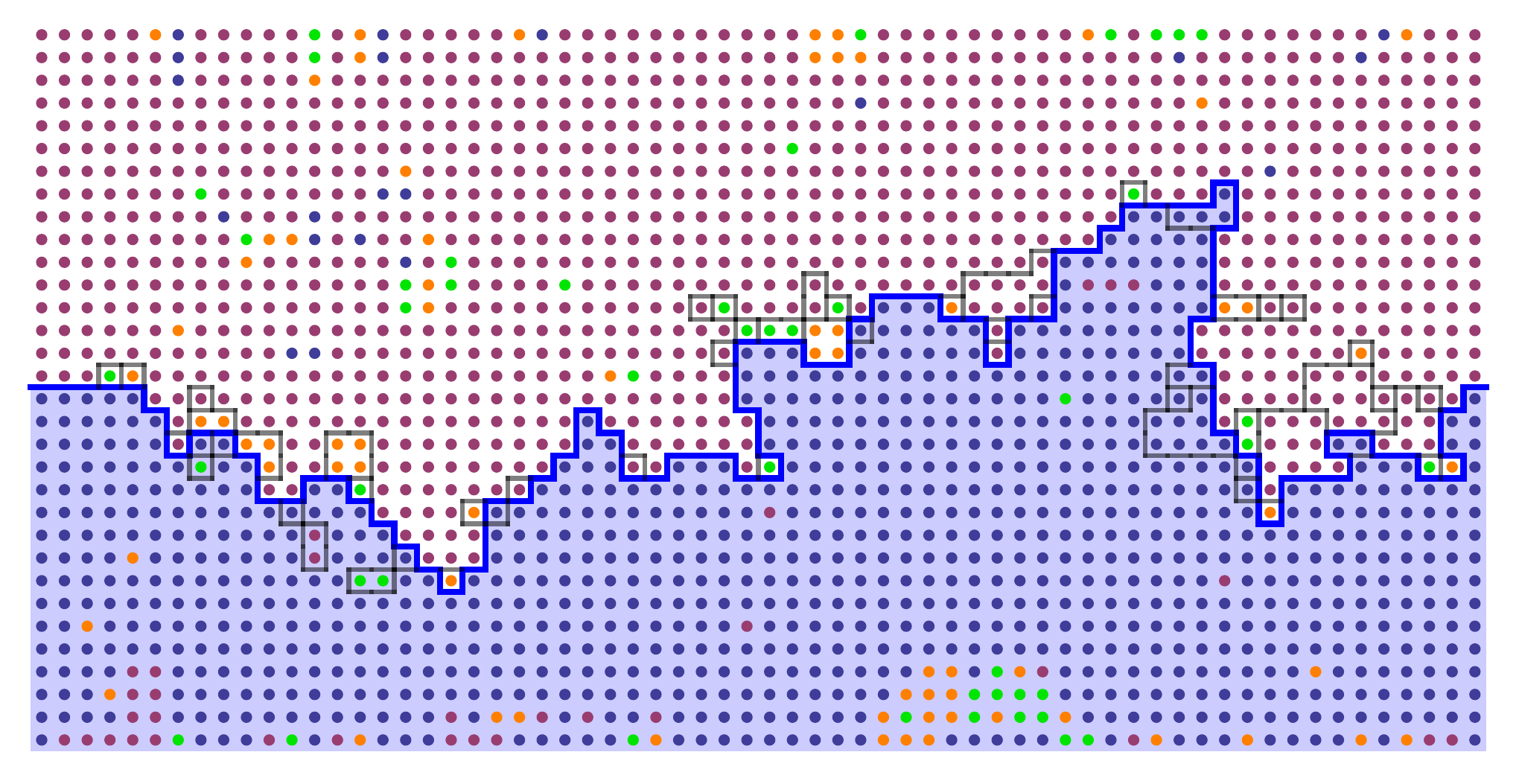}};
    \node (fig3) at (8.6,1.35) {
   	\includegraphics[width=0.3\textwidth]{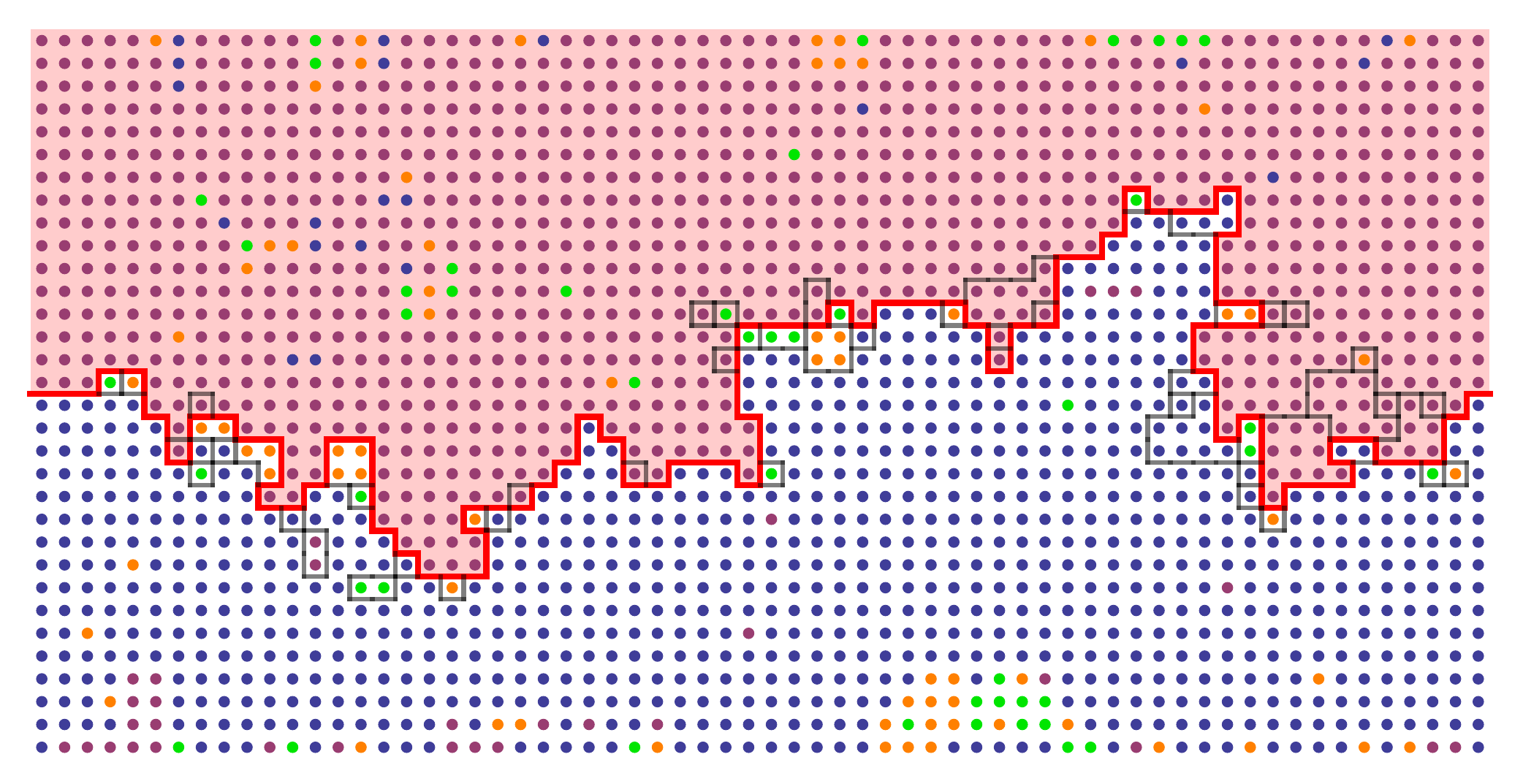}};
    \end{tikzpicture}
    \caption{A Potts configuration (left) and its two Potts interfaces $\cI_\Blue$ and $\cI_\Red$ (right).}
    \label{fig:inter-redblue}
\end{figure}

\begin{figure}
\vspace{-0.1in}
    \begin{tikzpicture}
   \node (fig1) at (0.5,0) {
    	\includegraphics[width=0.65\textwidth]{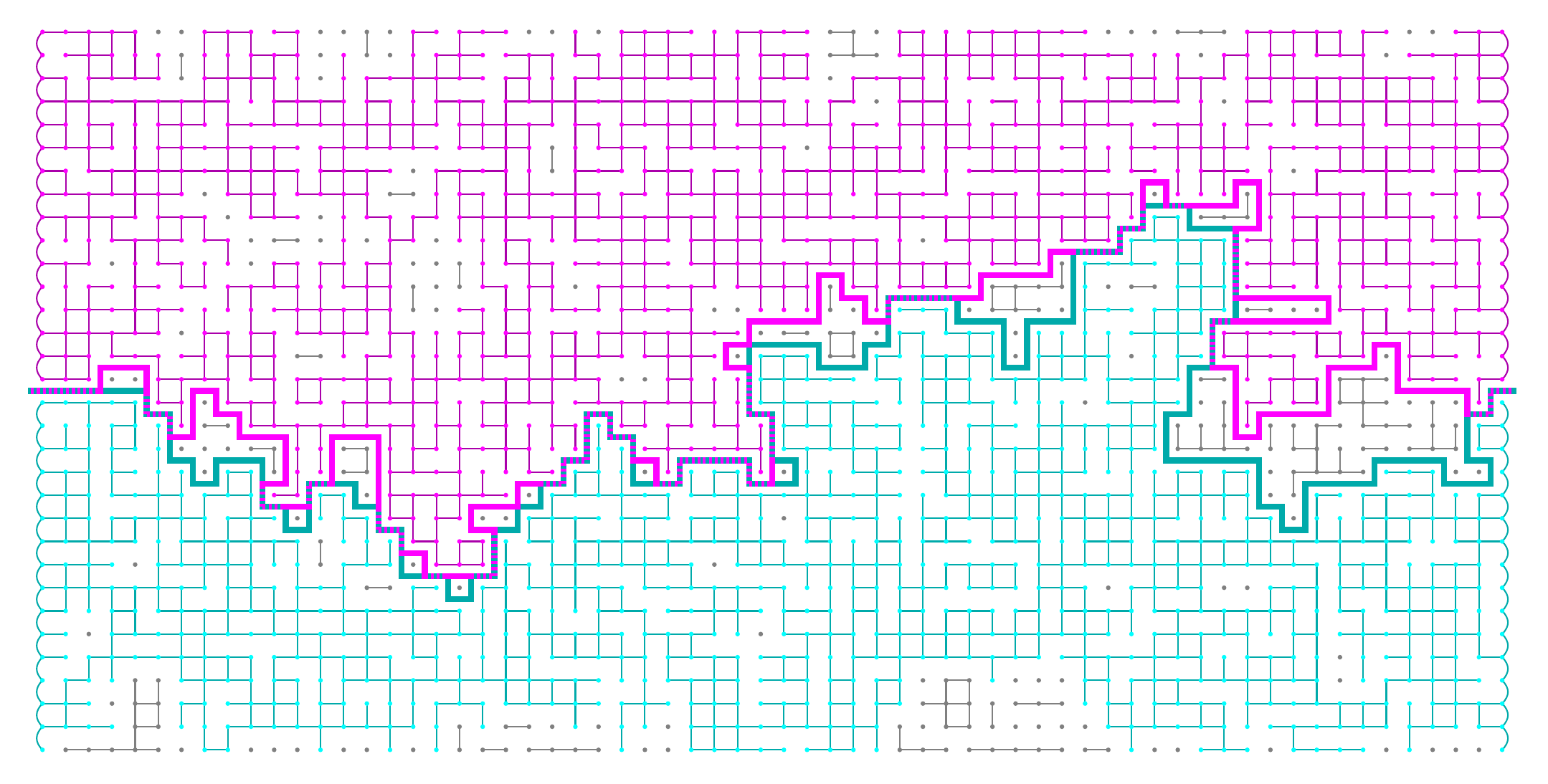}};
   \node (fig2) at (8.6,-1.35) {
   	\includegraphics[width=0.3\textwidth]{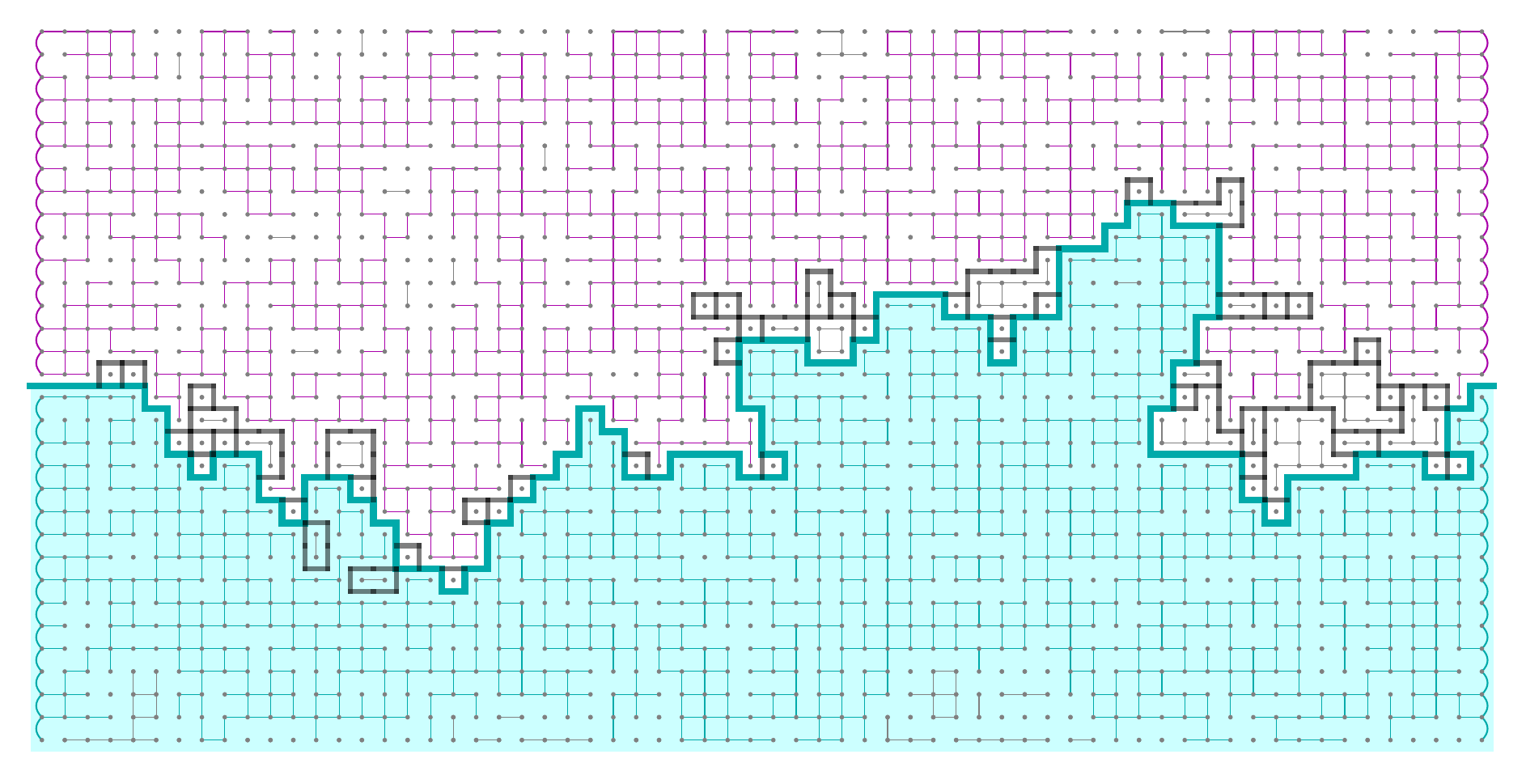}};
    \node (fig3) at (8.6,1.35) {
   	\includegraphics[width=0.3\textwidth]{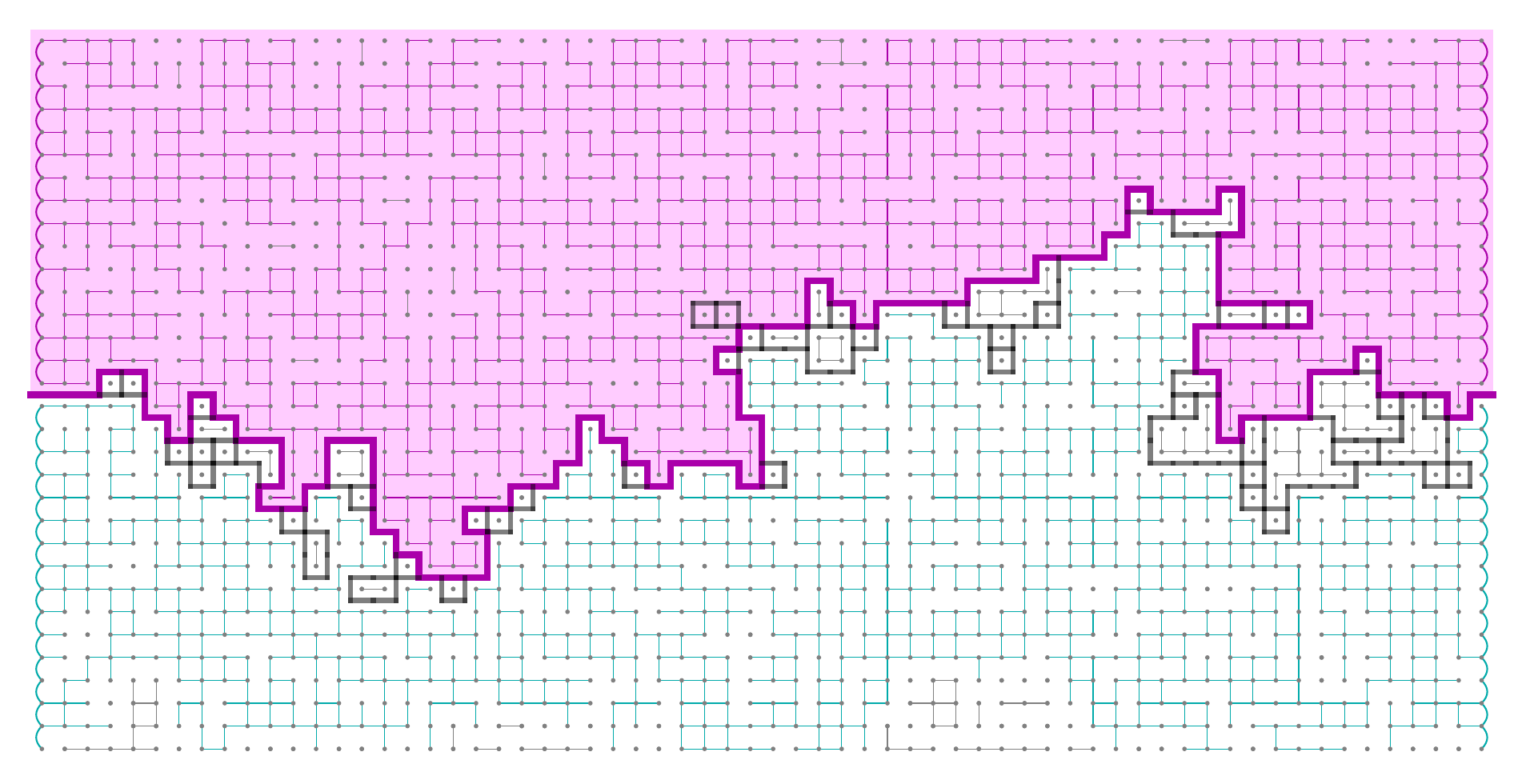}};
    \end{tikzpicture}
    \caption{A random-cluster configuration (left), and its interfaces $\cI_\Bot$ and $\cI_\Top$ (right).}
    \label{fig:inter-bottop}
\end{figure}

\begin{figure}
\vspace{-0.1in}
\includegraphics[width=0.75\textwidth]{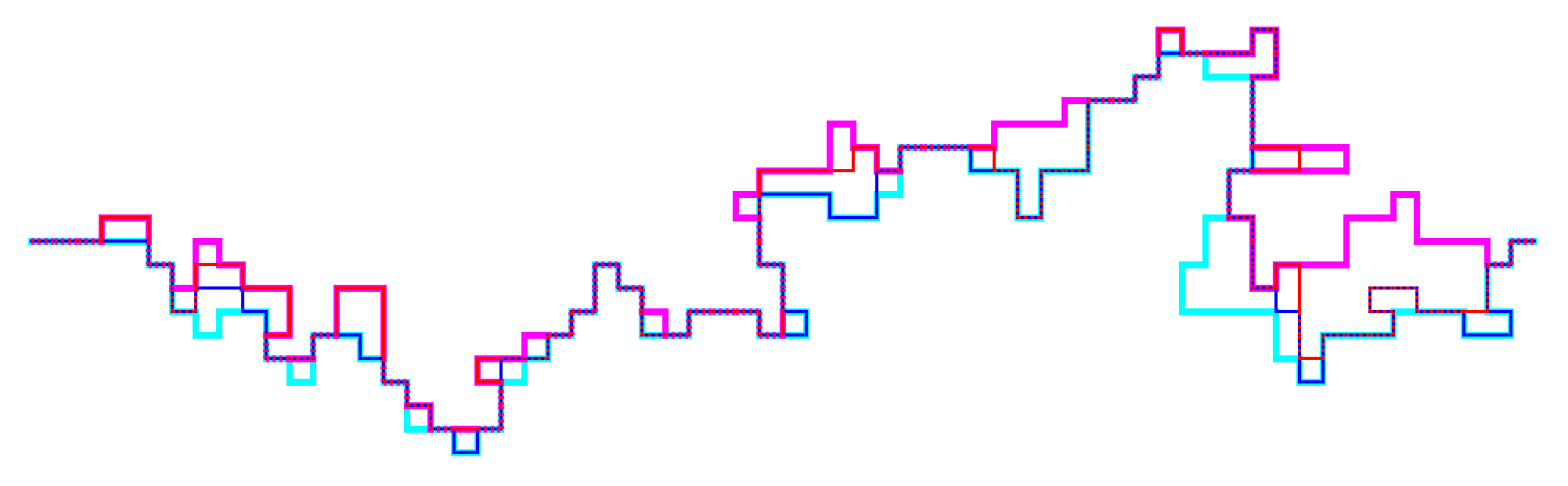}
    \caption{The four coupled interfaces, $\cI_{\Top}$ (pink highlighted), $\cI_{\Red}$ (red line), $\cI_{\Blue}$ (blue line) and $\cI_{\Bot}$ (cyan highlighted), showing the ordering $\cI_{\Top} \succeq \cI_{\Red}\succeq \cI_{\Blue} \succeq \cI_{\Bot}$.}
    \label{fig:four-inter}
\end{figure}
Under each of the boundary conditions specified above, we can define two Potts interfaces $\cI_\Blue$ and $\cI_\Red$, and two coupled random-cluster interfaces, $\Ibot$ and $\Itop$, defined as follows:
\begin{definition}[FK and Potts interfaces]\label{def:FK-Potts-interfaces}
Let $(\sigma,\omega)\sim\bP_n^\fl$ or $(\sigma,\omega)\sim\bP_n^\sfh$. The below definition applies to both, with the sets ``$\Blue$ boundary vertices", and ``$\Red$ boundary vertices" interpreted according to the setting. \begin{enumerate}
    \item\label{it:V-blue-red} Let $\Vblue$ be the connected component (via adjacency in $(\Z+\frac12)^3$) of $\Blue$ vertices in~$\sigma$ incident to the $\Blue$ boundary vertices; augment it to $\Ablue$ by adding all connected components of $(\Z+\frac12)^3 \setminus \Vblue$ whose outer vertex boundary is a subset of $\Vblue$. (This guarantees that all bubbles in the $\Blue$ phase are in $\Ablue$.)
    Define $\Vred$, $\Ared$ analogously. (See \cref{fig:inter-redblue}.)
    \item \label{it:V-top-bot} Let $\Vtop$ denote the vertices of the open cluster of $\omega$ incident to the $\Red$ boundary vertices; augment it to $\Atop$ by adding all connected components of $(\Z+\frac12)^3\setminus \Vtop$ whose outer vertex boundary is a subset of $\Vtop$. 
    Define $\Vbot$, $\Abot$ analogously. (See \cref{fig:inter-bottop}.)
\end{enumerate}
The interface $\cI_\Blue$ is the set of plaquettes dual to (closed) edges $(x,y)$ of $(\Z+\frac12)^3$ with $x\in\Ablue$ and $y\in\Ablue^c$. 
Analogously, define $\cI_\Red$, $\Ibot$ and $\Itop$ w.r.t.\ $\Ared$, $\Abot$ and $\Atop$, respectively.

By the FK--Potts coupling, under either set of boundary conditions $(\sigma,\omega)\sim\bP_n^\fl$ or $(\sigma,\omega)\sim\bP_n^\sfh$, one has $\Atop\subset  \Ared \subset  \Ablue^c$ and $\Abot \subset  \Ablue(\subset  \Ared^c)$. Then the four interfaces are ordered (see \cref{fig:four-inter}) via the  partial order on subsets of the vertices\footnote{Associate $\cI_\Blue$ and $\Ibot$ with $\Ablue$ and $\Abot$, resp., and associate $\cI_\Red$ and $\Itop$ with $\Ared^c$ and  $\Atop^c$, resp.\ (i.e., in all cases, use the part incident to the \emph{bottom} face of the boundary conditions) to arrive at this partial order.
}:
\begin{equation}\label{eq:interface-order} \Itop \succeq \cI_\Red \succeq \cI_\Blue \succeq \Ibot\,.\end{equation}
\end{definition}

The four interfaces $\Itop,\cI_\Red,\cI_\Blue,\Ibot$ are all subsets of a single \emph{full} interface, denoted $\cI_\Full$, which is the object that can be studied via cluster expansion, as done by Gielis and Grimmett~\cite{GielisGrimmett02}.
\begin{definition}[Full interface]
Let $\omega\sim \bar\pi_n^\fl$ or $\omega\sim\bar\pi_n^\sfh$. The below definition applies to both, with the sets ``$\Blue$ boundary vertices", and ``$\Red$ boundary vertices" interpreted according to the setting. Consider the set $F$  of plaquettes dual to closed edges of $\omega$. 
The full interface $\cI_\Full$ is the maximal $1$-connected component of $F$ (plaquettes $f,f'$ are $1$-adjacent if they share an edge) incident to the boundary plaquettes (plaquettes dual to edges between the $\Blue$ and  $\Red$ boundary vertices).
\label{def:full-interface}
\end{definition}

N.B., $\cI_{\Full}$ is amenable to a low-temperature cluster expansion analysis, because it induces fully wired boundary conditions on both the set of sites above it, and the set of sites below it. 

\begin{proposition}[{\cite{GielisGrimmett02,Grimmett_RC}}]
\label{prop:grimmett-cluster-exp}
Let $q\geq 1$ and $n,m\geq 1$. There exist $p_0<1$ so the following holds for all $p \geq p_0$.
Denoting by $|I|$ the number of plaquettes in $I$, and by $\kappa(I)$ the number of open clusters in the configuration where the only closed edges are those dual to plaquettes in $I$, we have
\begin{equation}\label{eq:CE}
    \bar{\pi}^\dob_{\Lambda'_{n,m}}(\cI_\Full=I) = \frac{1}{Z_{n,m,p,q}}p^{|\partial I|}(1-p)^{| I|}q^{\kappa(I)}\exp\Big(\sum_{f\in \cI} \g(f, I)\Big)\,,
\end{equation}
where $\partial I$ is the set of plaquettes that are $1$-connected to $I$ but not in $I$, and $\g$ is some function satisfying the following: there exist absolute constants $c, K > 0$ such that
\begin{align}\label{eq:g-bound-1-face}
    |\g(f, I)| &\leq K\,, \\
\label{eq:g-bound-2-faces}
    |\g(f, I) - \g(f', I')| &\leq Ke^{-cr(f, I; f', I')}\,,
\end{align}
for all $f, I, f', I'$, where 
\[ r(f, I; f', I') =  \sup \{r: I \cap B_r(f) \cap \Lambda'_{n,m}\equiv I' \cap B_r(f')\cap \Lambda'_{n,m}\}\,,\] i.e., $r(f, I; f', \cI')$ is the largest radius such that the interfaces $I, I'$ agree on the balls of this radius around $f, f'$ (the intersections with a ball of this radius are translates of one another). 
\end{proposition}

\begin{remark}\label{rem:diff-domains}
The above result is stated in 
\cite[Lem.~9]{GielisGrimmett02},\cite[Lem~7.118]{Grimmett_RC}
for the limit as $m\to\infty$, but its analog was
first established for $\bar \pi^{\sfh}_{\Lambda'_{n,m}}(\cI_\Full=I)$ in \cite[Lem.~7]{GielisGrimmett02},\cite[Lem~7.104]{Grimmett_RC} for all $n,m\geq 1$.
The proof further applies to any box $\Lambda = ([-\frac n2,\frac n2]^2\times [m_1,m_2])\cap(\Z+\frac12)^3$ for $m_1 \leq 0$ and $m_2\geq 1$, with the caveat that the bound on $\g$ in terms of $r(f,I;f',I')$ (\cref{eq:g-bound-2-faces}) might render \cref{eq:CE} useless if $m_1$ or $m_2$ were taken too close to $0$ (as $r$ will be capped at the distance to the boundary).
\end{remark}

We note in passing that, in \cref{sec:identification-of-rates}, devoted to studying the limiting large deviation rates of local maxima in the interfaces, we will need to appeal to cluster expansion expressions for more general $1$-connected sets of dual-to-closed plaquettes. See \cref{prop:cluster-expansion} and its proof.

Recall from the introduction that if $f$ is a plaquette in $(\Z+\frac12)^3$---so its midpoint $x=(x_1,x_2,x_3)$ has two half-integer coordinates and one integer coordinate---we write $\hgt(f) := x_3$. In what follows, we address the maximal (minimal) height achieved by an interface $I$ above a given $(x_1,x_2)$,
\[ \overline\hgt_{(x_1,x_2)}(I) := \max\{x_3\,:\;(x_1,x_2,x_3) \in I\}\quad\mbox{and}\quad
\underline\hgt_{(x_1,x_2)}(I) := \min\{x_3\,:\;(x_1,x_2,x_3) \in I\}\] 
(where one views $I$ as a subset of $\R^3$, the union of its plaquettes). 

\begin{theorem}[{\cite[Thm.~3]{GielisGrimmett02},\cite[Thm.~7.144]{Grimmett_RC}}]\label{thm:gg-interface-height}
Let $q\geq 1$, $n \geq 1$ and $m\geq n/2$.
There exist $p_0<1$ such that, for every $p>p_0$ there is some $a=a_p>0$ so the following holds,
\[ \bar\pi^0_{\Lambda'_{n,m}}(\overline\hgt_x(\cI_\Full)\geq k) \leq \exp(-a_p k)\quad\mbox{for every $x=(x_1,x_2)\in(\Z+\frac12)^2$ and $k=k(n)\geq 1$}\,,\]
and, by symmetry, the same holds for the event $\{\underline\hgt_x(\cI_\Full) \leq -k\}$.
\end{theorem}
\begin{remark}
    \label{rem:interface-height-const}
The above result was stated for the limit $m\to\infty$, and as explained in \cref{rem:diff-domains}, is further applicable to $\Lambda'_{n,m}$ with some modifications (the potential pinning effects of the hard floor and ceiling are straightforwardly shown to be negligible when they are positioned at height $\pm c n$; we give a detailed proof in \cref{sec:pf-of-gg-interface-height} for completeness).
In our context of $p=1-e^{-\beta}$, and for $\beta$ large enough the result above holds true for $a = \beta - C$ where $C>0$ is an absolute constant.
\end{remark}

An important ingredient in the proofs is comparison between the hard floor setting of $\mu_n^\fl$ and the ``soft floor'' setting of the conditional distribution of $\mu_n^\sfh$ given that $\cI_\Blue$ lies above height $0$, namely
\begin{equation}
    \label{eq:mu-hat-measure}
    \musoft_n^{\sfh} := \mu_n^\sfh (\cdot \mid \cI_\Blue \subset  \cL_{\geq 0})\,, 
\end{equation}
where $\cL_{\geq 0}$ denotes the positive half-space $\R^2\times[0,\infty)$. (Throughout the paper, we use the notation $\cL_h = \R^2\times\{h\}$ for the slab at height $h$, and write $\cL_{\geq h} = \bigcup_{z\geq h} \cL_z$, and similarly for $\cL_{> h},\cL_{\leq h},\cL_{<h}$.)

\subsection*{Notational comment} Let us conclude this preliminary subsection with a few notational remarks. Throughout the paper $\beta$ will be taken to be a sufficiently large constant independent of $n$, $p$ will equal $1-e^{-\beta}$ per the FK--Potts coupling, and $n$ will always be sufficiently large. We use $\epsilon_\beta$ in different lines to represent (possibly different) small constants that go to zero as $\beta \to\infty$, and $C$ in different lines to represent (possibly different) $\beta$-independent constants.

\subsection{A logarithmic upper bound on the top interface}\label{subsec:log-upper-bound}
As explained by step (2) in the sketch at the beginning of this section, the lower bound argument requires as an input, a logarithmic upper bound on the maximal height of $\Itop$ in the measure $\bar\pi_n^\fl$. Recall that $\cL_{\ge h}$ is the slab $\mathbb R^2\times [h,\infty)$. 

\begin{theorem}
    \label{thm:weak-upper-fl-top}
We have 
    \begin{equation}
        \bar\pi_n^\fl(\Itop \cap \cL_{\geq (14/\beta)\log n} \neq \emptyset) \leq n^{-1+\epsilon_\beta}\,.
    \end{equation}
Consequently, by \cref{eq:interface-order}, the same holds for $\Ibot$ under $\bar\pi_n^\fl$, as well for $\cI_\Blue,\cI_\Red$ under $\mu_n^\fl$. 
\end{theorem}

While the logarithmic upper bound on $\Itop$ implies one for $\cI_\Blue$, our argument for establishing this bound for $\Itop$ goes through first proving such a result directly on $\cI_\Blue$ (\cref{cor:weak-upper-fl-blue} below) and then using the rigidity of the $4$ interfaces (\cref{cor:dist-Ifull-Ibot} below) to conclude the theorem above.

\subsubsection{Upper bound on the blue interface with a soft floor}
We begin with an upper bound on the maximal height of $\cI_\Blue$ in the soft-floor setting of $\musoft_n^\sfh$, with the boundary conditions $\sfh$ set high enough---larger than an absolute constant multiple of $h_n^*$.

\begin{lemma}\label{lem:soft-upper-lower-bound}
    Let $h_0 = \lfloor (3/\beta)\log n\rfloor$. We have
    \begin{equation}
        \musoft_n^{h_0}(\cI_\Blue \cap \cL_{\geq 2h_0} \neq \emptyset) \leq n^{-1+\epsilon_\beta}\,.
    \end{equation}
\end{lemma}
\begin{proof}
By \cref{thm:gg-interface-height,rem:interface-height-const}, a union bound over $x=(x_1,x_2)\in[-\frac n2,\frac n2]^2\cap(\Z+\frac12)^2$, as well as on the two events $\{\overline\hgt_x(\cI_\Full)\geq h_0\}$ and $\{\underline\hgt_x(\cI_\Full) \leq -h_0\}$ in \cref{thm:gg-interface-height} (which, in our context here of $\bar\pi_n^{h_0}$, address the intersection of $\cI_\Full$ with $\cL_{2h_0}$ and $\cL_{0}$, respectively), gives
\[\bar\pi_n^{h_0}\left(\cI_\Full \cap (\cL_{\leq 0}\cup\cL_{\geq 2h_0})\neq \emptyset\right) \leq 2n^2 e^{-(\beta-C)h_0} \leq n^{-1+\epsilon_\beta}\,,\]
where we absorbed the constant pre-factor into the $\epsilon_\beta$ term. In particular, as $\cI_\Blue \subset  \cI_\Full$, we have
\[\mu_n^{h_0}\left(\cI_\Blue \cap (\cL_{\leq 0}\cup\cL_{\geq 2h_0})\neq \emptyset\right) \leq n^{-1+\epsilon_\beta}\,,\]
and thus
\[ \musoft_n^{h_0}(\cI_\Blue \cap \cL_{\geq 2h_0}\neq\emptyset) \leq \frac{
\mu_n^{h_0}(\cI_\Blue \cap \cL_{\geq 2h_0}\neq\emptyset)}
{\mu_n^{h_0}(\cI_\Blue \subset  \cL_{\geq 0})} \leq \frac{n^{-1+\epsilon_\beta}}{1-n^{-1+\epsilon_\beta}} = (1+o(1))n^{-1+\epsilon_\beta} \]
(using the preceding display in both the numerator and denominator above), as required.
\end{proof}

\subsubsection{Comparing the soft and hard floor measures}
We next wish to move from $\musoft_n^{\sfh}$ to~$\mu_n^\fl$. The next proposition, which may be of independent interest, establishes a comparison between these measures via a subtle revealing and monotonicity argument that utilizes the FKG property of the fuzzy Potts model. It is also key to the sharper $h_n^*$ upper bound proved in Section~\ref{sec:sharp-upper}. 

\begin{proposition}\label{prop:Potts-monotonicity}
    Let $\sfh \geq 0$, and in the context of both $\musoft_n^\sfh$ and $\mu_n^\fl$, let $A$ be an event on configurations $\sigma$ on $\Lambda'_{n,n}$ that is increasing in the set of $\Blue$ vertices (that is, the set of configurations in $A$ is closed under the operation of coloring any vertex in $\Blue$) and measurable w.r.t.\ the set $\Ablue^c$.
    Then
    \[
        \mu_n^\fl( A) \le \musoft_n^{\sfh }(A)
    \,,\]
    where the set $\Ablue^c$ is understood w.r.t.\ the corresponding boundary conditions on either side.  
\end{proposition}
\begin{proof}
Consider $\sigma\sim\mu_n^\sfh$, and let us examine the event $\{\cI_\Blue\subset  \cL_{\geq 0}\}$ delineating $\musoft_n^\sfh$.  
Recalling \cref{def:FK-Potts-interfaces}, \cref{it:V-blue-red}, we have $v\in\Ablue^c$ if and only if $v\notin \Vblue$ and there exists a path in   $(\Z+\frac12)^3$, connecting $v$ to a $\Red$ boundary vertex via internal vertices in  $\Ablue^c (\subset  \Vblue^c)$. Further note that
\begin{equation}\label{eq:Iblue-equiv}
\left\{\cI_\Blue \subset  \cL_{\geq 0}\right\} = \bigcap_{v\in\cL_{<0}} \{v\in \Ablue\} 
\,.
\end{equation}
Indeed, the interface $\cI_\Blue$ is nothing but the set of plaquettes separating nearest-neighbors $u,v\in(\Z+\frac12)^3$ with $u\in\Ablue$ and $v\in\Ablue^c$ (in fact with $u\in \Vblue$, as $ \Ablue\setminus\Vblue$ can never be adjacent to $\Ablue^c$). Thus, 
\[ \underline\hgt_{(x_1,x_2)}(\cI_\Blue) =\min\left\{x_3 : (x_1,x_2,x_3-\tfrac12)\in \Ablue\,,\,(x_1,x_2,x_3+\tfrac12)\in \Ablue^c\right\}\,,\]
giving the identity in \cref{eq:Iblue-equiv}.  

Determining whether $\cI_\Blue\subset  \cL_{\geq 0}$ must be done with care; a good example to keep in mind in this context is $\sigma$ that colors every $v\in\cL_{<0}$ in $\Blue$ and every $v\in\cL_{>0}$ in $\Red$ except for $Q=\{\pm\frac12\}^3$, where we swap the assignment, coloring $v\in Q\cap \cL_{<0}$ in $\Red$ and $v\in S\cap\cL_{>0}$ in $\Blue$. In this example, $[0,1]^2\times\{-1\} \subset \cI_\Blue$ (in particular $\cI_\Blue\nsubseteq \cL_{\geq 0}$) and similarly $[0,1]^2\times\{1\}\subset \cI_\Red$, since $Q\subset \Ablue^c \cap \Ared^c$. On the other hand, if we were to also set the color of $(\frac32,\frac12,\frac12)$ to $\Blue$, it would connect $Q\cap\cL_{>0}$ to $\Vblue$, thus leading to $Q\subset\Ablue$ and to $\cI_\Blue=\cI_\Red \subset  \cL_{\geq 0}$. 

\begin{figure}
  \begin{tikzpicture} 
  
  \pgfmathsetmacro{\circmult}{0.35}
  \pgfmathsetmacro{\circscale}{0.6}

   \begin{scope}[shift={(2.7,3)}]
  \foreach \a / \b in {1/5,2/5,3/6,3/7,4/8,5/7, 6/6,7/5,8/6,9/7,10/6,11/6,12/6,13/5,14/5,15/5, 16/5}
	\node [circle,fill=blue!50,inner sep=0.2ex,outer sep=0.1cm] (V\a\b) at (\a*\circmult,\b*\circmult) {};
 \newcommand{\interA}{($(V15.north) + (-0.5*\circmult,0)$) -- (V15.north)  to  (V25.north) to [bend right=45] (V36.west) to (V37.west) to [bend left=45] (V37.north) to [bend right=45] (V48.west) to[bend left=45]  (V48.north) to [bend left=45] (V48.east) to [bend right=45] (V57.north) to [bend left=45] (V57.east) to [bend right=45] (V66.north) to [bend left=45] (V66.east) to [bend right=45] (V75.north) to [bend right=45] (V86.west) to [bend left=45] (V86.north) to [bend right=45] (V97.west) to [bend left=45] (V97.north) to [bend left=45] (V97.east) to [bend right=45] (V106.north) -- (V126.north) to [bend left=45] (V126.east) to [bend right=45] (V135.north) -- (V165.north) to ++ (0.5*\circmult,0)};
	\newcommand{\closeA}{($(V165.north) + (0.5*\circmult,0)$) -- ++ (0,-4*\circmult) -- ++(-16*\circmult,0) -- ++(0,4*\circmult) }
	\path[fill=gray!20] \interA -- \closeA;
	\draw[ultra thick, black] \interA ;
	
  \foreach \a / \b in {1/6, 1/7, 1/8, 1/9, 4/6,4/7,5/5,5/6,6/5,
  9/5,9/6,10/5,11/5,12/5,
  16/6, 16/7, 16/8, 16/9}
	\node [circle,fill=red!50,
 inner sep=0.4ex,outer sep=0.1cm] (V\a\b) at (\a*\circmult,\b*\circmult) {};
  \foreach \a / \b in {1/2,1/3,1/4, 1/5,2/5,3/5,3/6,3/7,4/5,4/8,5/7, 6/6,7/5,8/5,8/6,9/7,10/6,11/6,12/6,13/5,14/5,15/5,16/2,16/3,16/4,16/5}
	\node [circle,fill=blue!50,inner sep=0.4ex,outer sep=0.1cm] (V\a\b) at (\a*\circmult,\b*\circmult) {};

    \draw[dashed, gray, thin] ($0.5*(V15)+0.5*(V16)$)--($0.5*(V165)+0.5*(V166)$);    

    \node [font=\large,gray!75!black] at (3,1) {$V_0$};
    \node [font=\large,gray!75!black] at (3,3.1) {$V_1$};
    \node [font=\large] at (4.5,2.5) {$\Gamma$};
    
	 \end{scope}

  \begin{scope}[scale=0.9,shift={(0,0)}]
  \foreach \a / \b in 
  {1/5,2/5,3/6,3/7,
 4/8,5/4,5/8, 6/7,7/5,7/6,8/5,8/6,9/4,12/4,13/5,14/5,15/5, 16/5} \node [circle,fill=blue!50,inner sep=0.2ex,outer sep=0.1cm] (V\a\b) at (\a*\circmult,\b*\circmult) {};

  \foreach \a / \b in 
  {1/6,2/6,2/7,3/8,3/9,4/9,5/9,6/8,7/6,7/7,8/7,8/8,9/8,10/8,11/7,12/7,13/6,14/6,15/6, 16/6} \node [circle,fill=blue!50,inner sep=0.2ex,outer sep=0.1cm] (V\a\b) at (\a*\circmult,\b*\circmult) {};
 
 \newcommand{\interB}{($(V15.north) + (-0.5*\circmult,0)$) -- (V15.north)  to  (V25.north) to [bend right=45] (V36.west) to (V37.west) to [bend left=45] (V37.north) to [bend right=45] (V48.west) to [bend left=45] (V48.north) -- (V58.north) to [bend left=45] (V58.east) to[bend right=45]  (V67.north) to [bend left=45] (V67.east) to [bend right=45] (V76.north) -- (V86.north) to [bend left=45] (V86.east) -- (V85.east) to [bend right=45] (V94.north) -- (V124.north) to [bend right=45] (V135.west) to [bend left=45] (V135.north) -- (V165.north) to ++ (0.5*\circmult,0)};
	\newcommand{\closeB}{($(V165.north) + (0.5*\circmult,0)$) -- ++ (0,-4*\circmult) -- ++(-16*\circmult,0) -- ++(0,4*\circmult) }
 \newcommand{\interR}{($(V16.south) + (-0.5*\circmult,0)$) -- (V16.south)  to  (V26.south) to [bend right=45] (V26.east) -- (V27.east) to [bend left=45] (V38.south) to [bend right=45] (V38.east) to [bend left=45] (V49.south) -- (V59.south) to [bend left=45]  (V58.east) to [bend right=45] (V68.south) to [bend left=45] (V77.west) to [bend right=45] (V77.south) -- (V87.south) to [bend right=45] (V87.east) to [bend left=45] (V98.south) -- (V108.south) to [bend left=45] (V117.west) to [bend right=45] (V117.south) -- (V127.south) to [bend left=45] (V136.west) to [bend right=45] (V136.south) -- (V166.south) to ++ (0.5*\circmult,0)};
	\newcommand{\closeR}{($(V165.north) + (0.5*\circmult,0)$) -- ++ (0,4*\circmult) -- ++(-16*\circmult,0) -- ++(0,-4*\circmult) }

 \path[fill=blue!15,
 transform canvas={shift={(0,-0.03)}}
 ] \interB -- \closeB;
 \path[fill=red!15,
 transform canvas={shift={(0,0.03)}}
 ] \interR -- \closeR;
	\draw[thick, red!50,transform canvas={shift={(0,0.03)}}] \interR ;
	\draw[thick, blue!50,transform canvas={shift={(0,-0.03)}}] \interB ;
	
  \foreach \a / \b in { 
  3/6,3/7,4/8,5/7,5/8, 6/6,6/7,8/6,9/7,10/6,10/7,11/6,12/6,  
  11/8,12/8,13/8,
  1/5,2/5,3/5,4/5,7/5,7/6,8/5,13/5,14/5,15/5,16/5}
	\node [circle,fill=blue!50,inner sep=0.4ex,outer sep=0.1cm] (V\a\b) at (\a*\circmult,\b*\circmult) {};
 
    \foreach \a in {1,2,...,16}
    \foreach \b in {2,3,4}
	\node [circle,fill=blue!50,inner sep=0.4ex,outer sep=0.1cm] (V\a\b) at (\a*\circmult,\b*\circmult) {};

  \foreach \a / \b in {1/6, 1/7, 1/8, 1/9,2/6,2/7,2/8,2/9,3/8,3/9,4/6,4/7,4/9,5/5,5/6,5/9,6/5,6/8,6/9,7/7,7/8,7/9,
    8/7,8/8,8/9,9/5,9/6,9/8,9/9,10/5,10/8,10/9,11/5,11/7,11/9,12/5,12/7,12/9,13/6,13/7,13/9,14/6,14/7,14/8,14/9,15/6,15/7,15/8,15/9,
  16/6, 16/7, 16/8, 16/9,
  8/3,9/3,10/3,11/3}
	\node [circle,fill=red!50,
 inner sep=0.4ex,outer sep=0.1cm] (V\a\b) at (\a*\circmult,\b*\circmult) {};
 
    \draw[dashed, gray, thin] ($0.5*(V15)+0.5*(V16)$)--($0.5*(V165)+0.5*(V166)$);

    \node [circle,draw=black,line width=1pt,
 inner sep=0.5ex,outer sep=0.1cm] at (V136) {};
 
	 \end{scope}
  \begin{scope}[scale=0.9,shift={(6.5,0)}]
  \foreach \a / \b in 
  {1/5,2/5,3/6,3/7,
 4/5,4/8,5/4,5/8, 6/7,7/5,7/6,8/5,8/6,9/7,10/7,11/6,12/4,13/5,13/6,14/5,15/5, 16/5} \node [circle,fill=blue!50,inner sep=0.2ex,outer sep=0.1cm] (V\a\b) at (\a*\circmult,\b*\circmult) {};

  \foreach \a / \b in 
{1/6,2/6,2/7,3/8,3/9,4/9,5/9,6/8,7/7,8/7,8/8,9/8,10/8,11/7,13/7,14/6,15/6, 16/6} \node [circle,fill=blue!50,inner sep=0.2ex,outer sep=0.1cm] (V\a\b) at (\a*\circmult,\b*\circmult) {};
 
 \newcommand{\interB}{($(V15.north) + (-0.5*\circmult,0)$) -- (V15.north)  to  (V25.north) to [bend right=45] (V36.west) to (V37.west) to [bend left=45] (V37.north) to [bend right=45] (V48.west) to [bend left=45] (V48.north) -- (V58.north) to [bend left=45] (V58.east) to[bend right=45]  (V67.north) to [bend left=45] (V67.east) to [bend right=45] (V76.north) -- (V86.north) to [bend right=45] (V97.west) to [bend left=45] (V97.north) -- (V107.north) to [bend left=45] (V107.east) to [bend right=45] (V116.north) -- (V136.north) to [bend left=45] (V136.east) to [bend right=45] (V145.north) -- (V165.north) to ++ (0.5*\circmult,0)};
	\newcommand{\closeB}{($(V165.north) + (0.5*\circmult,0)$) -- ++ (0,-4*\circmult) -- ++(-16*\circmult,0) -- ++(0,4*\circmult) }
 \newcommand{\interR}{($(V16.south) + (-0.5*\circmult,0)$) -- (V16.south)  to  (V26.south) to [bend right=45] (V26.east) -- (V27.east) to [bend left=45] (V38.south) to [bend right=45] (V38.east) to [bend left=45] (V49.south) -- (V59.south) to [bend left=45]  (V58.east) to [bend right=45] (V68.south) to [bend left=45] (V77.west) to [bend right=45] (V77.south) --
 (V87.south) to [bend right=45] (V87.east) to [bend left=45] (V98.south) -- (V108.south) to [bend left=45] (V117.west) to [bend right=45] (V117.south) -- (V137.south) to [bend left=45] (V146.west) to [bend right=45] (V146.south) -- (V166.south) to ++ (0.5*\circmult,0)};    
	\newcommand{\closeR}{($(V165.north) + (0.5*\circmult,0)$) -- ++ (0,4*\circmult) -- ++(-16*\circmult,0) -- ++(0,-4*\circmult) }

 \path[fill=blue!15,
 transform canvas={shift={(0,-0.03)}}
 ] \interB -- \closeB;
 \path[fill=red!15,
 transform canvas={shift={(0,0.03)}}
 ] \interR -- \closeR;
	\draw[thick, red!50,transform canvas={shift={(0,0.03)}}] \interR ;
	\draw[thick, blue!50,transform canvas={shift={(0,-0.03)}}] \interB ;
	
  \foreach \a / \b in { 
  3/6,3/7,4/8,5/7, 5/8,6/6,6/7,8/6,9/7,10/6,10/7,11/6,12/6,  11/8,12/8,13/8,
  1/5,2/5,3/5,4/5,7/5,7/6,8/5,13/5,13/6,14/5,15/5,16/5}
	\node [circle,fill=blue!50,inner sep=0.4ex,outer sep=0.1cm] (V\a\b) at (\a*\circmult,\b*\circmult) {};
 
    \foreach \a in {1,2,...,16}
    \foreach \b in {2,3,4}
	\node [circle,fill=blue!50,inner sep=0.4ex,outer sep=0.1cm] (V\a\b) at (\a*\circmult,\b*\circmult) {};

    \draw[dashed, gray, thin] ($0.5*(V15)+0.5*(V16)$)--($0.5*(V165)+0.5*(V166)$);

  \foreach \a / \b in {1/6, 1/7, 1/8, 1/9,2/6,2/7,2/8,2/9,3/8,3/9,4/6,4/7,4/9,5/5,5/6,5/9,6/5,6/8,6/9,7/7,7/8,7/9,
8/7,8/8,8/9,9/5,9/6,9/8,9/9,10/5,10/8,10/9,11/5,11/7,11/9,12/5,12/7,12/9,13/7,13/9,14/6,14/7,14/8,14/9,15/6,15/7,15/8,15/9,
  16/6, 16/7, 16/8, 16/9,
  8/3,9/3,10/3,11/3}
	\node [circle,fill=red!50,
 inner sep=0.4ex,outer sep=0.1cm] (V\a\b) at (\a*\circmult,\b*\circmult) {};
 
    \node [circle,draw=black,line width=1pt,
 inner sep=0.5ex,outer sep=0.1cm] at (V136) {};
  \end{scope}
\end{tikzpicture}
\caption{Top: the vertex set $\Gamma$ from the proof of \cref{prop:Potts-monotonicity} for $\sfh=0$, revealed with the rest of the vertices of $\cL_{<0}$ to form $V_0$ (marked in gray). Bottom: in this example, flipping the color of a single site $v\notin V_0$ (circled in black) is pivotal 
 for the event $\{\cI_\Blue\subset \cL_{\geq 0}\}$.}
\label{fig:V0-Iblue-Ired}
\vspace{-0.1in}
\end{figure}

With this in mind, consider $\sigma\sim\musoft_n^{\sfh}=\mu_n^\sfh(\cdot \mid \cI_\Blue\subset \cL_{\geq 0})$; 
let us reveal the color of every $v\in\cL_{<0}$, as well as the $\noB$ component $\cC_v$ of every $v\in\cL_{-1/2}$. Denote by $V_0$ the vertices whose color was revealed, and let $\Gamma$ denote its inner vertex boundary (vertices $x\in V_0$ adjacent to some $y\notin V_0$). By definition, every $x\in \Gamma$ is $\Blue$; indeed, either $x\in\cL_{-1/2}$, or it is in the external boundary of  $\cC_v$, and must be $\Blue$ since no $\cC_v$ reaches a $\Red$ boundary vertex (as this would imply that $v\in\Ablue^c$). Let $V_1$ denote the vertices whose color was not yet revealed.
(N.B.\ every $(\Z+\frac12)^3$-path from $x\in V_1$ to $y\in\cL_{<0}$ must pass through $\Gamma$.) See \cref{fig:V0-Iblue-Ired} for an illustration of the sets $V_0$ and $\Gamma$, as well as the fact that conditional event $\{\cI_\Blue\subset \cL_{\geq 0}\}$ gives information on the colors of $V_1$, as every path from a $\noB$ site $v\in\cL_{-1/2}$ to the boundary must go through some $u\in\Vblue$, or else we will have $v\in\Ablue^c$. (Were we to ignore this conditional event and only look at the colors of $V_0$, then even though the vertices in $\Gamma$ are all $\Blue$, some of these may or may not be in $\Vblue$, depending on the colors in~$V_1$. This plays a role in the conditioning, since some of $\noB$ vertices $v\in\cL_{-1/2}$ must be encapsulated in $\Ablue$ via a boundary of $\Vblue$ vertices using sites also from~$V_1$.)

For a realization $\sigma\restriction_{V_0}=\eta$ (that is compatible with the event $\cI_{\Blue} \subset  \cL_{\geq 0}$), let $\cB_\eta$ denote the $\sigma\restriction_{V_1}$-measurable event $\{\cI_{\Blue}\subset \cL_{\geq0}\}$ under $\mu_n^\sfh$ conditional on $V_0\,,\sigma\restriction_{V_0}=\eta$. Formally, $B_{\eta}$ is the event that every vertex of $\Gamma$ not shown to be in $\Ablue$ by $\eta$, is indeed part of $\Ablue$ in $\eta$ concatenated with $\sigma\restriction_{V_1}$. 
By this reasoning, 
\begin{align*}
    \musoft_n^\sfh( \cdot \mid V_0,\, \sigma\restriction_{V_0} = \eta) = \mu_n^\sfh(\cdot \mid V_0 ,\,\sigma\restriction_{V_0} = \eta,\, \cB_\eta)\,.
\end{align*}
Note that $\cB_\eta$ is increasing in $\Blue$ vertices (stemming from the monotonicity of the set $\Ablue$ in the set of $\Blue$ vertices). 

Having exposed $V_0$ and $\sigma\restriction_{V_0}$, we wish to show, uniformly over the realization $\eta$ of the latter, that
\begin{equation}\label{eq:mu(A)-given-eta} \mu_n^\sfh(A \mid V_0\,,\sigma\restriction_{V_0}=\eta\,,\cB_\eta) \geq \mu_n^\fl(A)\end{equation}
(at which point, taking an expectation over $V_0$ and $\eta$ will conclude).
To this end, we will establish the following, where here and in what follows, we say that a coloring $\eta$ of a vertex set $S\supseteq \cL_{<0}$ realizes $V_0=S$ if every $\sigma$ that agrees with $\eta$ on $S$ satisfies $V_0(\sigma)=\eta$ (it is important to recall that one can deterministically verify that $V_0(\sigma)=S$ from $\sigma\restriction_S$).

\begin{claim}\label{clm:change-eta-to-eta'}
For every fixed subset of vertices $S\supseteq \cL_{<0}$ and configuration $\eta:S\to \{1,\ldots,q\}$ realizing $V_0=S$, and every $\eta'$ obtained from $\eta$ by coloring an arbitrary subset of $S$ in $\Blue$,
\begin{align}\label{eq:change-eta-to-eta'}
    \mu_n^\sfh(\Ablue^c\in\cdot \mid  \sigma\restriction_{S} = \eta, \cB_\eta)  = \mu_n^\sfh( \Ablue^c\in\cdot \mid \sigma\restriction_{S} = \eta', \cB_\eta)\,.
\end{align}
\end{claim}
\begin{proof}
We will establish that \begin{equation}\label{eq:nts-in-monotonicity-claim}\text{$\forall \sigma : \sigma\restriction_S = \eta$ and $\sigma\restriction_{S^c} \in B_\eta$, changing $\eta$ to $\eta'$ does not change $\Ablue$} \end{equation}

To show this, we first argue that $S\subset  \Ablue$ on the events $\sigma\restriction_S =\eta$ and $\sigma\restriction_{S^c}\in\cB_\eta$. Indeed, suppose that $x\in S$ is not in $\Ablue$. If $x\in \cL_{<0}$, then since $\cI_{\Blue}\subset  \cL_{\geq0}$, then \cref{eq:Iblue-equiv} gives a contradiction. If $x\in \cL_{>0}$ and $x\notin \Ablue$, there must be a $\Ablue^c$ path connecting $x$ to the $\Red$ boundary $\partial \Lambda'_{n,n} \cap \cL_{>0}$; at the same time, $x$ must be incident to, or part of, a $\noB$-component $\cC_v$ for some $v\in \cL_{<0}$. Appending said path to the path connecting $x$ to $v$ in $\cC_v$ gives a $\Ablue^c$ path from the $\Red$ boundary to $\cL_{<0}$ (the vertices of $\cC_v$ are not in $\Vblue$, and every path from $\Ablue$ to $\Ablue^c$ goes through $\Vblue$),
contradicting that $\cI_{\Blue} \subset  \cL_{\ge 0}$. Now, $S\subset  \Ablue$ means that changing colors in $S$ to $\Blue$ does not change $\Ablue$, and evidently does not affect the configuration on $S^c$ which $\Ablue^c$ is a subset of.

Having established~\eqref{eq:nts-in-monotonicity-claim}, this implies that on $\{\sigma\restriction_{S} = \eta,\, \sigma\restriction_{S^c}\in\cB_\eta\}$, a realization of $\Ablue^c$ corresponds to the same set of configurations 
on $\sigma\restriction_{S^c}$ as it does on $\{\sigma_S=\eta',\sigma_{S^c}\in\cB_\eta\}$. The event $B_\eta$ is by definition measurable with respect to $\sigma\restriction_{S^c}$. Therefore, we can apply the domain Markov property to see that the law in the left-hand of \cref{eq:change-eta-to-eta'} only depends on $\sigma\restriction_S = \eta$ through the inner vertex boundary of $S$, which is $\Blue$ in $\eta$ by definition (recall the set $\Gamma$ defined as the inner vertex boundary of $V_0$), and hence also in $\eta'$. This concludes the proof of \cref{eq:change-eta-to-eta'}.
\end{proof}

To now show \eqref{eq:mu(A)-given-eta}, take $\eta'$ in \cref{clm:change-eta-to-eta'} to be all-$\Blue$, and use that $A$ is measurable with respect to  $\Ablue^c$, to see that 
\begin{align*}
    \mu_n^\sfh(A\mid \sigma\restriction_S = \eta, \cB_\eta) = \mu_n^\sfh (A \mid \sigma\restriction_S \equiv \Blue, \cB_\eta)\,.
\end{align*}
The right-hand, by Domain Markov, is nothing but a Potts model on $\Lambda_{n,n}$ (as opposed to $\Lambda'_{n,n}$), conditional on $\cB_\eta$ as well having the sites of $S$ be all-$\Blue$, and having boundary conditions that are $\Blue$ in $\cL_{<\sfh}$ and $\Red$ above it (the sites of $S$ separate the bottom face from a path to any vertex of $S^c$). Notice that the configuration on $\cL_{<0}$ being $\Blue$ ensures that interpreting $\Ablue$ w.r.t.\ these boundary conditions is the same as w.r.t.\ $\mu_n^\sfh$ on $\Lambda_{n,n}'$. 
At this point, we wish to use that these are all modifications of $\mu_n^\fl$ that are increasing in the set of $\Blue$ vertices, so the FKG inequality for the fuzzy Potts model should allow us to drop the conditionings and only decrease the probability. 
Still, unlike the Potts model, there is no Domain Markov Property in fuzzy Potts, and so this step must be done with care. 
First, let  
\[ \mu_n^\Gamma = \mu_{\Lambda_{n,n}}^\sfh\left(\sigma\restriction_{S^c}\in\cdot\given \sigma\restriction_{S}= \eta, \cB_\eta\right)\] 
denote the Potts distribution on $S^c$ given $\sigma\restriction_{S} = \eta$ as well as $\cB_\eta$, where the $\Red$ external boundary vertices (those in $\cL_{>\sfh}$) are replaced by a single (high degree) $\Red$ boundary vertex $v_\Red$, and the
external boundary vertices in $S \cup \cL_{<\sfh}$ are all $\Blue$. (Replacing the $\Red$ boundary vertices by a single $v_\Red$ does not change the Potts model, yet it will facilitate easy transitions to the fuzzy Potts model and back.)  As an additional adjustment, let us connect each of the $\Blue$ external boundary vertices in $\partial \Lambda_{n,n}\cap (\cL_{>0}\cap \cL_{<\sfh})$ to $v_\Red$ with, a bond of strength $J>0$ (to be taken arbitrarily large). This, again, does not affect $\mu_n^\Gamma$ by the Domain Markov property.

At this point, we move to the fuzzy Potts model $\phi_n^\Gamma$, 
noting that $\Ablue^c$ (thus also $A$, by assumption) is measurable under the fuzzy Potts projection $f_{\textsf{bf}}$. 
By FKG, and the fact that $A$ is increasing in $\Blue$ vertices, we can
\begin{enumerate}[(a)]\item remove the $\Blue$ boundary conditions on the sites of $S\cap \cL_{>0}$ and  $\partial \Lambda_{n,n}\cap (\cL_{>0}\cap \cL_{<\sfh})$, and \item remove the conditioning on the event $\cB_\eta$,
\end{enumerate} 
arriving at $\phi_n^\fl$, a fuzzy Potts measure on $\Lambda_n$ with a hard $\Blue$ floor at height $0$, a single vertex $v_\Red$ replacing all the $\Red$ boundary vertices in $\cL_{>\sfh}$, and that vertex connected to every formerly $\Blue$ boundary condition vertex in $\partial \Lambda_{n,n}\cap \cL_{>0}\cap \cL_{<\sfh}$ via edges with interaction strength $J$):
\[
\mu_n^\Gamma(A) = \phi_n^\Gamma(A)  \geq \phi_n^\fl(A) = \mu_n^\fl(A) + O(n\sfh e^{- \beta J})\,.\]
Here, the last identity uses that by a union bound, the $\Fuzzy$ boundary vertices will have the same color except with probability $O(n\sfh e^{- \beta J})$; a single $\Fuzzy$ boundary vertex in fuzzy Potts corresponds to a single $\noB$ boundary vertex in Potts, which by symmetry can be taken to be $\Red$. Taking $J$ arbitrarily large completes the proof.
\end{proof}

\begin{remark}
In \cref{prop:Potts-monotonicity} we required that $A$, the event that is increasing in the set of $\Blue$ vertices, should depend only on $\Ablue^c$. The assertion of the proposition would fail if we omitted that requirement; e.g., if $\mathsf{o}$ denotes the origin vertex $(\frac{1}{2},\frac{1}{2},\frac{1}{2})$ and $A=\{\sigma_{\mathsf{o}} = \Blue\}$ then one would have
$ \musoft_n^0(A) < \mu_n^\fl(A)$ (whereas the event $A=\{\sigma_{\mathsf{o}}\in\Ablue\}$ would be covered by the proposition and have $\musoft_n^0(A)\geq \mu_n^\fl(A)$).
\end{remark}

\begin{corollary}\label{cor:weak-upper-fl-blue}
For large enough $n$,
    \begin{equation}
        \mu^\fl_n(\cI_\Blue \cap \cL_{\geq  (6/\beta)\log n} \neq \emptyset) \leq n^{-1+\epsilon_\beta}\,.
    \end{equation}
\end{corollary}
\begin{proof}
Let $h_0 = \lfloor(3/\beta)\log n\rfloor$, let $A$ be the event $\{\cI_\Blue \cap \cL_{\geq 2h_0} \neq \emptyset \}$, and recall the upper bound from \cref{lem:soft-upper-lower-bound} on $\musoft_n^{h_0}(A)$. Since $A$ is increasing in the set 
 of $\Blue$ vertices, and is measurable w.r.t.~$\Ablue^c$, we may apply \cref{prop:Potts-monotonicity} for $\sfh=h_0$ to extend this bound to $\mu_n^\fl(A)$, as required.
\end{proof}

\subsubsection{Rigidity of the four interfaces}
In this subsection we employ a standard Peierls argument to infer the rigidity of the interfaces  $\Ibot\preceq \cI_\Blue \preceq \cI_\Red \preceq \Itop$ from \cref{def:FK-Potts-interfaces}, which will then allow us to extend the logarithmic upper bound from $\cI_\Blue$ to the the other three interfaces. 

Recall that $\Abot \subset  \Atop^c$. The following lemma bounds the tail probability for the boundary size of a $(\Z+\frac12)^3$-connected component in $\Atop^c \setminus \Abot$. 
\begin{lemma}\label{lem:Gamma-outside-Atop-Abot}
Let $\omega\sim\bar\pi_n^\fl$, consider $v\in\Lambda_{n,n}$. Let $\cC_v$ be the (potentially empty) $(\Z+\frac12)^3$-connected component of the vertex $v$ in $\Lambda_{n,n}\setminus (\Atop \cup \Abot)$. Let $\Gamma_v$ be its plaquette boundary, that is, $f\in \Gamma_v$ if it is dual to the $(\Z+\frac12)^3$-bond $(u,w)$ with $u\in\cC_v$ and $w\notin\cC_v$. Then 
    \begin{align*}
        \bar\pi_n^\fl(|\Gamma_v | \ge r) \le e^{ - (\beta-C) r/2}\,.
    \end{align*}
\end{lemma}
\begin{proof}
It is convenient for this proof to have the vertex set of $\Lambda_{n,n}$ include its external boundary vertices, letting those in $\cL_{>0}$ be part of $\Atop$ in $\omega\sim\bar\pi_n^\fl$, and those of $\cL_{<0}$ be part of $\Abot$. As such, every plaquette $f\in\Gamma_v$ must separate $x\in \cC_v$ from some $y\in\Abot \oplus \Atop$, whence $f\in \Ibot \oplus \Itop$. 

Let 
\[A_k := \{|\Gamma_v|=k\,,\,|\Gamma_v \cap \Itop|\geq k/2\}\,.\]
The map $\Phi$ that opens every edge dual to $\Gamma_v \cap \Itop$ maintains that $\Atop,\Abot$ are disconnected in $\omega$ (having  only opened edges between $\Atop$ and $\cC_v\subset  \Abot^c$), and
\[ \frac{\bar\pi_n^\fl(\Phi(\omega))}{\bar\pi_n^\fl(\omega)} \geq \Big(\frac{p}{(1-p)q}\Big)^{|\Gamma_v\cap\Itop|}\geq e^{(\beta-\log q)k/2}
\]
uniformly over $\omega\in A_k$ (each edge modification gains a factor of $\frac{p}{1-p}$, perhaps decreasing  $\kappa(\omega)$ by~$1$).
Enumerating over the connected set of plaquettes forming $\Gamma_v$ now shows that $\bar\pi_n^\fl(A_k) \leq \exp(-(\beta-C)k/2)$ for some absolute constant $C>0$.
The analogous argument treating 
\[B_k := \{|\Gamma_v|=k\,,\,|\Gamma_v\cap\Ibot|\geq k/2\}\] via opening the edges $\Gamma_v\cap\Ibot$ shows that $\bar\pi_n^\fl(B_k) \leq \exp(-(\beta-C)k/2)$ (with the same $C>0$), which, since $\{|\Gamma_v|\geq r\}\subset  \bigcup_{k\geq r} (A_{k}\cup B_{k})$, concludes the proof.
\end{proof}

Similarly, we can easily bound the differences between the $\Top$ and $\Full$ interfaces.
\begin{lemma}\label{lem:bound-on-hairs}
    Fix an edge $e \in \Lambda_n$ and consider the dual face $f = f_e$. Let $C_f$ be the (potentially empty) 1-connected component of $f$ in $\cI_\Full \setminus \cI_\Top$. Then, 
    \begin{align*}
        \bar{\pi}_n^\fl(|C_f| \ge r) \le e^{ - (\beta-C) r}\,.
    \end{align*}
    The same statement holds if we define $C_f$ to be the 1-connected component of $f$ in $\cI_\Full \setminus \cI_\Bot$.
\end{lemma}
\begin{proof}
   Take the map $\Phi$ that opens every edge dual to $C_f$. Since $C_f \subset \cI_{\Full} \setminus \cI_{\Top}$, this operation keeps the resulting configuration in $\sep_{n,n}$. The same short computation as in the proof of~\cref{lem:Gamma-outside-Atop-Abot} immediately implies the desired bound.
\end{proof}

\begin{corollary}\label{cor:dist-Ifull-Ibot}
Fix $\beta$ large enough. There exists a constant $C > 0$ such that for any $k \geq 0$,
\[ \bar\pi_n^\fl\left(\max_x \overline\hgt_x(\cI_\Full) > \max_x \overline\hgt_x(\Ibot) + k\right) \leq Cn^3e^{-(\beta - C)k/3}\,, \]
and
\[ \bar\pi_n^\fl\left(\min_x \underline\hgt_x(\cI_\Full) < \min_x \underline\hgt_x(\Itop) - k\right) \leq Cn^3e^{-(\beta - C)k/3}\,.  \]
By the ordering of the interfaces~\eqref{eq:interface-order}, the distance between the maximum of any of the interfaces $\Full, \Top, \Red, \Blue, \Bot$ satisfies the same bound, and likewise for the minimum.
\end{corollary}
\begin{proof}
Every $f\in\Itop\setminus \Ibot$  belongs to the boundary $\Gamma_v$ of some component $\cC_v$ as per the definition in \cref{lem:Gamma-outside-Atop-Abot}. Note that said boundary must also contain some plaquette $f'\in\Ibot\setminus\Itop$ (some external boundary vertex of $\cC_v$ must be in $\Abot$, or else $\cC_v$ would be part of $\Atop$), whence $|\Gamma_v|$ is an upper bound on the length of a path of plaquettes connecting $f,f'$. Hence, the difference between $\max_x \overline\hgt_x(\Itop)$ and $ \max_x \overline\hgt_x(\Ibot)$ is bounded by $\max_v |\Gamma_v|$. Similarly, the difference between $\max_x\overline\hgt_x(\cI_\Full)$ and $\max_x\overline\hgt_x(\cI_\Top)$ is bounded by $\max_f |C_f|$.
We conclude by Lemmas~\ref{lem:Gamma-outside-Atop-Abot}--\ref{lem:bound-on-hairs} and a union bound over the $O(n^3)$ vertices in $\Lambda_n$.
\end{proof}

\begin{proof}[\textbf{\emph{Proof of \cref{thm:weak-upper-fl-top}}}]
Let $(\sigma,\omega)\sim\bP_n^\fl$. 
By taking $k = (12/\beta)\log n$ in \cref{cor:dist-Ifull-Ibot}, combined with the ordering of the interfaces in \cref{eq:interface-order}, 
with probability at least $1-n^{-1+\epsilon_\beta}$ we have that 
 for every $x$,
\[ \overline\hgt_x(\Itop) - \overline\hgt_x(\cI_\Blue) \leq 
\overline\hgt_x(\Itop) - \overline\hgt_x(\Ibot)\leq (12/\beta)\log n\,. 
\]
The bound on $\max_x \overline\hgt_x(\cI_\Blue)$ from  \cref{cor:weak-upper-fl-blue} thus concludes the proof. 
\end{proof}

\subsection{Monotonicity of the effect of lifting the top interface with a hard vs.\ soft floor}\label{subsec:Itop-monotonicity}

The basic mechanism for establishing the typical height of the interface above a floor (hard or soft) is to compare the effect of lifting the entire interface, say by a single $n\times n \times 1$ slab: one wishes to weigh the gain from entropic repulsion (the ability to drive deeper spikes in the bulk of the surface) against the loss in energy (due to the additional plaquettes along the perimeter of the newly added slab). However, for the Potts interfaces (e.g., $\cI_\Blue$), the interactions with the hard floor might outweigh the entropic repulsion effect (through the change in partition functions induced below the interface). In fact, curiously, their exponential rates appear to coincide, at least up to the first order term as a function of $\beta$. Fortunately, the interface $\cI_\Top$ (but not $\cI_{\Full}$ nor $\Ibot$) exhibits a nicer interaction with the hard floor; the following lemma shows, roughly, that  interactions with the hard floor make it ``only easier to lift the interface'' then.

\begin{lemma}\label{lem:compare-fl-dob}
For any $\Top$ interface $I\subset \cL_{\geq 0}$ and $j\geq 1$, let $\Theta_j I$ denote the interface obtained by shifting $I$ up by $j$ and adding the $4jn$ additional boundary plaquettes of heights in $[0,j]$,  and
set
\begin{equation}\label{eq:Xi-fl-dob-def} \Xi_j^{\fl}(I) := \frac{\bar\pi_n^{\fl}(\Itop = \Theta_j I)}{\bar\pi_n^{\fl}(\Itop =  I)}\qquad\mbox{and} \qquad \Xi^{\dob}_j(I) := \frac{\bar\pi_n^{\dob} (\Itop = \Theta_j I)}{\bar\pi_n^{\dob} (\Itop = I)}\,. \end{equation}
If $I$ and $j$ are such that $\max_{x}\overline\hgt_x(I)< j \leq n/2$, then
$ \Xi_j^\fl(I) \geq \Xi_j^\dob(I)$.
\end{lemma}

To compare the effects of the map $\Theta_j$ in $\bar\pi_n^\fl$ vs.\ $\bar\pi_n^\dob$, we will need the following consequence of well-known expressions for derivatives of free energies in edge parameters. 
\begin{claim}\label{clm:free-energy}
For any finite graph $(V,E)$ and an edge set $\tilde E$, fix $p,q$ and let $\tilde Z_G^{(\theta)}$ be the partition function of the $(p,q)$-FK model on $G=(V, E \cup \tilde E)$ modified to have edge probability $\theta$ for $e\in\tilde E$. 
Denote by $\tilde\pi_G^{(\theta)}$ the probability of an edge configuration $\omega$ in this model. Then 
\[ \log \tilde Z_G^{(\theta_1)} - \log \tilde Z_G^{(\theta_0)} = \sum_{e\in \tilde E}\int_{\theta_0}^{\theta_1} \frac{\tilde\pi_G^{(\theta)}(\omega_e = 1)}{\theta(1-\theta)}\,\d \theta \,.\]
\end{claim}
\begin{proof}
As per \cref{eq:fk-measure}, we have 
\[ \tilde Z_G^{(\theta)} = \sum_{\omega}   \Big(\frac{p}{1-p}\Big)^{| \omega\cap E|}\Big(\frac{\theta}{1-\theta}\Big)^{|\omega\cap \tilde E|} q^{\kappa(\omega)} =: \sum_{\omega} W_{\omega}
\,.\]
Noting that $\frac{\d}{\d\theta} W_{\omega} = \big( |\omega\cap \tilde E|/[\theta (1-\theta)]\big)W_{\omega}$,  if we consider expectation w.r.t.\ $\tilde\pi_G^{(\theta)}$ , then 
\[
\frac{\d}{\d \theta} \log \tilde Z_G^{(\theta)} 
= \frac1{\tilde Z_G^{(\theta)}}\sum_{\omega} 
\frac{|\omega\cap \tilde E|}{\theta(1-\theta)}W_{\omega}
= \sum_{e\in \tilde E}  \frac{\tilde \pi_G^{(\theta)}(\omega_e =1)}{\theta(1-\theta)}
\,.
\]
Integrating over $\theta$ yields the claim. 
\end{proof}

\begin{lemma}\label{lem:compare-fl-dob-integral}
Let $I\subset \cL_{\geq 0}$ be a realization for $\Itop$
and $j\geq 1$ such that $\max_x\overline\hgt_x(I) \leq n-j$. Let $G_{I_{\downarrow}} = (V,E\cup \tilde E)$ denote the induced subgraph of $\Lambda'_{n,n}$ on $V=\Atop^c(I)$,
augmented via edges $\tilde E$ connecting every $v\in\cL_{-1/2}\cap\Lambda'_{n,n}$ 
to an (arbitrarily chosen) external boundary vertex $v_*\in\cL_{<0}\setminus\Lambda_{n,n}'$. Let $\tilde\pi_{I_{\downarrow}}^{(\theta)}$ be the FK model on~$G_{I_{\downarrow}}$ with  probabilities $p=1-e^{-\beta}$ in $E$ and $\theta$ in $\tilde E$, and boundary conditions  wired 
on $\cL_{<0}\cap \partial\Lambda'_{n,n}$ (and free elsewhere).
Then $\Xi_j^\fl(I)$ and $\Xi_j^\dob(I)$ from \eqref{eq:Xi-fl-dob-def} satisfy
    \begin{equation}
        \log \Xi_j^\fl(I) - \log \Xi_j^\dob(I) = \sum_{e \in \tilde E} \int_0^1 \frac{\tilde\pi_{(\Theta_j I)_{\downarrow}}^{(\theta)}(\omega_e = 1) - \tilde\pi^{(\theta)}_{I_\downarrow}(\omega_e = 1)}{\theta(1-\theta)} \,\d\theta\,.
\label{eq:log-Xi-fl-log-Xi-dob-diff}    \end{equation}
\end{lemma}

\begin{proof}
We need to define two counterparts to $\tilde\pi_{I_\downarrow}^{(\theta)}$. First,
let $G_{I_\uparrow}$ be the induced subgraph of $\Lambda'_{n,n}$ on $\Atop(I)$, and define $\pi_{I_\uparrow}$ to be the FK model on $G_{I_\uparrow}$ (all edge probabilities are $p$, i.e., there is no special edge set $\tilde E$) with boundary conditions that are wired on $\cL_{>0}\cap \partial \Lambda'_{n,n}$ (and free elsewhere). Second, let $G$ be the graph $\Lambda'_{n,n}$ augmented by the set of edges $\tilde E$ (so that $G_{I_\downarrow}$ is its induced subgraph on $\Atop^c(I)$), and let  $\tilde \pi_G^{(\theta)}$ be its corresponding FK model, with edge probabilities $\theta$ on $\tilde E$, and boundary conditions as in $\bar\pi_{n}^\dob$, conditioned on no path connecting the $\Top$ boundary $\cL_{>0}\cap\partial\Lambda_{n,n}'$ to the $\Bot$ boundary $\cL_{<0}\cap\partial\Lambda_{n,n}'$.
Notice that $\tilde\pi_{G}^{(\theta)}$ 
interpolates between the boundary conditions $\dob$ and $\fl$:
\begin{equation}\label{eq:tilde-pi-interpolation} \tilde\pi^{(0)}_G = \bar\pi^\dob_{\Lambda'_{n,n}}\quad,\quad \tilde\pi^{(1)}_G (\omega\restriction_{\Lambda_{n,n}}\in\cdot) = \bar\pi^\fl_{\Lambda_{n,n}} \,.\end{equation}
Letting $\tilde Z^{(\theta)}_{G}$, $\tilde Z^{(\theta)}_{I_\downarrow}$ and $Z_{I_{\uparrow}}$ denote the partition functions of $\tilde\pi_{G}^{(\theta)}$, $\tilde\pi_{I_\downarrow}^{(\theta)}$ and $\pi_{I_\uparrow}$, respectively, we infer the following from Domain Markov: 
\[ \tilde\pi_G^{(\theta)}(\Itop = I) = \frac1{\tilde Z_G^{(\theta)}} (1-p)^{|I|} \tilde Z_{I_\downarrow}^{(\theta)} Z_{I_\uparrow}^*\,,\]
where $Z_{I_\uparrow}^*$ restricts the sum in $Z_{I_\uparrow}$ only to configurations $\omega$ where every $v\in \Atop(I)$ adjacent to $I$ is in $\Vtop(\omega)$ (i.e.,
if $v$ is such that, for some $u$, the edge $(u,v)$ of $(\Z+\frac12)^3$ is dual to a plaquette $f\in I$, then there must be an open path in $\omega$ connecting $v$ to the $\Top$ wired boundary $\cL_{>0} \cap\partial \Lambda'_{n,n}$). Indeed, as was the case with the interface $\cI_\Blue$, which consisted of plaquettes dual to  $(\Z+\frac12)^3$-edges between $\Vblue$ and $\Ablue^c$, the interface $\Itop$ examined here consists of plaquettes dual to $(\Z+\frac12)^3$-edges between $\Vtop$ and $\Atop^c$, supporting the above identity.

Comparing the last display for the interfaces $I$ and $\Theta_j I$ (both legal as $\max_x \overline\hgt_x(I) + j \leq n$) we find that
    \begin{equation*}
        \frac{\tilde\pi_G^{(\theta)}(\Itop = \Theta_j I)}{\tilde\pi^{(\theta)}_G(\Itop =  I)} = (1-p)^{4jn}\frac{\tilde Z^{(\theta)}_{(\Theta_j I)_\downarrow} Z^*_{ (\Theta_jI)_\uparrow}}{\tilde Z^{(\theta)}_{I_\downarrow} Z^*_{I_\uparrow}}\,.
    \end{equation*}
Recalling \cref{eq:tilde-pi-interpolation} and the definition of $\Xi_j^\fl$ and $\Xi_j^\dob$, the values of this ratio at $\theta=1$ and $\theta=0$ are nothing but $\Xi_j^\fl(I)$ and $\Xi_j^\dob(I)$ , respectively (for the latter, note that the events $\{\Itop = \Theta_j I\}$ and $\{\Itop = I\}$ are measurable w.r.t.\ $\omega\restriction_{\Lambda_{n,n}}$). As a consequence, 
    \begin{equation}
     \frac{\Xi^\fl_j(I)}{\Xi^\dob_j(I)} = \frac{\tilde Z^{(1)}_{ (\Theta_j I)_\downarrow} Z_{ (\Theta_jI)_\uparrow}^*}{\tilde Z^{(1)}_{I_\downarrow} Z_{I_\uparrow}^*} \cdot\frac{\tilde Z^{(0)}_{I_\downarrow}  Z_{I_\uparrow}^*}{\tilde Z^{(0)}_{(\Theta_jI)_\downarrow} Z_{(\Theta_j I)_\uparrow}^* } = \frac{\tilde Z^{(1)}_{(\Theta_j I)_\downarrow} }{\tilde Z^{(1)}_{ I_\downarrow}}\cdot\frac{\tilde Z^{(0)}_{I_\downarrow}}{\tilde Z^{(0)}_{(\Theta_jI)_\downarrow} }\,.
    \end{equation}
 Therefore,
    \begin{equation}
        \log\Xi^\fl_j(I) - \log \Xi^\dob_j(I)=  \log\frac{\tilde Z^{(1)}_{ (\Theta_j I)_\downarrow}}{\tilde Z^{(0)}_{(\Theta_jI)_\downarrow}}-
        \log\frac{\tilde Z^{(1)}_{I_\downarrow}}{\tilde Z^{(0)}_{I_\downarrow}}\,.
    \end{equation}
    The proof is concluded by applying \cref{clm:free-energy}, once to $G_{(\Theta_j I)_\downarrow}$ and once to $G_{I_\downarrow}$.
\end{proof}

We now conclude \cref{lem:compare-fl-dob}, showing that for the appropriate pairs $(I,j)$, the map $I\mapsto \Theta_j I$ is more costly under the soft floor measure than the hard floor measure.

\begin{proof}[\textbf{\emph{Proof of \cref{lem:compare-fl-dob}}}]
Observe that the assumption on $(I,j)$ implies that $\max_x\overline\hgt_x(I) + j < 2j\leq n$, qualifying for an application of \cref{lem:compare-fl-dob-integral}.
It thus suffices to show that, for such $I$ and~$j$,
\begin{equation}\label{eq:tilde-pi-monotone-on-tilde-E}
 \tilde\pi^{(\theta)}_{(\Theta_j I)_\downarrow}(\omega_e=1)\geq  \tilde\pi^{(\theta)}_{I_\downarrow} (\omega_e=1) \end{equation}
 holds for all $e\in\tilde E$ (implying that the right-hand of \cref{eq:log-Xi-fl-log-Xi-dob-diff} is nonnegative).
 We claim that \cref{eq:tilde-pi-monotone-on-tilde-E} in fact holds for every edge $e$ in $G_{I_\downarrow}$.  To see this, consider $G_{(\Theta_j I)_\downarrow}$ and the associated FK measure $\tilde \pi_{(\Theta_j I)_\downarrow}^{(\theta)}$, where we additionally impose that every edge $(u,v)$ dual to some $f\in I$ should be closed. The hypothesis on $j>\max_x \overline{\hgt}(I)$ implies that this set of edges is fully included in $G_{(\Theta_j I)_{\downarrow}}$, and by the Domain Markov property, its marginal on $G_{I_\downarrow}$ has the same law of $\tilde\pi^{(\theta)}_{I_\downarrow}$. The proof is then concluded by the FKG inequality. 
\end{proof}

\begin{remark}
Our proof crucially uses the fact that the event $\Itop = I$ is measurable w.r.t.\ the configuration ``above $I$'' (on $G_{I_\uparrow}$ in the notation from the proofs above), whereas it is only ``below $I$'' (on $G_{I_\downarrow}$) where there is a difference between the hard floor and soft floor measures, as $I\subset  \cL_{\geq 0}$ in both. Thus, the nontrivial portion (``above $I$'') cancels in the ratio $\Xi_j^\fl / \Xi_j^\dob$, and leaves us with a ratio of probabilities in two standard FK models. At this point we may use FKG as $\Itop$ reveals a free boundary below it, c.f.\ $\cI_{\Full}$ or $\cI_{\Bot}$ which would reveal a wired boundary below it, and therefore not go in the right direction. 
(Indeed, the result cannot hold for $\Ibot$, as it would lead to a lower bound of $(1+\epsilon_\beta)h_n^*$ in \cref{thm:main},
violating the upper bound $h_n^*$ given there.)

Similarly, this argument would not apply to $\cI_\Blue$ or $\cI_\Red$; e.g., for the Ising model, each of the events $\cI_\Blue=I$ and $\cI_\Red = I$ gives positive information on $G_{I_\downarrow}$ (and negative information on $G_{I_\uparrow}$), whence the inequality analogous to \cref{eq:tilde-pi-monotone-on-tilde-E} would be reversed.
\end{remark}

\subsection{Lower bound for the top interface}\label{subsec:log-lower-bound}
Our aim is now to show the following logarithmic lower bound, in fact $(1-\epsilon) h_n^*$, on the typical height of $\cI_{\Top}$. 

\begin{proposition}\label{prop:Itop-lower}
    There exists an absolute constant $C>0$ and constant $\epsilon_\beta$ such that
    \[\bar\pi_n^\fl( |\{x\in (\mathbb Z+\tfrac{1}{2})^2  \,:\; \underline \hgt_x(\Itop) \le \tfrac{1}{4\beta + C}\log n\}| >\epsilon_\beta n^2 ) \le n^{-1+\epsilon_\beta}\,.
    \]
\end{proposition}

\subsubsection{Map to lift $\Itop$}\label{subsec:shift-map}
We wish to consider the probability of a realization of $\Itop$ versus its shift up by some height. The previous subsection established that this map can be done under the measure $\bar \pi_n^{\mathsf{dob}}$ rather than $\bar \pi_n^{\fl}$ for the sake of lower bounds on $\Itop$.

\begin{lemma}\label{lem:lift-Itop}
Let $I\subset \cL_{\geq 0}$ be a realization for $\Itop$, and $j\geq 1$. Then $\Xi_j^\dob(I)$ from \cref{eq:Xi-fl-dob-def} has
    \[
\Xi_j^\dob(I)     \geq 
        e^{-4(\beta + C)jn} \,.
    \]
\end{lemma}

\begin{proof}
It suffices to prove the lemma for $j=1$, that is, to show that
$\bar\pi_n^\dob(\Itop = I)/\bar\pi_n^\dob(\Itop = \Theta_1 I) $ is at most $ e^{4(\beta+C)n}$
(which, when iterated $j$ times, would yield for the sought bound on $\Xi_j^\dob(I)$).
    Let $\cF(I)$ be the set of $\Full$ interfaces which have $I$ as the $\Top$ interface; that is,
    \[ \cF(I) = \{ I_0 = \cI_\Full(\omega)\,:\; \omega\mbox{ satisfies }\Itop(\omega)=I\}\,.\]

\begin{claim}
Let $\omega\sim \bar\pi_n^\dob$. Then $\Itop(\omega)$ is measurable w.r.t.\ $\cI_\Full(\omega)$. 
\end{claim}
\begin{proof}
Let $\omega_0$ be the configuration such that $(\omega_0)_e = 0$ if and only if $e$ is dual to a plaquette $f \in I_0$. It suffices to show that for any $\omega$ with $\cI_\Full(\omega) = I_0$, we have $\Atop(\omega) = \Atop(\omega_0)$, as that would imply $\Itop(\omega) = \Itop(\omega_0)$. Since $\omega$ must be pointwise below $\omega_0$, we have immediately that $\Vtop(\omega) \subset  \Vtop(\omega_0)$, and hence $\Atop(\omega) \subset  \Atop(\omega_0)$. To show the other inclusion, suppose for contradiction that $\Atop(\omega_0) \setminus \Atop(\omega)$ is nonempty. Then, there must in particular be some $v \in \Atop(\omega_0) \setminus \Atop(\omega)$ such that $v$ is adjacent to some $u \in \Atop(\omega_0)^c$ (if no such $v$ exists, then $\Atop(\omega_0) \setminus \Atop(\omega)$ must consist of finite components surrounded by vertices of $\Atop(\omega_0) \cap \Atop(\omega)$, but this contradicts the construction of $\Atop(\omega)$ which would have included such finite components). Let $C(v)$ be the finite random-cluster component containing $v$. Then $C(v)$ has a 1-connected boundary of dual-to-closed plaquettes which contains the plaquette dual to the edge $(u,v)$, hence the entire 1-connected boundary is a subset of $\cI_\Full$. Thus, $v$ is in a finite component even in $\omega_0$, which together with the fact that $v$ is adjacent to $u \in \Atop(\omega_0)^c$ implies that $v \in \Atop(\omega_0)^c$, concluding the contradiction.    
\end{proof}
 As a consequence of the above claim, we can write 
    \begin{equation}\label{eq:Itop=I-in-terms-of-Full}
        \bar \pi_n^\dob(\Itop = I) = \bar \pi_n^\dob(\cI_\Full \in \cF(I))\,.
    \end{equation}
 We claim that it will therefore suffice to show that 
    there is a variant of the lifting map $\Theta_1$ for $\Full$ interfaces, call it $\Phi_1$, such that for every $I$ and every $I_0\in\cF(I)$, we have 
 \begin{align}
    \bar\pi_n^\dob(\cI_\Full = I_0) &\leq e^{4(\beta + C)n} \bar \pi_n^\dob(\cI_\Full = \Phi_1 I_0) \label{eq:nts-Itop-1}\,,
\end{align}
and 
\begin{align}
\sum_{I_0\in \cF(I)} \bar\pi_n^\dob(\cI_\Full = \Phi_1 I_0) &\leq e^{C n} \bar\pi_n^\dob(\cI_\Full \in \cF(\Theta_1 I))  \label{eq:nts-Itop-2}\,.
\end{align}
Indeed, modulo these two inequalities, we can infer from \cref{eq:Itop=I-in-terms-of-Full} that
   \begin{align*}
        \bar \pi_n^\dob(\cI_\Full \in \cF(I)) &= \sum_{I_0 \in \cF(I)} \bar\pi_n^\dob(\cI_\Full = I_0) \leq \sum_{I_0 \in \cF(I)} e^{4(\beta + C)n} \bar \pi_n^\dob(\cI_\Full = \Phi_1 I_0) \\
        & \leq e^{4(\beta + 2C)n} \bar \pi_n^\dob(\cI_\Full \in \cF(\Theta_1 I)) = e^{4(\beta + 2C)n}\bar \pi_n^\dob (\Itop = \Theta_1 I)\,.
    \end{align*}
(The inequality in the first line used \cref{eq:nts-Itop-1} and the one in the second line used \cref{eq:nts-Itop-2}.)
\Cref{eq:nts-Itop-1} controls the change in probability incurred by applying $\Phi_1$ (the ``lifting up by 1'' map) to the $\Full$ interface, where we have the cluster expansion of \cref{prop:grimmett-cluster-exp} at our disposal. 
\Cref{eq:nts-Itop-2} relates the probability of having a $\Full$ interface that is lifted by $1$ to that of having a $\Top$ interface that is shifted by $1$ (more precisely, an union bound on the former probability: the left-hand is an upper bound on $\bar\pi_n^\dob(\cI_\Full\in\Phi_1 \cF(I))$ via the non-disjoint events $\{\cI_\Full=\Phi_1 I_0\}$). (Note that, unlike the shift map $\Theta_1$ that acts on (a realization of) $\Itop$, working with a $\Full$ interface must be done with extra care, as it is not a surface---it consists of additional ``hairs'' of dual-closed plaquettes---the culprit behind the fact that the shift map $\Phi_1$ acting on $\cI_\Full$ is not a bijection.)

Let $\Phi_1$ be the map on a $\Full$ interface $I_0$ such that $\Phi_1(I_0) = J_0$, where the plaquette set of $J_0$ is the union of the set of plaquettes in $I_0$ shifted up by 1, together with the set of $4n$ vertical plaquettes at height $1/2$ which are adjacent to the boundary vertices in $\partial \Lambda_{n,n}$. Recall that in the context of $\bar \pi_n^\dob$, we are using $\partial_\Red \Lambda'_{n,n} = \cL_{>0}\cap\partial \Lambda'_{n,n}$ and $\partial_\Blue \Lambda'_{n,n} = \cL_{<0}\cap\partial \Lambda'_{n,n}$. 

\begin{claim}\label{clm:shift-map-well-defined}
    For every realization $I_0$ for $\cI_\Full$ under $\bar\pi_n^\dob$, the map $\Phi_1$ outputs a valid full interface $J_0=\Phi_1(I_0) $ separating $\partial_\Blue\Lambda'_{n,n}$ from $\partial_\Red\Lambda'_{n,n}$.
\end{claim}
    \begin{proof}
        We need to show that $J_0$ is 1-connected and separates $\partial_\Blue\Lambda'_{n, n}$ from $\partial_\Red\Lambda'_{n, n}$. 
        
        The fact that $J_0$ is 1-connected follows from observing that the translation of $I_0$ up by height 1 is 1-connected, the $4n$ vertical plaquettes at height 1/2 adjacent to boundary vertices in $\partial\Lambda_{n, n}$ are 1-connected, and these two sets are also 1-adjacent.
        
        To show that $J_0$ is separating, 
        suppose for contradiction that there is a path $P$ from $x \in \partial_\Red \Lambda'_{n,n}$ to $y \in \partial_\Blue \Lambda'_{n,n}$ that does not cross any plaquette of $J_0$. Let us first consider the case where $x\in \cL_{\geq 3/2}$. Then we can consider the translation $P'$ of the path $P$ by height $-1$. Since $P$ does not cross any plaquette of~$J_0$, the path $P'$ does not cross any plaquette of~$I_0$; so, the fact $x - \ez \in \partial_\Red \Lambda'_{n,n}$ implies that~$P'$ is a path from $\partial_\Red \Lambda'_{n,n}$ to $\partial_\Blue \Lambda'_{n,n}$ that does not cross a plaquette of $I_0$, contradicting the fact that $I_0$ is a valid interface. For the remaining case $x\in \cL_{1/2}\cap\partial\Lambda'_{n,n} =: X$, the new (rather than shifted) plaquettes added to $J_0$ (of which there are at most $4n$) ensure that any path in $\Lambda'_{n,n}$ from $X$ to $X^c$ must first pass through a vertex  $x'\in \cL_{3/2}\cap \partial \Lambda'_{n,n}$, putting us back in the first case.
    \end{proof}

\begin{claim}\label{clm:shift-map-energy}
    We have the following bounds for the terms in the cluster expansion:
    \begin{align*}
        &|J_0| - |I_0| \leq 4n\,,\qquad
        |\partial J_0| - |\partial I_0| \leq Cn\,,\qquad
        \kappa(J_0) \geq \kappa(I_0)\,,\\
        &\bigg|\sum_{f \in J_0} \g(f, J_0) - \sum_{f \in I_0} \g(f, I_0)\bigg| \leq Cn\,.
    \end{align*}
    Hence, we have by ~\cref{prop:grimmett-cluster-exp} for the $\Full$ interface that 
    \begin{equation*}
    \bar\pi_n^\dob(I_0) \leq e^{4(\beta + C)n}\bar\pi_n^\dob(J_0)\,.
    \end{equation*}
\end{claim}
\begin{proof}
    The first inequality $|J_0| - |I_0| \leq 4n$ follows by construction of $J_0$. The second inequality similarly follows because the only difference lies in adding the up to $4n$ plaquettes to $J_0$, and each plaquette added can only change $\partial J_0$ by a constant. Adding plaquettes (equivalent to closing edges) can only increase the number of open clusters, so $\kappa(J_0) \geq \kappa(I_0)$. Finally, call $I_1$ the set of plaquettes of $I_0$ shifted up by $1$, and for each $f \in I_0$ let $f' \in I_1$ be the copy of $f$ shifted up by $1$. We then have
    \begin{equation*}
        \bigg|\sum_{f \in J_0} \g(f, J_0) - \sum_{f \in I_0} \g(f, I_0)\bigg| \leq \sum_{f \in J_0 \setminus I_1} |\g(f, J_0)| + \sum_{f: f' \in J_0 \cap I_1} |\g(f', J_0) - \g(f, I_0)|\,.
    \end{equation*}
    The first term is bounded by $4Kn$, and we can bound the second term since the radius $r(f, I_0; f', J_0)$ is obtained by the distance to either one of the up to $4n$ plaquettes in $J_0\setminus I_1$ or one of the $4n$ horizontal boundary plaquettes  at height $0$ in the Dobrushin boundary conditions. If the union of these two plaquette sets is called $F$, then we have
    \begin{equation*}
        \sum_{f: f' \in J_0 \cap I_1} |\g(f', J_0) - \g(f, I_0)| \leq \sum_{f' \in J_0 \cap I_1}\sum_{g \in F} Ke^{-cd(f', g)} \leq K|F| \leq K'n\,. \qedhere
    \end{equation*}
\end{proof}
Combining the above two claims proves \cref{eq:nts-Itop-1}.
For \cref{eq:nts-Itop-2}, we need the following claim.

\begin{claim}\label{clm:shift-map-entropy}
    Every $J_0$ in the image of the map $\Phi_1$ has at most $e^{Cn}$ preimages~$I_0$.
\end{claim}
\begin{proof} 
Denote by $I_1$, as in the previous proof, the plaquettes of $I_0$ shifted up by $1$, and let us decompose $J_0 = I_1 \cup (J_0 \setminus I_1)$. If we know which of the $4n$ vertical plaquettes adjacent to boundary vertices at height $1/2$ are in $I_1$ and which are in $J_0 \setminus I_1$, then we can recover $I_1$ and hence $I_0$ from the image. There are at most $2^{4n}$ subsets of possible plaquettes to attribute to $J_0 \setminus I_1$, as claimed.
\end{proof}

We will now argue that
\begin{align}
     \Big\{\cI_\Full \in  \bigcup_{I_0\in \cF(I)}\Phi_1 I_0\Big\} &\subset  \{\cI_\Full \in \cF(\Theta_1 I)\}  \label{eq:nts-Itop-2b}\,,
\end{align}
which together with \cref{clm:shift-map-entropy} will establish \cref{eq:nts-Itop-2}.

Let $\omega_0$ be the configuration that only has closed edges $e$ dual to plaquettes $f \in I_0$. Similarly, let $\omega_1$ be the configuration that only has closed edges $e$ dual to plaquettes $f \in \Phi_1 I_0$. We need to show that 
$\Vtop(\omega_1)$ is nothing but $\Theta_1 \Vtop(\omega_0)$ (the set of vertices in $\Vtop(\omega_0)$ shifted up by $1$). 

Notice that if we take $\omega_0$, shift all edges up by $1$, and then close the $4n$ edges connecting a boundary vertex in $\partial_\Red \Lambda'_{n,n}$ at height $1/2$ to a vertex in $\Lambda'_{n,n} \setminus \partial_\Red \Lambda'_{n,n}$ also at height $1/2$, then we exactly get~$\omega_1$. Call~$\omega'_1$ the configuration obtained by shifting the edges of $\omega_0$ up by $1$ (without closing extra edges). 

First we show that $\Theta_1 \Vtop(\omega_0) \subset  \Vtop(\omega_1)$. Indeed, note that for every path $P$ of open edges in $\omega_0$ connecting a vertex $x$ to $\partial_\Red \Lambda'_{n,n}$, its shift up by $1$, call it $P'$, will be a path of open edges in $\omega'_1$ connecting $x + \ez$ to $\partial_\Red \Lambda'_{n,n}$. It remains to show that closing the up to $4n$ edges to obtain $\omega_1$ does not cut off this path. Indeed, if one of those edges were in the path $P'$, this would imply that $P'$ contains a boundary vertex in $\cL_{1/2}$; thus, $P$ would contain a boundary vertex in $\cL_{-1/2}$, contradicting the fact that there is no path of open edges from $\partial_\Blue \Lambda'_{n,n}$ to $\partial_\Red \Lambda'_{n,n}$.

Now we show that $\Vtop(\omega_1) \subset  \Theta_1 \Vtop(\omega_0)$. If $x \in \Vtop(\omega_1)$ and $P'$ is a path of open edges in $\omega_1$ connecting $x$ to a boundary vertex $y \in \partial_\Red \Lambda'_{n,n}$, then it must have such a path to a $y$ at height at least $3/2$ (as all $4n$ faces along $\cL_{3/2}\cap\partial\Lambda'_{n,n}$ have been made dual-to-closed). Then $P'$ is still an open path in $\omega'_1$ (since $\omega'_1 \geq \omega_1$), and its shift down by $1$, call it $P$, is an open path in $\omega_0$. Thus $P$ is an open path connecting $x - \ez$ to $y - \ez$, and $y - \ez \in \partial_\Red \Lambda'_{n,n}$ since $y -  \ez\in \cL_{\geq 1/2}$. 

This establishes \cref{eq:nts-Itop-2b}, thereby concluding the proof of the lemma.
\end{proof}

\subsubsection{Injecting entropy via downward spikes}\label{subsec:spikes}
Now that we have controlled the weight change of lifting $\Itop$ under $\bar \pi_{n}^\fl$ via its weight change under $\bar \pi_n^\dob$, we can return to the hard-floor measure and inject entropic downward spikes to argue that this lifting and the added opportunity for downward spikes to the interface is preferential to having predominantly been at too low a height.  

\begin{proof}[\textbf{\emph{Proof of \cref{prop:Itop-lower}}}]
    Let $H:= \tfrac{14}\beta \log n$ and let $h:= \frac{1}{4\beta + C_0}\log n$ for $C_0$ to be chosen later. For $\epsilon_\beta$, define the set $\mathbf{I}_{\mathsf{bad}}$ to be the set of possible realizations $I$ of $\Itop$ satisfying
    \begin{align*}
        |\{x\in (\mathbb Z+\tfrac{1}{2})^2 \,:\; \underline \hgt_x(I) \le h\}| >\epsilon_\beta n^2 \qquad \text{ and } \qquad \max_x \overline\hgt_x(I) \le H.
    \end{align*}
    By a union bound with \cref{thm:weak-upper-fl-top}, it will suffice for us to show 
        \begin{align}\label{eq:nts-Itop-lower-bound}\bar\pi_n^\fl(\Itop \in  \mathbf{I}_{\mathsf{bad}} ) \leq e^{-n}\,. \end{align}
    Fix any $I\in \mathbf{I}_{\mathsf{bad}}$.  
    Consider the operation $\Theta_H$ which lifts $I$ by $H$ and adds all $4Hn$ vertical plaquettes along the boundary at heights between $0$ and $H$. By \cref{lem:compare-fl-dob} and \cref{lem:lift-Itop}, we have 
    \begin{align}\label{eq:Itop-lifting-cost-fl}
        \frac{\bar \pi_n^\fl(\cI_{\Top}  = \Theta_{H} I)}{\bar \pi_n^\fl(\cI_{\Top} = I)} \ge e^{ - 4(\beta + C) H n}\,.
    \end{align}
    Next, for a subset $A$ of $xy$-plane index points $A \subset (\mathbb Z+\frac{1}{2})^2 \times \{0\}$ such that $\underline{\hgt}_x (I) \le h$, consider the following operation $\Phi_{A}$ on $\{\omega: \cI_{\Top} = \Theta_H I\}$. 
    Obtain $\Phi_A(\omega)$ by,  for every $x\in A$,
    \begin{enumerate}
        \item opening the $h$ vertical edges (dual to horizontal plaquettes) between vertices in the column $$v_i^x = x+ (\underline{\hgt}_x(\Theta_H I) + \tfrac{1}{2} - i)\ez\,, \qquad i=0,\ldots,h\,,$$
        (if they are not open already). 
        \item closing the $4h$ horizontal edges (dual to vertical plaquettes) incident to vertices $(v_i^x)_{i=1,\ldots,h}$, as well as the one vertical edge (dual to a horizontal plaquette) between $v_{h}^x$ and $v_{h}^x - \ez$ (if they are not closed already). 
    \end{enumerate}
    In $\omega$, all vertices $v_0^x$ are in $\Vtop$, and so step (1) above is exactly adding a downward vertical spike of height $h$ to $\Vtop$ at every column above $x\in A$. The resulting $\omega$ does not violate the disconnection event $\mathfrak{D}_{n,n}$ because every vertex added to $\Vtop$ by step (1) is disconnected from $\partial_\Blue \Lambda_{n,n}$ by the union of the original $\Theta_H I$ and the edges closed in step (2). By direct comparison of weights, notice that for every $\omega$, we have 
    \begin{align}\label{eq:top-spikes-weight-change}
        \frac{\bar \pi_n^\fl(\Phi_A(\omega))}{\bar \pi_n^\fl(\omega)} \ge \Big(\Big(\frac{1-p}{p}\Big)^{4h+1} q^{4h+1}\Big)^{|A|}\,.
    \end{align}
    From a realization of $\Phi_A(\omega)$, we can read off the set $A$ by first reading off the set $\Itop(\Phi_A(\omega))$, then finding all index points $x\in \cL_{0}$ above which the interface $\Itop(\Phi_A(\omega))$ takes a height below $H$. Indeed, every such $x$ is a member of $A$ because only vertices in columns through $A$ are added to $\Vtop$ by the operation, and all members of the original set $\Vtop$ were above height $H$ in $\Theta_H I$. Moreover, every member of $A$ is counted in this set of $x$ because $\underline{\hgt}_x(\Theta_H I) - h= H + \underline{\hgt}_x(I) -h$ is below $H$. As a result, for every $\omega'$ in the image of $\Phi_A$ (there can only one $A$ for which this is the case since $A$ is determined by $\omega'$), 
    \begin{align*}
        |\{\omega: \Phi_A(\omega) = \omega'\}|\le 2^{(5h+1) |A|}\,,
    \end{align*}
    by reading off $A$, and $\underline{\hgt}_x(\Itop(\omega'))$ for every $x\in A$ from $\omega'$ as above, reading off from that $v_i^x$ for all $x\in A$, then enumerating over the $2^{5h+1}$ possible values $\omega$ could take on the incident edges for the $h$ vertices $(v_i^x)_{i=1}^h$. From any such $\omega$, we can also read off $I$ (as compared to $\Theta_H I$) because the map $\Theta_H$ is a bijection for $I\subset \cL_{\ge 0}$. 
    
    Putting the above together, we therefore have that 
    \begin{align}
       1 & \ge \sum_{I \in \mathbf{I}_{\mathsf{bad}}} \,\sum_{A\subset \{x: \underline\hgt_x I \le h\}} \,\sum_{\omega'\in \Phi_A(\{\omega: \Itop(\omega) = \Theta_H I\})} \bar \pi_n^\fl(\omega') \nonumber \\ 
       & \ge    \sum_{I \in \mathbf{I}_{\mathsf{bad}}} \,\sum_{\omega: \Itop(\omega)=  \Theta_H I} \, \sum_{A\subset \{x: \underline\hgt_x I \le h\}} 2^{-(5h+1) |A|} e^{ - (4\beta +C) h |A|} \bar \pi_n^\fl(\omega) \nonumber 
    \end{align}
where the right-hand side of~\eqref{eq:top-spikes-weight-change} became $e^{ -(4\beta + C) h|A|}$. Using that for every $I\in \mathbf{I}_{\mathsf{bad}}$, one has $|\{x: \underline\hgt_x I\le h\}|>\epsilon_\beta n^2 $, then up to a change of the constant $C$, one has 
    \begin{align*}
              1 & \ge \sum_{I\in \mathbf{I}_{\mathsf{bad}}} \bar \pi_n^\fl(\Itop = \Theta_H I) \Big(1+ e^{ - (4\beta +C) h}\Big)^{|\{x: \underline\hgt_x I\le h\}|}  \ge \sum_{I\in \mathbf{I}_{\mathsf{bad}}} \bar \pi_n^\fl(\Itop = \Theta_H I) \Big(1+ e^{ - (4\beta +C) h}\Big)^{\epsilon_\beta n^2}\,.  
    \end{align*}
    By~\eqref{eq:Itop-lifting-cost-fl}, the above becomes 
    \begin{align*}
        1 & \ge \Big(1+ e^{ - (4\beta +C) h}\Big)^{\epsilon_\beta n^2} \sum_{I\in \mathbf{I}_{\mathsf{bad}}} e^{ - (4\beta + C) Hn} \bar \pi_n^\fl(\Itop= I) \\ 
        & = \Big(1+ e^{ - (4\beta +C) h}\Big)^{\epsilon_\beta n^2} e^{ - (4\beta + C) Hn} \bar \pi_n^\fl(\Itop \in \mathbf{I}_{\mathsf{bad}})\,.
    \end{align*}
    Using $1+x \ge e^{x/2}$ for $x\le 1$ and rearranging, we get 
    \begin{align*}
        \bar \pi_n^\fl(\Itop \in \mathbf{I}_{\mathsf{bad}}) \le \exp\Big( - \tfrac{\epsilon_\beta}{2} n^2 e^{ - (4\beta + C)h} + (4\beta + C) H n\Big)\,.
    \end{align*}
    By our choices of $h$ and $n$, it becomes clear that there is a choice of $C_0$, say $C_0 = 2C$, such that the negative term in the exponential is larger than $ n^{1 + \epsilon_\beta'}$ (regardless of $\epsilon_\beta$), while the positive term is $O(n\log n)$, implying the requisite bound of~\eqref{eq:nts-Itop-lower-bound}. 
\end{proof}

\begin{remark}\label{rem:lower-bound-rate}
The final step in the proof of \cref{prop:Itop-lower} was to derive entropic repulsion via comparing the energetic cost of the shift map with the entropy gain from planting spikes consisting of straight columns of height $h$. This part of the argument is suboptimal since $\Itop$ is in fact propelled by the large deviation rate of its downward deviations, which are typically random-walk like as opposed to straight columns. We expect that after lifting $\cI_{\Top}$ by $C \log n$ (for a large enough $C>0$) using \cref{lem:lift-Itop} so that the hard floor is far from the interface, the detailed analysis in \cite{GL-entropic-repulsion}, which showed entropic repulsion driven by pillar rates for the Ising interface above a soft floor, can be adapted to the interface $\Itop$. This would replace the lower bound $\frac{1}{4\beta+C_0}\log n$ in \cref{prop:Itop-lower} (and consequently the one in \cref{thm:main}) by $(1-o(1))\frac1{\xi'}\log n$, where $\xi'$ is the analog of $\xi$ in $\pi_\infty^{\mathsf{w}}$ (the rate of point-to-plane connections in the complement of the infinite cluster).
\end{remark}
\subsection{Extending the lower bound to the blue interface}

We would now like to extend to $\cI_\Blue$ the lower bound of $\frac1{4\beta+C}\log n$ obtained in \cref{prop:Itop-lower} on the height of most plaquettes in $\Itop$. Note that, to that end, the additive error from \cref{cor:dist-Ifull-Ibot} (whereby $\cI_\Blue$ is sandwiched between $\Ibot,\Itop$) would be far too crude; instead, we will rely on the following result.

\begin{lemma}\label{lem:Ifull-wall-faces} There is an absolute constant $C>0$ such that, if $\delta = C/\beta$, then
\[ \bar\pi_n^\fl\left(|\cI_\Full|\geq (1+\delta)n^2\right) \leq e^{ - n^2}\,.\]
In particular, on this event, $\underline\hgt_x(\cI_\Full) = \overline\hgt_x(\cI_\Full)$ for all but at most $\delta n^2$ plaquettes $x\in (\Z+\frac12)^2$, and under $\bP_n^\fl$ all four interfaces $\cI_\Bot,\cI_\Blue,\cI_\Red,\cI_\Top$ coincide with $\cI_\Full$ on the corresponding plaquette.
\end{lemma}
\begin{proof}
Let $I_0$ be the ``trivial interface,'' consisting of the $n^2$ horizontal plaquettes of $\cL_0 \cap \Lambda_{n,n}$. As mentioned in \cref{rem:diff-domains}, the cluster expansion of \cref{prop:grimmett-cluster-exp} for $\cI_{\Full}$ is valid also for $\bar\pi_n^\fl$ (taking $m_1 = 0$ and $m_2 = n$ to arrive at $\Lambda_{n,n}$, and $\sfh=0$ for the appropriate boundary conditions).
Using this (namely, \cref{eq:CE}), if $I$ is  realization of a $\Full$ interface with $n^2 + s$ plaquettes, then
\[ \frac{\bar\pi_n^\fl(\cI_\Full = I)}{\bar\pi_n^\fl(\cI_\Full = I_0)} \leq \Big(\frac{(1-p) q}{p^C}\Big)^s e^{K (2n^2+s)} \leq e^{C n^2 - \beta s}\,,\]
with $K>0$ the absolute constant from \cref{prop:grimmett-cluster-exp} and $C>0$ another absolute constant. Note that, unlike the more careful analysis of the map $\Phi_1$ in \cref{subsec:shift-map} (which would correspond to bounding  $|\sum_{f\in I}\g(f,I) - \sum_{f'\in I_0} \g(f',I_0)|$ via absorbing $K$ for each $f\in I\oplus I_0$, and another constant contribution for each such $f$ from summing over the exponential decay of its effect on other plaquettes through the $\g$ term) we resorted to the naive bound $(|I|+|I_0|)\|\g\|_\infty \leq (2+s)Kn^2$; this is due to $I_0$ being adjacent to the boundary of the box, whereby $r(f,I;f',I_0)$ is $1$ for all $f'\in I_0$ (demonstrating some of the difficulty in overcoming the potential pinning effect of the hard floor).

There are at most $e^{C (n^2+s)}$ interfaces $I$ with $n^2+s$ plaquettes, for some other absolute $C>0$ (viewed as a rooted subgraph of size $n^2+s$ in a bounded-degree graph). Thus,
\[ \bar\pi_n^\fl(|\cI_\Full|\geq (1+\delta)n^2)\leq \sum_{s\geq \delta n^2}\,\sum_{I: |I|=n^2+s} \frac{\bar\pi_n^\fl(\cI_\Full = I)}
{\bar\pi_n^\fl(\cI_\Full = I_0)} \leq e^{ C' n^2-\beta \delta n^2}\,,
\]
and taking $\delta = (C'+1)/\beta$ concludes the proof.
\end{proof}

\begin{proof}[\textbf{\emph{Proof of lower bound in Theorem~\ref{thm:main}}}]
By \cref{lem:Ifull-wall-faces}, we can assume that all but at most $\delta n^2 = C n^2/\beta$ plaquettes of $\cI_\Full$ are horizontal plaquettes with no other plaquettes of $\cI_\Full$ sharing the same $xy$-coordinates (simply by the fact that every $x \in (\Z+\tfrac12)^2 \times {0}$ needs to have at least one plaquette above it, and this already accounts for $n^2$ plaquettes). In particular, each of these plaquettes are common to $\cI_\Full$, $\cI_\Blue$ and $\cI_\Top$ and are one of the minimal height plaquettes of $\cI_\Top$ studied in \cref{prop:Itop-lower}. Hence, on the event $|\cI_\Full| < (1+\delta)n^2$, we have
\begin{align*}
    \big\{|\cI_{\Blue} \cap \cL_{\le \frac{1}{4\beta + C}\log n}| >\epsilon_\beta' n^2\big\} \subset  \big\{|\{x\in (\mathbb Z+\tfrac{1}{2})^2 \, :\; \underline \hgt_x(\Itop) \le h\}| >\epsilon_\beta' n^2 - \delta n^2\big\} \,,
\end{align*}
which concludes the proof via \cref{prop:Itop-lower} so long as $\epsilon_\beta' - \delta$ is bigger than the $\epsilon_\beta$ there, which is of course attained via $\epsilon_\beta' = \epsilon_\beta + C/\beta$. This proves that the lower bound of \cref{thm:main} holds with probability $1 - n^{-1 + \epsilon_\beta}$.   
\end{proof}

\section{Sharp upper bound}\label{sec:sharp-upper}
In this section, we prove the sharp upper bound of $h_n^*$~\eqref{eq:h_n^*} on the typical interface height in Theorem~\ref{thm:main}. By the monotonicity argument of~\cref{prop:Potts-monotonicity} via the fuzzy Potts model, it essentially suffices to establish such an upper bound on the typical interface height under the ``soft-floor" Potts measure, i.e., conditioned on $\cI_{\Blue} \subset \cL_{\ge 0}$, and with the Dobrushin boundary conditions going from $\Red$ to $\Blue$ at height $\sfh = h_n^*$. Namely,  
we will show that under $\musoft^{h_n^*}_n=\mu_{n}^{h_n^*}( \cdot \mid \cI_{\Blue}\subset \cL_{\ge 0})$ as per \cref{eq:mu-hat-measure},
 the interface has fewer than $\epsilon_\beta n^2$ many plaquettes above $h_n^*$. 

The proof of this will go in two parts, by lower bounding the probability of the event $\{\cI_{\Blue}\subset\cL_{\ge 0}\}$ under $\mu_{n}^{h_n^*}$, and at the same time, upper bounding under $\mu_{n}^{h_n^*}$ (which is rigid about height $h_n^*$) the large deviation probability of its interface having more than $\epsilon_\beta n^2$ plaquettes at heights $\ge h_n^* + 1$. The ratio of these will bound the probability of the latter event under the conditional measure $\musoft_n^{h_n^*}$. 

\subsection{Remark on finite box vs.\ cylinder}
As a technicality it turns out to be convenient to work with the Potts interfaces in infinite and semi-infinite cylinders (keeping the $xy$-side lengths $n$). The following lemma is standard (similar to e.g., \cite[Lem.~8]{GielisGrimmett02},\cite[Lem~7.117]{Grimmett_RC}) and shows that these induce essentially the same measure on interfaces. Recall $\Lambda_{n,\ell}$ and $\Lambda_{n,\ell}'$ from~\eqref{eq:floor-domain}--\eqref{eq:dob-domain}. 

\begin{lemma}\label{lem:finite-vs-infinite-cylinder}
    Consider the Potts measures $\mu_{\Lambda_{n,\ell}'}^\sfh$ and $\mu_{\Lambda_{n,m}'}^\sfh$ for $n\le \ell \le m$ and with $|\sfh|\le \ell/2$. Then, 
    \begin{align*}
        \|\mu_{\Lambda_{n,\ell}'}^\sfh (\cI_{\Blue}\in \cdot) - \mu_{\Lambda_{n,m}'}^\sfh (\cI_\Blue \in \cdot) \|_{\tv} \le \exp(- n)\,.
    \end{align*}
    Similarly, if we consider $\mu_{\Lambda_{n,\ell}}^\fl, \mu_{\Lambda_{n,m}}^\fl$ with $\ell \le m$, then 
    \begin{align*}
                \|\mu_{\Lambda_{n,\ell}}^\fl(\cI_{\Blue}\in \cdot) -\mu_{\Lambda_{n,m}}^\fl (\cI_{\Blue}\in \cdot)\|_{\tv}\le \exp( - n)\,.
    \end{align*}
\end{lemma}

\begin{proof}
    Sample random-cluster and Potts configurations $(\omega, \sigma) \sim \bP_{\Lambda_{n,\ell}'}^\sfh$ and $(\omega',\sigma')\sim \bP_{\Lambda_{n,m}'}^\sfh$ independently. We say a set $\mathcal S$ of edges is a \emph{separating surface in $\omega$ between $\cL_{3\ell/4}$ and $\cL_\ell$} if $\omega(\mathcal S) \equiv 1$ and every 1-connected path of plaquettes in $\Lambda_{n,\ell}$ between $\cL_{3\ell/4}$ and $\cL_\ell$ must have a plaquette dual to some edge of $\mathcal S$. 
    If $\omega\wedge \omega'$ is the pointwise minimal configuration, define the event $\mathcal G$ as 
    \begin{enumerate}
        \item there exists a separating set $\mathcal S_{N}$ in $\omega \wedge\omega'$ between $\cL_{3\ell/4}$ and $\cL_{\ell}$, and 
        \item there exists a separating set $\mathcal S_{S}$ in $\omega\wedge \omega'$ between $\cL_{-\ell}$ and $\cL_{-3\ell/4}$. 
    \end{enumerate}

    On the event $\mathcal G$, expose the highest realizing $\mathcal S_N$ and the lowest realizing $\mathcal S_S$. Even though there is no order on separating surfaces, the highest (resp., lowest) such is well-defined as the outer boundary of the set generated by exposing the $1$-connected component of $\partial_N \Lambda_{n,\ell}'$ (resp., $\partial_S \Lambda_{n,\ell}'$) in the set of plaquettes dual to edges that are closed. 
    
    The highest $\mathcal S_N$ is measurable with respect to the states of edges above it, and the lowest $\mathcal S_S$ with respect to the states of edges below it. Therefore, by the domain Markov property, we can resample the configurations $(\omega,\sigma)$ and $(\omega',\sigma')$ in the region sandwiched between these separating surfaces using the identity coupling (the colorings on $\mathcal S_N$ must be $\Red$ in both, and on $\mathcal S_S$ $\Blue$ in both, so indeed the measures induced on the regions sandwiched by the separating surfaces have the same distribution). Since $\cI_{\Full}$ and therefore $\cI_{\Blue}$ must be sandwiched between these two separating surfaces when $|\sfh|\le \ell/2$, on the event $\mathcal G$ we have $$\cI_{\Blue}(\omega,\sigma) \equiv \cI_{\Blue}(\omega',\sigma')\,.$$
    
    It therefore suffices to give the upper bound   $$\bP_{\Lambda_{n,\ell}'}^\sfh\otimes \bP_{\Lambda_{n,m}'}^\sfh\big((\omega, \omega')\notin \mathcal G\big) \le e^{ - n}\,.$$ 
    In order to bound this probability, notice that in each of the two measures, by \cref{thm:gg-interface-height} and \cref{rem:interface-height-const}, the probability of $\cI_{\Full} \cap \cL_{3\ell/4}\ne \emptyset$ or $\cI_{\Full}\cap \cL_{-3\ell/4}$ is at most $e^{ - c \beta n}$. Under both distributions, we can therefore expose the outer boundaries of $\cI_{\Full}$ above and below, which will  induce a pair of wired random-cluster measures (with no conditioning) dominating the infinite-volume super-critical random-cluster measure. By standard Peierls bounds for the super-critical random-cluster measure at large $\beta$, with probability at least $1-e^{ - c\beta n}$, there are common surfaces $\mathcal S_N$ and $\mathcal S_S$ giving the event $\mathcal G$. Indeed, on the complement of the event $\mathcal G$, either for north or south (say w.l.o.g.\ for north) and either for $\omega$ or $\omega'$ (say w.l.o.g.\ $\omega$), there is a 1-connected path of plaquettes between $\cL_{3\ell/4}$ and $\cL_\ell$ such that at least half of the edges dual to these plaquettes are closed. This probability is bounded by the exponential tails on sizes of closed components under the infinite-volume random-cluster measure). 

    The argument for the floor measures $\mu_{\Lambda_{n,\ell}}^\fl$ and $\mu_{\Lambda_{n,m}}^\fl$ is analogous, only needing the existence of $\mathcal S_N$, and no need for $\mathcal S_S$ since the southern boundary is already shared at height zero. The only exception is that we cannot apply \cref{thm:gg-interface-height} directly to bound $\bP_{\Lambda_{n, \ell}}^\fl(\cI_\Full \cap \cL_{3\ell/4} \neq \emptyset)$, so we will additionally use a monotonicity argument involving $\cI_\Blue$ (in the same way that we proved an upper bound for $\cI_\Top$ going through $\cI_\Blue$ first in \cref{subsec:log-upper-bound}). Compare $\bP_{\Lambda_{n, \ell}}^\fl$ to $\bP_{\Lambda_{n, \ell}}^{\ell/2}$. The latter is just a shift of the measure $\bP_{\Lambda'_{n, \ell/2}}^\dob$ by $\ell/2$, so we have by \cref{thm:gg-interface-height} that
    \[\bP_{\Lambda_{n, \ell}}^{\ell/2}(\cI_\Full \cap \cL_{3\ell/4} \neq \emptyset) = \bP_{\Lambda_{n, \ell/2}'}^{\dob}(\cI_\Full \cap \cL_{\ell/4} \neq \emptyset) \leq e^{-c\beta n}\,.\]
    Since $\cI_\Blue\subset\cI_\Full$, the same bound applies for $\cI_\Blue$. Now looking at the Potts marginals, the fuzzy Potts monotonicity and the fact that $\max_x\overline{\hgt}_x(\cI_\Blue)$ is increasing in the set of $\Blue$ vertices implies that we also have
    \[\mu_{\Lambda_{n, \ell}}^\fl(\cI_\Blue \cap \cL_{3\ell/4} \neq \emptyset) \leq e^{-c\beta n}\,.\]
    (The same argument is done at the end of the proof of \cref{prop:Potts-monotonicity}, see for details.) Finally, this implies an upper bound of $\bP_{\Lambda_{n, \ell}}^\fl(\cI_\Full \cap \cL_{7\ell/8} \neq \emptyset) \leq e^{-c\beta n}$  by \cref{cor:dist-Ifull-Ibot}.
\end{proof}

As such, we perform all arguments in this section on infinite and semi-infinite cylinders, establish the sharp upper bound of Theorem~\ref{thm:main} under $\mu_{\Lambda_{n,\infty}}^\fl$, which we denote in this section by $\mu_n^\fl$ for brevity,
and then translate it to the upper bound of Theorem~\ref{thm:main} under $\mu_{\Lambda_{n,n}}^\fl$ via Lemma~\ref{lem:finite-vs-infinite-cylinder} with $m$ sent to $\infty$. 

We further consider the Potts model on the infinite cylinder $\Lambda_{n,\infty}'$
wherein $\mu_{\Lambda'_{n,\infty}}^\dob$ assigns Dobrushin boundary conditions: $\Red$ in the positive half-space and $\Blue$ in the negative one. 
Overriding the notation from \cref{sec:logarithmic-lower-upper}, we define $\mu_{n}^\sfh$ as the $q$-state Potts model on $\Lambda'_{n,\infty}$ with Dobrushin boundary conditions that transition from $\Blue$ to $\Red$ at height $\sfh$, so that $\mu_{\Lambda'_{n,\infty}}^\dob = \mu_n^0$. Let $\musoft_{n}^\sfh$ be the same measure conditioned on $\cI_{\Blue}\subset \cL_{\ge 0}$ as per~\eqref{eq:mu-hat-measure}.  The stochastic domination observation of~\cref{prop:Potts-monotonicity} evidently still holds in the infinite cylinder, by taking the ceiling height to infinity and using the fact that the interface is almost surely finite. 

\subsection{Lower bounding the probability of positivity of the blue interface}
In this subsection, we establish large deviation rates on the maximum height oscillation of a rigid Potts interface. Namely, we establish the following lower bound on the probability of positivity of $\cI_{\Blue}$ under the no-floor measure $\mu_n^{h_n^*}$. 

\begin{lemma}\label{lem:prob-of-positivity}
For every $\beta>\beta_0$, there exists $\epsilon_\beta$ (going to zero as $\beta \to\infty$) such that  
    $$\mu_n^{h_n^*}(\cI_\Blue \subset  \cL_{\geq 0}) \geq e^{-(1+\epsilon_\beta) n}\,.$$
\end{lemma}

Since we are only looking for lower bounds, the proof will go by applying the FKG inequality through the fuzzy Potts model defined in \cref{subsec:random-cluster-fuzzy-Potts}.  By vertical translation, it suffices to consider the downwards oscillations of $\cI_{\Blue}$ from height zero under $\mu_{n}^\dob = \mu_{n}^0$, and by reflection, these have the same law as the upwards oscillations of $\cI_{\Red}$. Thus, we focus on the latter, i.e., on bounding 
\begin{align}
    \mu_{n}^\dob( \max_{x} \overline\hgt_x (\cI_\Red) \le h_n^*) \ge e^{ - (1+\epsilon_\beta) n} \,.
\end{align}
We need the appropriate monotone event in the set of $\noR$ sites, for which we have a sharp understanding of the rate as it relates to $\xi$ from~\cref{eq:xi}. Recall the definition of the augmented set of red vertices $\Ared$ from \cref{def:FK-Potts-interfaces}, whose boundary forms $\cI_{\Red}$. 

\begin{definition}[Non-red pillar]\label{def:non-red-pillar}
    Let $x$ be a vertex at height $1/2$. The non-red pillar at $x$, denoted $\cP_x^\noR$, is the connected component of vertices in $\Ared^c \cap \cL_{\ge 0}$ that contains $x$. The height of the $\cP_x^\noR$, denoted $\hgt(\cP_x^\noR)$ is the maximum height of a plaquette $f$ dual to an edge $(y, z)$ where $y\notin \cP_x^\noR$ and $z \in \cP_x^\noR$. 
\end{definition}

We can then define the pillar rate under the infinite-volume Dobrushin measure 
\begin{align}\label{eq:alpha-h}
\alpha_h  = \alpha_h(\beta,q):= -\log \mu^\dob_{\infty}(\hgt(\cP_{\mathsf{o}}^\noR) \geq h)\,.\end{align}
where we are using the subscript $\infty$ to mean the infinite-volume measure obtained as the $n\to\infty$ limit of $\mu_{n}^\dob$, and we recall that $\mathsf{o}$ is the origin vertex $(\frac{1}2, \frac 12, \frac 12)$. Section~\ref{sec:identification-of-rates} relates this pillar rate to $\xi_h$ from~\eqref{eq:xi} and thus to $h_n^*$, but most of the estimates of this section will be done in terms of $\alpha_h$.

The following lemma explains  why we move to expressing the event in terms of pillars. 
\begin{lemma}\label{lem:Potts-pillar-properties}
    The quantities $\max_{x} \overline{\hgt}_x(\cI_{\Red})$ and $\max_{x} \hgt(\cP_{x}^\noR)$ are the same. Moreover, the event $\{\hgt(\cP_x^\noR)\ge h\}$ is decreasing in the set of $\Red$ vertices, and is equivalent to the event $\mathcal A_{x,h}^\noR$ from~\cite[Def.~5.1]{ChLu24}. 
\end{lemma}
\begin{proof}
        Observe that $\Ared^c$ is connected (see e.g., ~\cite[Remark 2.15]{ChLu24} for this) and contains $\partial_\Blue \Lambda_{n,\infty}'$,  and so any vertex of $\Ared^c$ which has height $\geq 1/2$ is part of some non-red pillar $\cP_x^\noR$. In the other direction, every vertex in $\cP_{x}^\noR$ for some $x$ is an element of $\Ared^c$. Hence, $\max_x \overline\hgt_x(\cI_\Red)$, which is equal to the maximum height of $\Ared^c$, is equal to the maximum height of $\cP_x^\noR$ over all $x$. 
        
        For the monotonicity claim, changing a vertex from $\noR$ to $\Red$ can only increase the set $\Vred$, and consequently can only increase $\Ared$. 
        
        Finally, note that $\cA_{x, h}^\noR$ was defined in~\cite[Def.~5.1]{ChLu24} as the event that there exists a $\Ared^c$-path from $x$ to height $h - 1/2$ using only vertices of a random set $\cP_x$ (which was defined as a random-cluster pillar), but this random set $\cP_x$ is always a subset of $\cL_{>0}$ that contains $\cP_x^\noR$, so restricting to $\cP_x$ is the same as restricting to $\cL_{>0}$ when looking at the connected component of $\Ared^c$ vertices containing~$x$.
\end{proof}

By \cite[Prop.~2.24]{ChLu24} with the observation that the random-cluster pillar $\cP_x$ defined there necessarily contains $\cP_x^\noR$, we have the following weak bound on the height of $\cP_x^\noR$.
\begin{lemma}\label{lem:weak-pillar-bound}
    There exists $C > 0$ such that for all $h \geq 1$, all $x \in \cL_{1/2}$, and sufficiently large $\beta$, 
    \begin{equation*}
        \mu_n^\dob(\hgt(\cP_x^\noR) \geq h) \leq \exp(-4(\beta - C)h)\,.
    \end{equation*}
\end{lemma}

We also have the following decorrelation estimate from \cite[Cor.~A.7]{ChLu24}, which lets us use the rate $\alpha_h$ of \cref{eq:alpha-h} for the pillar deviation rate inside the finite domain $\Lambda'_{n, \infty}$ so long as $x$ is not too close to the boundary.
\begin{lemma}\label{lem:decorrelation-pillar}
    There exists a constant $C$ such that for all $x, y$ so that $d(x, \partial \Lambda'_{n, \infty}) \wedge d(y, \partial \Lambda'_{m, \infty}) \geq r$,
    \begin{equation*}
        |\mu_{n}^\dob(\hgt(\cP_x^\noR) \geq h) - \mu_{m}^\dob(\hgt(\cP_x^\noR) \geq h)| \leq e^{-r/C}\,.
    \end{equation*}
\end{lemma}

With the above two ingredients recalled, and the observations of \cref{lem:Potts-pillar-properties}, we can show the following lower bounds on the large deviations of the maximum height of $\cI_{\Red}$. 
\begin{lemma}\label{lem:potts-extrema-bound}
    For every $1 \leq h \leq \frac{1}{\sqrt{\beta}}\log n$,  we have
    \begin{equation*}
        \mu_n^\dob(\max_x \overline{\hgt}_x(\cI_\Red) < h) \geq \exp\big(-(1+o(1))n^2e^{-\alpha_h}\big)\,.
    \end{equation*}
    Consequently by reflection symmetry, 
    \begin{equation*}
        \mu_n^\dob(\min_{x} \underline{\hgt}_x(\cI_\Blue) > -h) \geq \exp\big(-(1+o(1))n^2e^{-\alpha_h}\big)\,.
    \end{equation*}
\end{lemma}

\begin{proof}
    By the first part of \cref{lem:Potts-pillar-properties}, we can write
    \[\mu_n^\dob\big(\max_{x}\overline{\hgt}_x(\cI_\Red) < h\big) = \mu_n^\dob\Big(\bigcap_{x\in \cL_{1/2}} \{\hgt(\cP_x^\noR) < h \}\Big)\,.\]
    As in \cref{subsec:random-cluster-fuzzy-Potts}, we look at the fuzzy Potts model this time singling out the color $\Red$ and viewing all non-red colors as $\Fuzzy$. Since this model has the FKG inequality~\eqref{eq:fuzzy-Potts-FKG} (identifying all $\Blue$ boundary vertices as a single $\Fuzzy$ vertex to encode that they are obligated to take the same $\noR$ color) and by the second part of \cref{lem:Potts-pillar-properties}, we have that
    \[\mu_n^\dob\Big(\bigcap_x \{\hgt(\cP_x^\noR) < h\}\Big) \geq \prod_x \mu_n^\dob(\hgt(\cP_x^\noR) < h)\,.\]
    where the product runs over $x\in (\mathbb Z+\frac{1}{2})^2 \times \{0\} \cap \Lambda_{n}'$. Terms of $x: d(x,  \partial \Lambda'_{n,\infty}) \leq \log^2 n$  are lower bounded by the complement of \cref{lem:weak-pillar-bound}. For $x$ such that $d(x, \partial \Lambda'_{n,\infty}) > \log^2 n$, we can apply \cref{lem:decorrelation-pillar} with $m \to \infty$ to lower bound the right-hand side by 
    \[\mu_n^\dob(\hgt(\cP_x^\noR) \geq h) \leq e^{-\alpha_h} + e^{-\log^2 n/C} = (1+o(1))e^{-\alpha_h}\,,\]
    where the equality used the facts that $h \le O(\log n)$ and $\alpha_h \le 4(\beta + C)h$ per~\cite[Proposition 1.5]{ChLu24} (a simple forcing argument to lower bound the probability of a deviation via the simplest straight column pillar).   
    Hence, using the inequality $1 - x \geq e^{-x/(1-x)}$ for $0< x < 1$, we get
    \begin{align*}
    \prod_x \mu_n^\dob(\hgt(\cP_x^\noR) < h) &\geq (1-e^{-4(\beta - C)h})^{4n\log^2n} \cdot (1-(1+o(1))e^{-\alpha_h})^{(n-\log^2n)^2}\\
    &\geq \exp\left(-e^{-4(\beta - C)h}4n\log^2 n(1+o(1)) -e^{-\alpha_h}n^2(1+o(1))\right)\\
    &= \exp\left(-e^{-\alpha_h}n^2(1+o(1))\right)\,,
    \end{align*}
    where the last equality used the upper bound on $h$ and the fact that $\alpha_h \le 4(\beta + C)h$ (so that $e^{ 8 C h} \le n^{\epsilon_\beta}$ cannot make up for the difference between $n\log^2 n$ and $n^2$). 
\end{proof}

We now plug in $h = h_n^* + 1$ to obtain \cref{lem:prob-of-positivity}, given the comparison of $\alpha_h$ and $\xi h$ for $\xi$ from~\eqref{eq:xi} (which will separately be established in \cref{sec:identification-of-rates}).

\begin{proof}[\textbf{\emph{Proof of \cref{lem:prob-of-positivity}}}]
 Theorem~\ref{thm:identifying-rate-with-infinite-limit} will prove that $\alpha_h$ is very close (up to an $\epsilon_\beta$ independent of $h$ exactly related to) to $\xi h$. In particular, by \cref{thm:identifying-rate-with-infinite-limit} and definition of $h_n^*$ \eqref{eq:h_n^*}, 
\begin{align}\label{eq:alpha-h_n^*+1-bound}
    \alpha_{h_n^*+1} &\ge \xi\cdot (\lfloor \xi^{-1}\log n\rfloor +1)  -\epsilon_\beta \ge \log n - \epsilon_\beta\,.
\end{align}
By applying the second part of \cref{lem:potts-extrema-bound} with $h= h_n^*+1$, and translating vertically by $h_n^*$, we conclude.
\end{proof}

\subsection{Large deviations upper bound on number of sites not at \texorpdfstring{$h_n^*$}{h\_n}}\label{subsec:ld-upper-bound}

The second part to proving the $h_n^*$ upper bound in \cref{thm:main} is  establishing that violating an upper bound on the typical height under $\mu_{n}^{h_n^*}$ is even more atypical than the $e^{- (1+\epsilon_\beta) n}$ lower bound given on interface positivity in \cref{lem:prob-of-positivity}. Note that this is an estimate purely on a no-floor rigid interface. 

\begin{lemma}\label{lem:large-deviation-area-in-walls}
There exists $C>0$ such that for every $\beta>\beta_0$
        \begin{align*}
        \mu_{n}^{h_n^*}\Big( \big|\{x : \overline\hgt_x(\cI_{\Blue}) \ge h_n^*+1\} \big| \ge \tfrac{C}{\sqrt{\beta}}n^2\Big) \le \exp( - \sqrt{\beta} n/C)\,. 
    \end{align*}
\end{lemma}

We will prove the above by bounding the set of all index points where the full random-cluster interface is not a single horizontal plaquette, under the corresponding Dobrushin random-cluster measure. By vertical translation let us work with $\bar{\pi}_{n}^\dob= \bar{\pi}_n^0$.

 We need to recall the definition of interface walls for $\cI_{\Full}$  from~\cite{GielisGrimmett02}. For a plaquette $f$, let $\rho(f)$ be its projection onto $\cL_0$. Recall that two plaquettes are 1-connected if their intersection contains a unit line segment, and 0-connected if their intersection contains a point.

\begin{definition}[Walls and Ceilings of $\cI_\Full$]
    Let $\cI_\Full^*$ be the union of $\cI_\Full$ (defined in Definition~\ref{def:full-interface}) together with all plaquettes which are horizontal and 1-connected to $\cI_\Full$. If a plaquette $f$ is the unique plaquette in $\cI_\Full^*$ with projection onto $\cL_{0}$ equal to $\rho(f)$, then $f$ is a ceiling plaquette. All other plaquettes of $\cI_\Full^*$ are wall plaquettes. A ceiling is a 0-connected component of ceiling plaquettes. A \emph{wall} is a pair $(A, B)$ where $W$ is a 0-connected component of wall plaquettes, $A = W \cap \cI_\Full$, and $B = W \cap (\cI_\Full^* \setminus \cI_\Full)$. (This decomposition is called a standard wall in \cite{Grimmett_RC}. By \cite[Lem.~7.125(i)]{Grimmett_RC}, there are no ceiling plaquettes in $\cI_\Full^* \setminus \cI_\Full$, so we do not need such a decomposition for ceilings.) 
\end{definition}

\begin{lemma}[{\cite[Lem.~12]{GielisGrimmett02},\cite[Lem.~7.127]{Grimmett_RC}}]
    In $\cI_\Full$, there is a 1-1 correspondence between interfaces and admissible families of walls (quotiented out by vertical translations), where two walls are called admissible if their projections are not 0-connected and they are contained in the domain (see items (i)--(ii) on page 211 of~\cite{Grimmett_RC} for the formal statement). 
\end{lemma}

\begin{definition}[Excess area]
    The excess area between two interfaces $\cI$ and $\cJ$ is the difference in their number of plaquettes, $\fm(\cI;\cJ) := |\cI| - |\cJ|$. The excess area of a wall $W = (A, B)$ of $\cI_\Full$, is $\fm(W) := |A| - |\rho(W)| := |A| - |\rho(A \cup B)|$. (When we write $|\rho(F)|$ for a set of faces $F$, we mean the area, or number of faces in $\rho(F)$, so projections of vertical faces do not contribute.)
\end{definition}

For a wall $W$, the projection of its \emph{hull}, $\rho(\hull{W})$, is the union of $\rho(W)$ with all finite components of $\mathbb R^2 \setminus \rho(W)$. Say a wall $W$ of $\cI_{\Full}$ is \emph{outermost} if there is no wall $W'$ of $\cI_{\Full}$ such that $\rho(W)\subset \rho(\hull{W'})$. Finally, two walls are called mutually external if neither $\rho(W)\subset \rho(\hull{W'})$ nor $\rho(W')\subset \rho(\hull{W})$. 

This combinatorial construction (along with some further constructions we do not need in what follows) were used in~\cite{GielisGrimmett02} to establish rigidity of $\cI_{\Full}$ under $\mu_{n}^\dob$. In particular, Peierls-type maps based on those in~\cite{Dobrushin73} were used to delete walls from the $1$-to-$1$ representation of interfaces, to get exponential tail bounds on the excess areas of walls of $\cI_{\Full}$. The following is an immediate corollary of \cite[Lem.~15]{GielisGrimmett02} (see also \cite[Lem.~7.132]{Grimmett_RC}).

\begin{corollary}
\label{cor:many-walls-tail}
    There exists $C>0$ such that for every $\beta>\beta_0$, and every arbitrary set of admissible mutually external walls $W_1,\ldots, W_N$, 
    \begin{align*}
        \bar{\pi}_{n}^\dob ( W_1,\ldots,W_N) \le \exp\Big( - (\beta - C) \sum_{i=1}^N \fm(W_i)\Big)\,,
    \end{align*}
    where we are using $\bar{\pi}_{n}^\dob(W_1,\ldots,W_N)$ to mean the probability that $\{W_1,\ldots,W_N\}$ are walls of $\cI_{\Full}$. 
\end{corollary}


Given Corollary~\ref{cor:many-walls-tail}, the following bound uses a union bound over the contribution from walls from diadically growing scales interpolating between those confining interior areas $1$ to $n^2$.  

\begin{lemma}\label{lem:outermost-walls}
    Let $\mathfrak U$ be the set of all outermost walls of the interface
$\cI_{\Full}$ under the no-floor Dobrushin random-cluster measure $\bar{\pi}_{n}^\dob$. There exists $C>0$ such that for all $\beta>\beta_0$, 
    \begin{align}\label{eq:nts-ints-of-walls}
        \bar\pi_{n}^\dob\Big(\sum_{W\in \mathfrak U} |\rho(\hull{W})| \ge \frac{C}{\sqrt{\beta}} n^2\Big) \le e^{ - \sqrt{\beta} n /C} \,.
    \end{align}
\end{lemma}
\begin{proof}
    Partition the set of all outermost walls into the sets $\mathfrak U_1,\ldots,\mathfrak U_{2L}$ for $L = \log_2 n$, where 
    \begin{align}\label{eq:fU-k}
        \mathfrak U_k = \{W\in \mathfrak U: |\rho(\hull{W})|\in [2^{k-1}, 2^k)\}\,.
    \end{align}
    We will show that for a suitable absolute constant $C_0$, for each $k = 1,\ldots,2L$, 
    \begin{align}\label{eq:nts-small-k}
        \bar\pi_{n}^\dob \Big(\sum_{W\in \mathfrak U_k} |\rho(\hull{W})| \ge \epsilon_{\beta,k}^2 n^2 \Big) \le e^{- (\beta - C) 2^{-\frac{k}{2}}{\epsilon_{\beta,k}^2 n^2}} \quad \text{where} \quad         \epsilon_{\beta,k} = \frac{C_0}{\beta^{1/4} \min\{k , 2L +1 -k\}}\,. 
    \end{align}
    Let us first conclude the proof using~\eqref{eq:nts-small-k}: for an absolute constant $C$, we have,  
    \begin{align*}
        \sum_{k=1}^{2L} \epsilon_{\beta,k}^2 \le \frac{C}{\sqrt{\beta}}\,,
    \end{align*}
    so the union of the above events covers $\sum_{W\in \mathfrak U}|\rho(\hull{W})| \ge \frac{C}{\sqrt{\beta}} n^2$. Summing the probability bounds of~\eqref{eq:nts-small-k}, we can bound the probability by $2L$ times the maximum of the probabilities, which is evidently attained at the term $k=2L$ where it is at most $\exp( - C_0 (\sqrt\beta - \frac{C}{\sqrt{\beta}}) n)$, yielding the claimed~\eqref{eq:nts-ints-of-walls} up to the change of constant. 

    We now move to proving~\eqref{eq:nts-small-k}. Fix $k$ and note by~\eqref{eq:fU-k} and isoperimetry in $\mathbb Z^2$, if $|\mathfrak U_k| = N$,  
    \begin{align}\label{eq:isoperimetry-inputs}
        N\le \sum_{W\in \mathfrak U_k} |\rho(\hull{W})| 2^{ 1-k}\,, \qquad\text{and} \qquad \sum_{W\in \mathfrak U_k} \fm(W) \ge 2 \sum_{W \in \mathfrak U_k} |\rho(\hull{W})| 2^{ - k/2}\,.
    \end{align}
    By applying Corollary~\ref{cor:many-walls-tail}, for any fixed realization of $\mathfrak U_k$, we have 
    \begin{align*}
        \bar\pi_{n}^\dob \Big(\{W\in \mathfrak U_k\} \text{ are walls in $\cI_{\Full}$}\Big) \le \exp\Big( - (\beta - C) \sum_{W\in \mathfrak U_k} \fm(W)\Big)\,.
    \end{align*}
    We union bound over collections $\mathfrak U_k$ as follows: letting $\chi$ be the fraction $\frac{1}{n^2}\sum_{W\in \mathfrak U_k} |\rho(\hull{W})|$,
    \begin{align*}
        \bar\pi_{n}^\dob\Big(\sum_{W\in \mathfrak U_k} |\rho(\hull{W})| \ge (\delta n)^2\Big) & \le \sum_{(\delta n)^2\le \chi n^2\le n^2} \sum_{N \le \chi n^2 2^{1-k}}\binom{n^2}{N}  \sum_{r \ge \chi n^2 2^{-\frac{k}{2} +1}} C^r e^{ - (\beta - C) r} \\ 
        & \le \sum_{(\delta n)^2\le \chi n^2\le n^2} \sum_{N \le \chi n^2 2^{1-k}} \binom{n^2}{N} e^{ - 2^{1-\frac{k}{2}}(\beta - C) \chi n^2}\,.
    \end{align*}
    The first line used~\eqref{eq:isoperimetry-inputs} for the bounds on the sums, the binomial is for the choice of $N$ points at which to root the walls of $\mathfrak U_k$, and the factor of $C^r$ is for the number of possible ways to allocate excess area $r$ to the different root points and enumerate over the realizations of the walls. Using the bound $\sum_{j\le p A} \binom{A}{j} \le \exp( \mathsf{H}(p) A) $ where $\mathsf{H}(p)$ is the binary entropy function $- p\log p - (1-p)\log (1-p)$, we get 
    \begin{align*}
        \bar\pi_{n}^\dob\Big(\sum_{W\in \mathfrak U_k} |\rho(\hull{W})| \ge (\delta n)^2\Big) \le \sum_{(\delta n)^2 \le \chi n^2 \le n^2} \exp\Big( \Big(\mathsf{H}(\chi 2^{1-k}) - 2^{1-\frac{k}{2}} (\beta - C) \chi\Big) n^2 \Big)\,.
    \end{align*}
    It therefore suffices for us to show (now with the choice of $\delta = \epsilon_{\beta, k}^2$) that for every $k$, for every $\chi \ge \epsilon_{\beta,k}^2$, we have $\mathsf{H}(\chi 2^{1-k}) \le (\beta - C) \chi 2^{ - k/2}$ (we would then absorb the prefactor of $n^2$ and be done). To show this, we use the bound $\mathsf{H}(p) \le p \log\frac{1}{p} + p$ (which holds as long as $p \le 1/2$, which is the case for all $k\ge 2$ trivially, and for $k=1$ by the fact that there cannot be distinct walls each of interior size only $1$ confining half the area of $\Lambda$). It suffices to show that 
    \begin{align*}
        2^{1-k} \log \tfrac{2^{k-1}}{\chi} + 2^{1-k} \le (\beta - C)2^{-\frac{k}{2}}\,,
        \end{align*}
        or, equivalently,
        \begin{align}\label{eq:entropy-algebra}
        (k-1) (\log 2) + \log \chi^{-1} +1 \le (\beta -C) 2^{\frac{k}{2}-1}\,.
    \end{align}
    Let us consider first the case where $k=1,\ldots,L$, in which case $\log \chi^{-1} \le \log \beta + \log(k^2)- \log (C_0^2)$. We can take $C_0$ large (independent of $\beta$ because $2^{\frac{k}{2}-1} \beta$ beats $\log \beta$ for all $k\ge 1$) to only consider large values of $k$, whence the exponential growth on the right-hand side of~\eqref{eq:entropy-algebra} of course dominates the left. Turning to the case where $k=L+1,\ldots,2L$, now $\log \chi^{-1} \le\log \beta  + 2 \log(2L + 1-k)- \log (C_0^2)$. For every $k\ge L$, the left-hand side of~\eqref{eq:entropy-algebra} is thus at most $O(\log n)$, whereas the right-hand side is at least $(\beta-C)n^{1/2}/2$. 
\end{proof}

\begin{proof}[\textbf{\emph{Proof of \cref{lem:large-deviation-area-in-walls}}}]
    It is equivalent by vertical translation to establish the upper bound on $$\mu_{n}^\dob \Big(\big|\{x : \overline{\hgt}_x(\cI_{\Blue}) \ge 1\} \big| \ge \tfrac{C}{\sqrt{\beta}}n^2\Big) \le \bar\pi_{n}^\dob \Big(\big|\{x : \overline{\hgt}_x(\cI_{\Full}) \ge 1 \}\big| \ge \tfrac{C}{\sqrt{\beta}}n^2\Big)\,,$$ 
    where we used inclusion of $\cI_{\Blue} \subset \cI_{\Full}$ and the FK--Potts coupling. Now any $x$ having $\overline{\hgt}_x(\cI_{\Full})\ne 0$ or $\underline{\hgt}_x(\cI_{\Full})\ne 0$ must belong to $\rho(\hull{W})$ for some wall $W$. Thus, the right-hand side is bounded by \cref{lem:outermost-walls}, concluding the proof. 
\end{proof}

We are now in position to combine everything to get the sharp upper bound in \cref{thm:main}. 
\begin{proof}[\textbf{\emph{Proof of upper bound in Theorem~\ref{thm:main}}}]
    By \cref{lem:prob-of-positivity} and~\cref{lem:large-deviation-area-in-walls},
    \begin{align*}
        \musoft_{n}^{h_n^*} \Big(\big|\{x : \overline{\hgt}_x(\cI_{\Blue}) \ge h_n^* + 1 \}\big| \ge \tfrac{C}{\sqrt{\beta}}n^2\Big) & \le \frac{\mu_{n}^{h_n^*}\Big(\big|\{x : \overline{\hgt}_x(\cI_{\Blue}) \ge h_n^* + 1 \}\big| \ge \tfrac{C}{\sqrt{\beta}}n^2\Big)}{\mu_{n}^{h_n^*}(\cI_{\Blue}\subset \cL_{\ge 0})} \\ 
        & \le  e^{ - \sqrt{\beta} n/C + (1+\epsilon_\beta)n}\,,
    \end{align*}
    yielding a bound of $e^{ - \sqrt{\beta}{n}/C}$ for large $\beta$, up to a change of the absolute constant $C$. 
    
    To conclude the upper bound of Theorem~\ref{thm:main}, we note that $|\{x : \overline{\hgt}_x(\cI_{\Blue}) \ge h_n^*+1\}| \ge \tfrac{C}{\sqrt{\beta}}n^2$ is a monotone event in the set of $\Blue$ vertices (because it is monotone in the augmented $\Blue$ vertex set $\Ablue$, and in turn $\Vblue$ which is expressed as a union of vertices connected by $\Blue$ paths to $\partial_\Blue \Lambda'_{n, \infty}$), and thus \cref{prop:Potts-monotonicity} gives the same for $\mu_n^\fl$. 

    Finally, as done in the proof of the lower bound, we observe that on the event $|\cI_\Full| < (1+\tfrac{C}{\beta})n^2$ (which by \cref{lem:Ifull-wall-faces} occurs with probability $1 - e^{-n^2}$ and hence can be assumed), we have
    \begin{align*}
    \big\{|\cI_{\Blue} \cap \cL_{\ge h_n^*+1}| >\tfrac{C}{\sqrt{\beta}} n^2 + \tfrac{C}{\beta} n^2\big\} \subset  \big\{|\{x: \overline \hgt_x(\cI_\Blue) 
    \ge h_n^*+1\}| >\tfrac{C}{\sqrt\beta} n^2 \big\} \,,
\end{align*}
which concludes the proof of the upper bound with the choice of $\epsilon_\beta = \tfrac{C}{\sqrt{\beta}}$, holding with probability $1 - e^{-\sqrt{\beta}n/C}$. 
\end{proof}

\section{Identification of the pillar rates with infinite-volume connectivity rates}\label{sec:identification-of-rates}
The preceding sections established our Theorem~\ref{thm:main} up to the deferred identification of the pillar rate $\alpha$ with the with the point-to-plane connectivity rate $\xi$. In this section, our goal is to give the precise relation between the pillar rate and the point-to-plane connectivity rate, and establish the additivity of the latter rate to be able to identify $h_{n}^*$ with $\lfloor \xi^{-1} \log n\rfloor$ per~\eqref{eq:h_n^*}.

\subsection{Comparison of pillar rate to point-to-plane connectivity rate}
Our approach is to work conditional on certain high probability events on the shape of the interface pillar under $\mu^\dob_\infty$ and the connected component of $\Vred^c$ under $\mu^\Red_\infty$ under which we can decorrelate the discrepancies in the boundary conditions using cluster expansion machinery. Recall that we denote the origin by $\mathsf{o} = (\frac12, \frac12, \frac12)$, and recall the definitions of the two rates we are interested in comparing, 
\begin{align*}
    \alpha_h = -\log \mu_{\infty}^{\dob}\big(\hgt(\cP_{\mathsf{o}}^\Blue) \ge h \big) \,, \quad \text{and} \quad \xi_h = - \log \mu_\infty^\Red\big(\mathsf{o}  \longleftrightarrow \cL_h \text{ in $\Vred^c$}\big)\,.
\end{align*}

For simplicity, define $A_{\mathsf{o}, h} = \{\mathsf{o}\longleftrightarrow \cL_h \text{ in $\Vred^c$}\}$. The main goal of this section is to establish the following. 
\begin{theorem}\label{thm:identifying-rate-with-infinite-limit}
    Fix any $q \ge 2$. We have for all $h\ge 1$ that 
    \begin{align*}
        |\alpha_h - \xi h|\le \epsilon_\beta\,.
    \end{align*}
\end{theorem}

For this, we will need a generalization of \cref{prop:grimmett-cluster-exp}, which may also be of independent interest. We begin with some preliminary definitions. Recall for any collection of plaquettes $F$, we let $\overline{F}$ denote the set of plaquettes which are 1-connected to $F$ (including $F$ itself), and set $\partial F := \overline{F} \setminus F$. If there is a reference domain $\cU$, then $\overline{F}$ is presumed to restrict to $\cU$. Let $f_e$ denote the plaquette $f$ dual to an edge $e$. Let $\mathsf{Cl}(F)$ be the event that for all $f_e \in F$, $\omega_e = 0$ and for all $f_e \in \partial F$, $\omega_e = 1$, i.e., that $F$ forms a $1$-connected dual-to-closed component. 

Now let $\cU$ be any subgraph of $\Z^3$ and let $\eta$ be any boundary condition. Let $\cB$ be the set of vertices in $\cU$ which are also adjacent to a vertex in $\Z^3 \setminus \cU$. Viewing $\eta$ in the dual graph as a set of dual-to-open and dual-to-closed plaquettes, we will say a collection of plaquettes $F$ is a compatible set for $(\cU, \eta)$ if for every dual-to-closed plaquette $f$ of $\eta$ in $\Z^3 \setminus \cU$, every plaquette $g$ which is 1-connected to $f$ and in $\cU$ is in $\overline{F}$. 

Finally, for $\cU$ finite, let $Z^\eta(\cU)$ denote the partition function for the random-cluster measure on $\cU$ with boundary conditions $\eta$.

\begin{proposition}\label{prop:cluster-expansion}
    There exists a constant $K > 0$ and functions $\g(f, F, \cU)$ taking as inputs $F$ a collection of plaquettes, $f \in F$, and $\cU$ a subgraph of the lattice $\Z^3$, such that $\g$ satisfies the following:
    \begin{enumerate}
        \item\label{it:L-inf-bound} For any $F$, $f \in F$, and $\cU \subset  \Z^3$, 
        \[|\g(f, F, \cU)| \leq e^{-\beta/K}\,.\]

        \item\label{it:exp-decay} For any $f \in F$, $f' \in F'$, and any domains $\cU, \cU' \subset  \Z^3$, we have
        \[|\g(f, F, \cU) - \g(f', F', \cU')| \leq e^{-\beta r/K}\,,\]
        where $r$ is the largest radius such that $B_r(f) \cap \cU \cong B_r(f') \cap \cU'$ and $B_r(f) \cap F \cong B_r(f') \cap F'$, with $B_r(f)$ denoting the $L_\infty$ ball of radius $r$ about the center of $f$, and $\cong$ denoting equivalence (either as graphs or plaquette sets) up to translation.
        \item\label{it:cluster-expansion} For any $\cU \subset  \Z^3$ with boundary condition $\eta$, and any compatible $F$ for $(\cU, \eta)$,
        \[\pi_\cU^\eta(\mathsf{Cl}(F)) = \frac{Z^{\mathsf{w}}(\cU)}{Z^\eta(\cU)}(1-e^{-\beta})^{|\partial F|}e^{-\beta|F|}q^{\kappa(F) - 1}\exp\Big(\sum_{f \in F} \g(f, F, \cU)\Big)\,,\]
        where $\mathsf{w}$ is the wired boundary conditions, $\kappa(F)$ is the number of connected components of $\cU$ after removing the edges $e$ such that $f_e \in F$, and we are viewing vertices which are wired together by open edges of $\eta$ in $\Z^3 \setminus \cU$ as connected.
    \end{enumerate}
\end{proposition}

    The proof of \cref{prop:cluster-expansion} is deferred to \cref{sec:appendix}.

\begin{remark}
    Note that the above proposition can be applied in the setting where $\eta$ is the Dobrushin boundary condition and $F$ is a potential realization of $\cI_\Full$, recovering the law of $\cI_\Full$ as in \cite[Lem.~7.104]{Grimmett_RC}. Even in this case however, we emphasize that the upper bounds in \cref{it:L-inf-bound,it:exp-decay} now decay with $\beta$, whereas they did not in the formulation of \cref{prop:grimmett-cluster-exp}. 
\end{remark}

Because we wish to use the above cluster expansion machinery on finite graphs to prove \cref{thm:identifying-rate-with-infinite-limit}, we will work in the finite graph $\Lambda_{n,n}'$. This poses no issues, as we have
\begin{align*}|\log \mu_{\infty}^{\dob}\big(\hgt(\cP_{\mathsf{o}}^\noR) \ge h \big) - \log \mu_n^{\dob}\big(\hgt(\cP_{\mathsf{o}}^\noR) \ge h \big)| = o(1)\,,
\end{align*}
and 
\begin{align*}|\mu_\infty^\Red\big(A_{\mathsf{o}, h}\big) - \mu_n^\Red\big(A_{\mathsf{o}, h}\big)| = o(1)\,,\end{align*}
by weak convergence of measures (and moreover with explicit rates of convergence via \cref{lem:decorrelation-pillar} and \cref{prop:cluster-expansion}). Thus, for each $h$, we can take $n \gg h$ large enough so that the error in using a finite box instead of infinite volume is smaller than $\epsilon_\beta$. Hence, in the remainder of this section, the lemmas will be stated in terms of the infinite volume quantities $\alpha_h$ and $\xi_h$, but the proofs will work with measures on $\Lambda_{n,n}'$. 

Recall that we have the FK--Potts measure on $\Lambda_{n,n}'$ with Dobrushin boundary conditions denoted $\bP_n^\dob$ with marginals $(\mu_n^\dob, \bar\pi_n^\dob)$. We also have the FK--Potts measure with wired red boundary conditions denoted $\bP_n^\Red$ with marginals $(\mu_n^\Red, \pi_n^{\mathsf{w}})$.

\begin{lemma}\label{lem:alpha-related-to-xi-Potts}
    Let $\beta$ be sufficiently large and $q\ge 2$. We have that for all $h\ge 1$, 
    \begin{align*}
         \big|\alpha_h - (\xi_h -2\beta + \log(q-1))\big|  \le \epsilon_\beta\,.
    \end{align*}
\end{lemma}

\begin{proof}
    The discrepancy of $2\beta - \log(q-1)$ occurs because of differences of the marginal on $\mathsf{o}= (\frac{1}{2}, \frac{1}{2},\frac{1}{2})$, so we will isolate this by defining $\cE^\dob$ and $\cE^{\mathsf{w}}$ as purely random-cluster events, intending to capture the typical behavior in the random-cluster marginals $\bar\pi^\dob_n$ and $\pi^{\mathsf{w}}_n$ when $\sigma_{\mathsf{o}}$ is in $\Vred^c$ under $\bP^\dob_n$ and $\bP^\Red_n$ respectively. First, we define the following cones:
    \begin{align*}
        \mathbf{C}_{\vee} := ([-\tfrac{1}{2},\tfrac{1}{2}]^2 \times [0,L_\beta]) \cup \{(z_1,z_2,z_3): z_3^2 \ge \sqrt{z_1^2 + z_2^2} \text{ for $z_3 \ge L_\beta$}\}\,,
    \end{align*}
    and 
    \begin{align*}
        \mathbf{C}_{\curlyvee}:= \cL_0 \cup \{(z_1,z_2,z_3): z_3 \le (\log \sqrt{z_1^2 + z_2^2})^2 \text{ for $\sqrt{z_1^2 + z_2^2}>L_\beta$}\}\,,
    \end{align*}
    where $L_\beta$ is chosen such that $L_\beta \to \infty$ as $\beta \to \infty$. Let $\cC$ be the outermost $1$-connected set of dual-to-closed plaquettes which separates $\mathsf{o}$ from $\partial_\Red \Lambda_{n,n}'$, i.e., there is no path from $\mathsf{o}$ to $\partial_\Red \Lambda_{n,n}'$ which does not pass through $\cC$.  Let $C_1$ be the four vertical plaquettes surrounding~$\mathsf{o}$. 
    \begin{itemize}
        \item Let $\cE_1$ be the event that $C_1 \subset \cC$, and $\cC \setminus C_1$ consists of two 1-connected components.

        \item Let $\cE_2$ be the event that the 1-connected component of $\cC \setminus C_1$ that is above $C_1$ lies inside the cone $\mathbf{C}_{\vee}$. Call this portion of the component $\cC^*$.

        \item Let $\cE_3^\dob$ be the event that the rest of $\cC$ lies in $\mathbf{C}_{\curlyvee}$ and $\mathsf{o}$ is $\omega$-connected to $\partial_\Blue \Lambda_{n,n}'$. Let $\cE_3^{\mathsf{w}}$ be that the rest of $\cC$ is just the single horizontal plaquette below $\mathsf{o}$.

        \item Let $\cE_4^\dob$ be the event that there is no $\omega$-path from $\partial_\Blue \Lambda_{n,n}'$ to $\partial_\Red \Lambda_{n,n}'$ (equivalently on $\cE_3^\dob$, no $\omega$-path from $\mathsf{o}$ to $\partial_\Red \Lambda_{n,n}'$). Let $\cE_4^{\mathsf{w}}$ be that there is no $\omega$-path from $\mathsf{o}$ to $\partial \Lambda_{n,n}'$.
    \end{itemize}
    Define $$\cE^\dob = \cE_1 \cap \cE_2 \cap \cE_3^\dob \cap \cE_4^\dob \qquad \text{and} \qquad \cE^{\mathsf{w}} = \cE_1 \cap \cE_2 \cap \cE_3^{\mathsf{w}} \cap \cE_4^{\mathsf{w}}\,.$$ Note that under $\bar\pi^\dob_n$ (and on $\cE^\dob$), $\cC$ will be equivalent to the $\Full$ interface $\cI_\Full$.
We observe (see~\cite[Lemma 4.5]{GL_DMP} for the algebra) that
\begin{align}\label{eq:sum-of-g-terms-cones}
    \sum_{f\in \mathbf{C}_\vee} \sum_{g\in \mathbf{C}_{\curlyvee}} e^{ - \beta d(f,g)/K} \le \epsilon_\beta\,.
\end{align}
We begin by showing that $\cE^\dob$ and $\cE^{\mathsf{w}}$ are indeed typical events, in the sense that
\begin{align}\label{eq:E_vee-high-probability}
    \bP_n^\Red(\cE^{\mathsf{w}} \mid A_{\mathsf{o}, h}) \ge  1-\epsilon_\beta \,, \qquad \text{and} \qquad \bP_n^{\dob}(\cE^\dob \mid \hgt(\cP_\mathsf{o}^\noR) \geq h) \ge 1- \epsilon_\beta\,.
\end{align}
from which it would follow that  
\begin{align*}
    \bP_n^\Red(A_{\mathsf{o}, h}) \leq (1 + \epsilon_\beta) \bP_n^\Red(\cE^{\mathsf{w}} ,  A_{\mathsf{o}, h})\,, 
    \end{align*}
    and
    \begin{align*}
        \bP_n^\dob(\hgt(\cP_\mathsf{o}^\noR) \geq h) \leq (1 + \epsilon_\beta) \bP_n^\dob(\cE^\dob, \hgt(\cP_\mathsf{o}^\noR) \geq h)\,.
\end{align*}
The second (more complicated) of the bounds in~\eqref{eq:E_vee-high-probability} is a consequence of the ``isolated pillar'' results of \cite{ChLu24}. There, a ``top pillar'' $\cP^\Top_{\mathsf{o}}$ was defined analogous to \cref{def:non-red-pillar} but using $\Vtop$ instead of $\Vred$ (we will not use this definition except to reference the results proved in that paper). If $E_h$ is the event that the top pillar reaches height $h$, then \cite[Thm.~3.8]{ChLu24} (in the case of the Ising model,~\cite[Thm.~4.2]{GL_DMP}) 
proves that 
\begin{equation}\label{eq:iso}
    \bar\pi_n^\dob(\cE^\dob \mid E_h) \geq 1 - \epsilon_\beta\,,
\end{equation}
(the event $\Iso_{x, L_\beta, h}$ defined there implies the event $\cE^\dob$ above). On the other hand, by~\cite[Lemmas 5.16, 5.17]{ChLu24} (with $\cA$ there taken to be $\hgt(\cP_\mathsf{o}^\noR) \geq h$ and $\Omega$ taken to be $\cE^\dob$), we have
\begin{equation}\label{eq:move-to-nice-space-conditional}
\left|\frac{\bP^\dob_n(\hgt(\cP_\mathsf{o}^\noR) \geq h \mid \cE^\dob)}{\bP^\dob_n(\hgt(\cP_\mathsf{o}^\noR) \geq h \mid E_h)} - 1\right| \leq \epsilon_\beta\,.
    \end{equation}
Combining \cref{eq:iso,eq:move-to-nice-space-conditional} with the observation that $\hgt(\cP_\mathsf{o}^\noR) \geq h \subset E_h$ concludes. 

We now argue that \cref{eq:E_vee-high-probability}, follows from a simplification of the same arguments of~\cite{ChLu24} that yielded \cref{eq:iso,eq:move-to-nice-space-conditional}. The proof of \cref{eq:iso} began by defining a cut point as a vertex $v\in \cP^\Top_{\mathsf{o}}$ with no other vertices in $\cP^\Top_{\mathsf{o}}$ having the same height, and such that removing the four vertical plaquettes surrounding $v$ cuts $\cC$ into two 1-connected components, with the component above $v$ contained in $\cL_{> \hgt(v)}$. The portion of $\cC$ above the lowest cut point was labeled the ``spine'' $\mathcal S_{\mathsf{o}}$. There was then a map which straightens the spine when it is too large, and deletes walls in the rest of $\cC$ when they are too close to the spine. Doing so deletes the portion of $\cP^\Top_\mathsf{o}$ below the spine, so the map finishes by appending a column to rejoin the spine to the rest of $\cC$. The proof follows the same way, treating all of $\cC$ as the top pillar. Namely, one can define a cut point and the spine of the component $\mathcal C$, and apply the same map to the spine portion. Since under $\pi^{\mathsf{w}}_n$ there is no need to preserve the disconnection event $\sep_{n, n}$ (recall \cref{eq:disconnection-event}) as there was under $\bar\pi^\dob_n$, the remainder $\cC\setminus \mathcal S_{\mathsf{o}}$ is handled by the simpler Peierls map which deletes it entirely and replaces it by a column of height $\hgt(v)$ up from $\mathsf{o}$. 
Finally, \cref{eq:move-to-nice-space-conditional} was shown by observing that when a portion of the spine is straightened, it becomes a column of vertices connected by open edges, and hence under the FK--Potts coupling they are all the same color. Clearly this observation holds regardless of whether the joint measure is $\bP_n^\Red$ or $\bP_n^\dob$.

Hence, with \cref{eq:E_vee-high-probability} in hand, it is sufficient to show 
\begin{align}
    \frac{\pi^{\mathsf{w}}_n(\cE^{\mathsf{w}})}{\bar\pi^\dob_n(\cE^\dob)} & = (1\pm \epsilon_\beta) qe^{ - 2\beta}\,, \qquad\text{and} \label{eq:need-to-show-prob-curly} \\
    \frac{\bP_n^\Red(A_{\mathsf{o}, h}\mid \cE^{\mathsf{w}})}{\bP_n^\dob(\hgt(\cP_\mathsf{o}^\noR) \geq h\mid \cE^\dob)} & = (1\pm \epsilon_\beta)\frac{q-1}{q}\,.\label{eq:need-to-show-A-given-curly}
\end{align}
(where we use the shorthand notation $a = b(1\pm\epsilon_\beta)$ to mean $a \in [b(1 - \epsilon_\beta),b(1 + \epsilon_\beta)]$). We begin with~\eqref{eq:need-to-show-prob-curly}. It is easy to see (by a Peierls argument) that $\pi^{\mathsf{w}}_n(\cE^{\mathsf{w}})$ is dominated by (i.e., $1-\epsilon_\beta$ of its mass is on) the event that $\cC$ is just the six plaquettes surrounding $\mathsf{o}$, which happens with probability $qe^{-6\beta}(1\pm \epsilon_\beta)$. Moreover, by rigidity of $\cI_{\Full}$ (\cref{thm:gg-interface-height})
\begin{align*}
    \bar \pi_n^\dob(f_{(\mathsf{o}-\ez,\mathsf{o})} \text{ is a ceiling plaquette}) \ge 1-\epsilon_\beta\,,
\end{align*}
and then with a cost of $(1\pm \epsilon_\beta) e^{ -  4\beta}$ we can add one open edge between $\mathsf{o} - \ez$ to $\mathsf{o}$ and close the five other incident edges to $\mathsf{o}$. This shows that $\bar\pi^\dob_n(\cE^\dob) = (1\pm \epsilon_\beta)e^{-4\beta}$.

We turn to establishing~\eqref{eq:need-to-show-A-given-curly}. First note that on $\cE^\dob$, the event $\hgt(\cP_\mathsf{o}^\noR) \geq h$ is equivalent to $A_{\mathsf{o}, h}$. Now, intending to utilize the FK--Potts coupling, we wish to first reveal in the random-cluster model the component $\cC$. In order for the event $A_{\mathsf{o}, h}$ to be possible, $\cC$ must be a set of plaquettes that confines a chain of vertices connecting $\mathsf{o}$ to height $h$. Crucially, the events $\cE^{\mathsf{w}}$ and $\cE^\dob$ were defined so that the set of plaquettes $C^*$ such that $\pi^{\mathsf{w}}_n(\cC^* = C^* \mid \cE^{\mathsf{w}})$ and $\bar\pi^\dob_n(\cC^* = C^* \mid \cE^\dob)$ both have positive probability is the same, let this set be denoted $\mathbf{C}^*$. In particular, $C^* \in \mathbf{C}^*$ if $C^*$ is a 1-connected component of plaquettes in $\mathbf{C}_\vee$ that forms part of a bounding surface in the sense that there can be no path in $\cL_{>0}$ from $\mathsf{o}$ to $\partial \Lambda_{n,n}'$ which does not cross a plaquette of $C^* \cup C_1$. Thus, we can express  
\begin{align*}
    \bP_n^\Red(A_{\mathsf{o}, h} \mid \cE^{\mathsf{w}}) & = \sum_{C^* \in \mathbf{C}^*} \pi^{\mathsf{w}}_n(\cC^* = C^* \mid \cE^{\mathsf{w}}) \bP_n^\Red(A_{\mathsf{o}, h} \mid \cC^* = C^*, \cE^{\mathsf{w}})\,, \\ 
    \bP_n^\dob(A_{\mathsf{o}, h} \mid \cE^\dob) & = \sum_{C^* \in \mathbf{C}^*} \bar\pi^\dob_n(\cC^* = C^* \mid \cE^\dob) \bP_n^\dob(A_{\mathsf{o}, h} \mid \cC^* = C^*, \cE^\dob)\,.
\end{align*}
Now, first studying the latter terms in the sums concerning the event $A_{\mathsf{o}, h}$, the difference between conditioning on $\cE^{\mathsf{w}}$ and $\cE^\dob$ is that on $\cE^\dob$ we know that the color $\sigma_{\mathsf{o}}$ is $\Blue$ (because of $\cE_3^\dob$), whereas $\cE^{\mathsf{w}}$ does not provide any information about the color $\sigma_{\mathsf{o}}$, so the conditional probability that $o\in \Vred^c$ (which on $\cE^{\mathsf{w}}$ is equivalent to $\sigma_{\mathsf{o}} = \noR$) is $\frac{q-1}{q}$. Given $\mathcal C^* = C^*$, its outer boundary is entirely in $\Vred$ except $\mathsf{o}$ which is not in $\Vred$. Therefore, membership in $\Vred$ interior to $\mathcal C^*$ is fully measurable with respect to the coloring on finite components in $\mathcal C^*$, and so the event $A_{\mathsf{o}, h}$ is measurable with respect to the same. Moreover, the coloring in $A_{\mathsf{o},h}$ is independent of the realization of the rest of $\cC$. Hence, we have
\[\bP_n^\Red (A_{\mathsf{o}, h} \mid \cC^* = C^*, \cE^{\mathsf{w}}) = \bP_n^\dob(A_{\mathsf{o}, h} \mid \cC^* = C^*, \cE^\dob)\frac{q-1}{q}\,.\]

Thus, it suffices to show that the random-cluster Radon--Nikodym derivative satisfies
\begin{align}\label{eq:conditional-ratio-of-C*}
    \frac{\pi^{\mathsf{w}}_n(\cC^* = C^* \mid \cE^{\mathsf{w}})}{\bar\pi^\dob_n(\cC^* = C^* \mid \cE^\dob)} = 1\pm \epsilon_\beta \qquad  \text{for every $C^* \in \mathbf{C}^*$}\,.
\end{align}
Beginning with the numerator, the event $\cE^{\mathsf{w}}$ is equivalent to asking that $\cC^* \in \mathbf{C}^*$, and that $\cC \setminus \cC^*$ is equal to the five plaquettes to the sides and below $\mathsf{o}$. By maximality, this requires that the other plaquettes 1-connected to $f_{(\mathsf{o}, \mathsf{o} - \ez)}$ are open. Let $F^{\mathsf{w}}$ be the union of these plaquettes with $\cC \setminus \cC^*$. We can then write $\pi^{\mathsf{w}}_n(\cdot \mid \cE^{\mathsf{w}})$ as $\pi^\eta_{\Lambda_{n,n}' \setminus F^{\mathsf{w}}}(\cdot \mid \cC^* \in \mathbf{C}^*)$, where $\eta$ is the boundary condition induced by the realization of $\cC \setminus \cC^*$ on $\cE^{\mathsf{w}}$.

For the denominator, the event $\cE^\dob$ is equivalent to asking that $\cC^* \in \mathbf{C}^*$, and that the set of plaquettes $\cC \setminus \cC^*$ satisfies a set of criterion independent of $\cC^*$. (Namely, that $\cC \setminus \cC^*$ contains $C_1$, is a subset of $\mathbf{C}_{\curlyvee} \cup C_1$, does not cut off $\mathsf{o}$ from $\partial_\Blue \Lambda_{n,n}'$, and combines with $\cC^*$ such that $\cC$ is an interface satisfying $\sep_{n, n}$. This last event appears as though the choice of $\cC \setminus \cC^*$ affects the possible choices of $\cC^*$, but the cut-point at $\mathsf{o}$ means it does not.) Exposing $\cC \setminus \cC^*$ reveals a set of dual-to-open plaquettes around it by maximality. Let $F^\dob$ be the union of these plaquettes with $\cC \setminus \cC^*$. Noting that the separation event $\sep_{n, n}$ is included in $\cE^\dob$, we can write $\bar\pi^\dob_n(\cdot \mid \cE^\dob)$ as $\E_{\cC \setminus \cC^*}[\pi^\gamma_{\Lambda_{n,n}' \setminus F^\dob}(\cdot \mid \cC^* \in \mathbf{C}^*) \mid \cE^\dob]$ where $\gamma$ is the boundary condition induced by revealing $\cC \setminus \cC^*$.

It thus suffices to show that for any $C^* \in \mathbf{C}^*$, and any realization of $\gamma$ under $\E_{\cC \setminus \cC^*}$,
\begin{align}\label{eq:ratio-of-C*}
    \frac{\pi^\eta_{\Lambda_{n,n}' \setminus F^{\mathsf{w}}}(\cC^* = C^*)}{\pi^\gamma_{\Lambda_{n,n}' \setminus F^\dob}(\cC^* = C^*)} = 1 \pm \epsilon_\beta\,,
\end{align}
as summing over $\cC^*$ then implies that $\pi^\eta_{\Lambda_{n,n}' \setminus F^{\mathsf{w}}}(\cC^* \in \mathbf{C}^*) = \pi^\gamma_{\Lambda_{n,n}' \setminus F^\dob}(\cC^* \in \mathbf{C}^*)(1 \pm \epsilon_\beta)$, which in turn implies \cref{eq:ratio-of-C*} even conditional on the event $\cC^* \in \mathbf{C}^*$ up to a change of the $\epsilon_\beta$. Then, averaging over $\gamma$ and applying the equivalence of measures discussed above implies \cref{eq:conditional-ratio-of-C*}.

Finally, to show \cref{eq:ratio-of-C*}, observe that each $C^* \in \mathbf{C}^*$ is a compatible set for $(\Lambda_{n,n}' \setminus F^{\mathsf{w}}, \eta)$ and $(\Lambda_{n,n}'\setminus F^\dob, \gamma)$, so we can apply \cref{prop:cluster-expansion} to obtain that
\begin{align*}
    \Big|\log\frac{\pi^\eta_{\Lambda_{n,n}' \setminus F^{\mathsf{w}}}(\cC^* = C^*)}{\pi^\gamma_{\Lambda_{n,n}' \setminus F^\dob}(\cC^* = C^*)} - \log \frac{Z^{\mathsf{w}}(\Lambda_{n,n}' \setminus F^{\mathsf{w}})}{Z^\eta(\Lambda_{n,n}' \setminus F^{\mathsf{w}})}\frac{Z^\gamma(\Lambda_{n,n}' \setminus F^\dob)}{Z^{\mathsf{w}}(\Lambda_{n,n}' \setminus F^\dob)}\Big| \leq \exp\Big(\sum_{f\in C^*} \sum_{g\in F^{\mathsf{w}} \cup F^\dob} e^{ -\beta d(f,g)/K}\Big)\,.
\end{align*}
In $\Lambda_{n,n}' \setminus F^{\mathsf{w}}$ with boundary condition $\eta$, there are two components of boundary vertices which are not wired together. The weight of a configuration $\omega$ in $Z^{\mathsf{w}}(\Lambda_{n,n}' \setminus F^{\mathsf{w}})$ is the same as in $Z^\eta(\Lambda_{n,n}' \setminus F^{\mathsf{w}})$ when these two components are wired together by $\omega$, and off by a factor of $q$ when they are not. However, the $\pi^\eta_{\Lambda_{n,n}' \setminus F^{\mathsf{w}}}$--probability that the two boundary components are not wired together is $\epsilon_\beta$ (this is true whenever the two boundary components can be connected by a path of constant length, as in this case, as the probability that an edge is open is always $\geq \frac{p}{p+q(1-p)} = 1 - \epsilon_\beta$ even conditional on a worst case configuration for the rest of the edges). This implies that the ratio of the two partition functions is at least $1 - \epsilon_\beta$, and similarly, $Z^\gamma(\Lambda_{n,n}' \setminus F^\dob)/Z^{\mathsf{w}}(\Lambda_{n,n}' \setminus F^\dob) \leq 1 + \epsilon_\beta$. Finally, the double summation above is at most $\epsilon_\beta$ per~\eqref{eq:sum-of-g-terms-cones}, concluding the proof.
\end{proof}

\subsection{Exact additivity of the rates}
We finally establish additivity (up to a $1\pm \epsilon_\beta$) of the sequence $\xi_h$; this allows us to replace $\xi_h$ by the limiting $\xi h$ without loss in our main theorem. We again truncate our consideration to high probability events on the interface's large deviation, and then establish the additivity on those events. The events will be of a similar cone-form to those in the previous subsection. 

It will actually be better to work with $\widetilde{\xi}_h$ which we define exactly as in~\eqref{eq:point-to-plane-connectivity}, but with an extra conditioning on $\{\sigma_{\mathsf{o} - \ez} = \noR\}$ (or equivalently since it is only one site, by color symmetry, being $\Blue$). This will turn out to be exactly additive. We first show the following relating $\xi_h$ to $\widetilde{\xi}_h$.

\begin{lemma}\label{lem:xi-tilde-related-to-xi}
    We have that for all $h \ge 1$, 
    \begin{align*}
        |\widetilde \xi_{h} - ({\xi}_h  - 2\beta + \log(q-1))|\le \epsilon_\beta \qquad \text{and thus by Lemma~\ref{lem:alpha-related-to-xi-Potts}} \quad |\widetilde \xi_{h} - \alpha_h|\le \epsilon_\beta\,.
    \end{align*}
\end{lemma}
\begin{proof}
    The cone event $\cE^{\mathsf{w}}$ implies that in the Potts marginal $\mu^\Red_n$, on the event $A_{\mathsf{o}, h}$, the $\noR$ component of $\mathsf{o}$ contains just a single vertex at height 1/2. From there, it is not hard to see by a Peierls map that for any $h\ge 1$, under $\mu_n^\Red(\cdot \mid A_{\mathsf{o}, h})$, the probability that $\mathsf{o}-\ez$ is $\noR$ is $(1\pm \epsilon_\beta) e^{ - 4\beta}$. A simple Peierls argument implies that unconditionally, the probability that $\mathsf{o} - \ez$ is $\noR$ is $(q-1)(1\pm \epsilon_\beta) e^{ - 6\beta}$. We then have
    \begin{align*}
        \mu_n^\Red(A_{\mathsf{o}, h} \mid \sigma_{\mathsf{o} - \ez}=\noR) = \frac{\mu_n^\Red(\sigma_{\mathsf{o} - \ez} = \noR\mid A_{\mathsf{o}, h})}{\mu_n^\Red(\sigma_{\mathsf{o} - \ez} =\noR)} \mu_n^\Red(A_{\mathsf{o}, h})= \frac{1}{q-1}(1\pm \epsilon_\beta) e^{2\beta}\mu_n^\Red(A_{\mathsf{o}, h})\,,
    \end{align*}
    which when rearranging and taking logarithms implies the claimed bound.     
\end{proof}

\begin{lemma}\label{lem:xi-tilde-additive}
    We have for all $h_1,h_2$, 
    \begin{align*}
        |\widetilde{\xi}_{h_1 + h_2}  -(\widetilde{\xi}_{h_1} + \widetilde{\xi}_{h_2})|\le \epsilon_\beta\,.
    \end{align*}
\end{lemma}
\begin{proof}
    Let $\widetilde \mu_n^{\Red}$ be the distribution $\mu_n^\Red$ conditioned on $\sigma_{\mathsf{o} - \ez}$ being $\noR$ (or equivalently by symmetry, being $\Blue$). This implies that in the corresponding joint measure $\widetilde\bP^\Red_n$, there is no $\omega$-path from $\mathsf{o} - \ez$ to $\partial \Lambda_{n,n}'$, or equivalently there is a confining set of dual-to-closed plaquettes which surrounds $\mathsf{o} - \ez$. Let $\cC$ be the 1-connected component of dual-to-closed plaquettes which contains the innermost such confining set. Call a vertex $y$ at a height $h$ a cut-point if it is the unique vertex at height $h$ which is separated from $\partial \Lambda_{n,n}'$ by $\cC$, and the only plaquettes of $\cC$ which have height $h$ are the four vertical plaquettes bounding the sides of $y$. Call $h$ a cut-height if there is a cut-point at height $h$.
    
    We define certain high probability cone events in the random-cluster model, which we will show all have probability $1-\epsilon_\beta$ conditional on $A_{\mathsf{o}, h_1+h_2}, A_{\mathsf{o}, h_1}, A_{\mathsf{o}, h_2}$ respectively in $\widetilde\bP^\Red_n$. We then show the approximate additivity of the rates under those events. The events are defined as follows: 
    \begin{enumerate}
        \item $\mathcal E_{\mathsf{Y}}$: is the event that heights $h_1 - \frac12, \ldots, h_1 - \frac12 + (L_\beta \wedge h_2)$ are cut-heights, and if $y = (Y_1, Y_2, h_1 - \frac12)$ is the unique vertex at height $h_1 - \frac12$ confined by $\cC$, then $\cC \cap \cL_{\geq h_1}$ lies inside the cone
        \begin{align*}
            \{(z_1,z_2,z_3): (z_3 - h_1)^2 \ge \sqrt{(z_1-Y_1)^2 + (z_2 - Y_2)^2}\}\,.
        \end{align*}
        \item $\mathcal E_{\vee}$: heights $\frac12, \ldots, (L_\beta \wedge h_1) - \frac12$ are cut-heights, and $\cC \cap \cL_{\geq 0}$ lies inside the cone 
        \begin{align*}
            \{(z_1,z_2,z_3): z_3^2 \ge \sqrt{z_1^2 + z_2^2}\}\,.
        \end{align*}
        \item $\cE_\parallel$: $h_1 - \frac12$ is a cut-height.
    \end{enumerate}
    In the same way that we described \cref{eq:E_vee-high-probability} following from the ``isolated pillar" results of \cite{ChLu24} now that we have the generalized cluster expansion for $\cC$, we also have the following bounds:
    \begin{align}\label{eq:iso-typical-for-additivity}
        \widetilde\bP_n^\Red(\cE_{\mathsf{Y}} \mid A_{\mathsf{o}, h_1+h_2}) \ge 1-\epsilon_\beta \qquad \widetilde\bP_n^\Red(\cE_{\vee} \mid A_{\mathsf{o}, h_2})\ge 1-\epsilon_\beta \qquad \widetilde\bP_n^\Red(\cE_\parallel \mid A_{\mathsf{o}, h_1})\ge 1-\epsilon_\beta\,.
    \end{align}
     
    It thus suffices to show 
    \begin{align}\label{eq:nts-additivity-conditioning}
        \frac{\widetilde\bP_n^\Red(\cE_{\mathsf{Y}} , A_{\mathsf{o}, h_1+h_2})}{\widetilde\bP_n^\Red(\cE_{\vee} , A_{\mathsf{o}, h_2})\widetilde\bP_n^\Red(\cE_\parallel , A_{\mathsf{o}, h_1})}  = \frac{\widetilde\bP_n^\Red(\cE_{\mathsf{Y}},A_{\mathsf{o}, h_1+h_2}\mid \cE_\parallel,A_{\mathsf{o}, h_1})}{\widetilde\bP_n^\Red(\cE_{\vee} , A_{\mathsf{o}, h_2})}= 1\pm \epsilon_\beta\,.
    \end{align}
    The first equality here is simply the fact that $\cE_{\mathsf{Y}}\cap A_{\mathsf{o}, h_1+h_2}\subset \cE_\parallel\cap A_{\mathsf{o}, h_1}$. Now we study the conditional joint measure $\widetilde\bP_n^\Red(\cdot \mid \cE_\parallel, A_{\mathsf{o}, h_1})$. We can expose the component of $\cC$ which is below the cut-height  at $h_1 - \frac12$, which reveals additionally a set of dual-to-open plaquettes around it by maximality of the component. This determines the set of vertices $V$ which are in a finite component of $\cC$ and in $\cL_{< h_1}$. We furthermore reveal the colors of $V$. Let $\eta$ be the FK--Potts boundary condition described by the above revealing, together with the wired red boundary at $\partial \Lambda_{n,n}'$. Let $E$ be the set of plaquettes which are fixed by $\eta$. Then, we can write \begin{align}\label{eq:additivity-domain-markov}\widetilde\bP^\Red_n(\cdot \mid \cE_\parallel, A_{\mathsf{o}, h_1}) = \E_\eta[\bP^\eta_{\Lambda_{n,n}' \setminus E}(\cdot)\mid \cE_\parallel, A_{\mathsf{o}, h_1}, \sigma_{0, 0, -\frac12} = \noR]\,,
    \end{align}
    as the events $\cE_\parallel$, $A_{\mathsf{o}, h_1}$, and $\sigma_{\mathsf{o} - \ez} = \noR$ are implied by the closed edges and colors given by $\eta$. 
    
    Now notice that given such a realization $\eta$ satisfying $\cE_\parallel\cap A_{\mathsf{o}, h_1}$, the event $\cE_{\mathsf{Y}}\cap A_{\mathsf{o}, h_1+h_2}$ is exactly equivalent to the translation by $(Y_1, Y_2, h_1)$ (measurable with respect to $\Gamma$) of the pair of events $\cE_{\vee} \cap A_{\mathsf{o}, h_2}$. By this translation, and denoting $\eta', E'$ as the shifted versions (by $(-Y_1, -Y_2, -h_1)$) of $\eta, E$, we have
    \begin{align*}
        \bP^\eta_{\Lambda_{n,n}' \setminus E}(\cE_{\mathsf{Y}}, A_{0, h_1 + h_2}) = \bP^{\eta'}_{\Lambda_{n,n}' \setminus E'}(\cE_{\vee}, A_{\mathsf{o}, h_2})\,.
    \end{align*}
    For an analogous setup in the denominator of \cref{eq:nts-additivity-conditioning}, we want to expose the component of $\cC$ which is below height 0. Let $\cE_{\sqcup}$ denote the event that $\cC \cap \cL_{< 0}$ consists only of the five plaquettes which are to the sides of and below $\mathsf{o} - \ez$. A standard Peierls map argument in the FK--Potts model shows that we have
    \begin{align*}
        \widetilde\bP^\Red_n(\cE_\sqcup) \geq 1- \epsilon_\beta \qquad \text{and} \qquad \widetilde\bP^\Red_n(\cE_\sqcup \mid \cE_\vee, A_{\mathsf{o}, h_2}) \geq 1- \epsilon_\beta\,. 
    \end{align*}
    (Indeed, for the first bound, the map would be to delete the outermost $1$-connected component of dual-to-closed plaquettes, and if this causes $\mathsf{o}-\ez$ to be in the same connected component as $\partial \Lambda_{n,n}'$, then force its six incident edges to all be closed, so that it can be assigned the requisite $\noR$ color. For the second one, the map would be to delete the $1$-connected component in $\cL_{< 0}$ of dual-to-closed plaquettes incident to $\mathsf{o}-\ez$, and replace it by the minimal five dual-to-closed plaquettes incident to $\mathsf{o}- \ez$ in $\cL_{<0}$.)
    This implies that
    \begin{align*}
        \widetilde\bP^\Red_n(\cE_\vee, A_{\mathsf{o}, h_2}\mid \cE_\sqcup) = \widetilde\bP^\Red_n(\cE_\vee, A_{\mathsf{o}, h_2})\frac{\widetilde\bP^\Red_n(\cE_\sqcup \mid \cE_\vee, A_{\mathsf{o}, h_2})}{\widetilde\bP^\Red_n(\cE_\sqcup)} = \widetilde\bP^\Red_n(\cE_\vee, A_{\mathsf{o}, h_2})(1\pm \epsilon_\beta)\,.
    \end{align*}
    Note that conditioning on $\cE_\sqcup$ is equivalent to conditioning on the aforementioned five plaquettes to be dual-to-closed, as well as on the plaquettes that are 1-connected to them and have height $< 0$ to be dual-to-open. Let $\zeta$ be the FK--Potts boundary condition with these dual-to-open and dual-to-closed plaquettes, the wired $\Red$ boundary at $\partial \Lambda_{n,n}'$, and $\sigma_{\mathsf{o} - \ez} = \Blue$ (by symmetry, any $\noR$ color is equivalent). Let $F$ be the set of plaquettes fixed by $\zeta$. Then we can write
    \begin{align*}
        \widetilde\bP^\Red_n(\cdot \mid \cE_\sqcup) = \bP^\zeta_{\Lambda_{n,n}' \setminus F}(\cdot)\,.
    \end{align*}
    
    Now, we observe that the random-cluster marginal of $\bP^{\eta'}_{\Lambda_{n,n}' \setminus E'}$ has the same law as the random-cluster measure on $\Lambda_{n,n}' \setminus E'$ with all the dual-to-open and dual-to-closed plaquettes in $\eta'$ as a boundary condition, further conditional on the event that there is no $\omega$-path from $\mathsf{o} - \ez$ to $\partial \Lambda_{n,n}'$. Call this random-cluster measure $\pi^{\eta'(\mathsf{RC})}_{\Lambda_{n,n}' \setminus E'}(\cdot \mid \sep_\mathsf{o})$, where $\sep_\mathsf{o}$ denotes this separation event for $\mathsf{o} - \ez$.
    Furthermore, by the FK--Potts coupling, the joint configuration under $\bP^{\eta'}_{\Lambda_{n,n}' \setminus E'}$ can be sampled by first sampling a random-cluster configuration $\omega$ under $\pi^{\eta'(\mathsf{RC})}_{\Lambda_{n,n}'\setminus E'}( \cdot \mid \sep_\mathsf{o})$, assigning vertices in the same $\omega$-component of $\mathsf{o} - \ez$ the same $\noR$ color as given by $\eta'$ (by symmetry, we may assume this is $\Blue$), assigning vertices in the $\omega$-component of $\partial \Lambda_{n,n}'$ the color $\Red$, and assigning all other $\omega$-components colors chosen independently at random. In the same way (with analogous notation), we can sample a joint configuration under $\bP^\zeta_{\Lambda_{n,n}' \setminus F}$ by first sampling a random-cluster configuration $\omega$ under $\pi^{\zeta(\mathsf{RC})}_{\Lambda_{n,n}' \setminus F}(\cdot \mid \sep_\mathsf{o})$, assigning vertices in the same $\omega$-component of $\mathsf{o} - \ez$ the color $\Blue$, assigning vertices in the $\omega$-component of $\partial \Lambda_{n,n}'$ the color $\Red$, and assigning all other $\omega$-components colors chosen independently at random.
    
    Now, let $\cC^*:= \cC \cap \cL_{\geq 0}$. Let $\mathbf{C}^*$ be the set of possible realizations of $\cC^*$ satisfying $\cE_\vee$ and $\sep_\mathsf{o}$ when combined with the boundary conditions $\eta'(\mathsf{RC})$ or $\zeta(\mathsf{RC})$ (note that because $\mathsf{o} - \ez$ is a cut-point for $\cC$ under both boundary conditions, the set $\mathbf{C}^*$ will be the same). We can now write 
    \begin{align*}
        \bP^{\eta'}_{\Lambda_{n,n}' \setminus E'}(\cE_\vee, A_{\mathsf{o}, h_2}) = \sum_{C^* \in \mathbf{C}^*} \pi^{\eta'(\mathsf{RC})}_{\Lambda_{n,n}' \setminus E'}(\cC^* = C^* \mid \sep_\mathsf{o})\bP^{\eta'}_{\Lambda_{n,n}' \setminus E'}(A_{\mathsf{o}, h_2} \mid \cC^* = C^*)\,,\\
        \bP^{\zeta}_{\Lambda_{n,n}' \setminus F}(\cE_\vee, A_{\mathsf{o}, h_2}) = \sum_{C^* \in \mathbf{C}^*} \pi^{\zeta(\mathsf{RC})}_{\Lambda_{n,n}' \setminus F}(\cC^* = C^* \mid \sep_\mathsf{o})\bP^{\zeta}_{\Lambda_{n,n}' \setminus F}(A_{\mathsf{o}, h_2} \mid \cC^* = C^*)\,.
    \end{align*}
    As $A_{\mathsf{o}, h_2}$ now depends only on the coloring of the finite components enclosed by $\cC^*$, which is independent of the rest of the FK--Potts model, we have
    \begin{align*}
        \bP^{\eta'}_{\Lambda_{n,n}' \setminus E'}(A_{\mathsf{o}, h_2} \mid \cC^* = C^*) = \bP^{\zeta}_{\Lambda_{n,n}' \setminus F}(A_{\mathsf{o}, h_2} \mid \cC^* = C^*)\,.
    \end{align*}
    Hence, in order to prove \cref{eq:nts-additivity-conditioning}, it suffices to prove that
    \begin{align}\label{eq:additivity-CE-ratio-conditional}
        \frac{\pi^{\eta'(\mathsf{RC})}_{\Lambda_{n,n}' \setminus E'}(\cC^* = C^* \mid \sep_\mathsf{o})}{\pi^{\zeta(\mathsf{RC})}_{\Lambda_{n,n}'\setminus F}(\cC^* = C^* \mid \sep_\mathsf{o})} = (1 \pm \epsilon_\beta)\,,
    \end{align}
    at which point we can conclude via the sequence of equalities between \cref{eq:additivity-domain-markov} and \cref{eq:additivity-CE-ratio-conditional}, and averaging over $\eta$ in \cref{eq:additivity-domain-markov}. Now, since each $\cC^* \in \mathbf{C}^*$ is a compatible set for $(\Lambda_{n,n}' \setminus E', \eta')$ and $(\Lambda_{n,n}' \setminus F, \zeta)$, we can apply \cref{prop:cluster-expansion} with the same argument used to show \cref{eq:ratio-of-C*}, getting that
    \begin{align*}
        \frac{\pi^{\eta'(\mathsf{RC})}_{\Lambda_{n,n}' \setminus E'}(\cC^* = C^*)}{\pi^{\zeta(\mathsf{RC})}_{\Lambda_{n,n}' \setminus F}(\cC^* = C^*)} = (1 \pm \epsilon_\beta)\,.
    \end{align*}
    Hence, it remains to show that
    \begin{align*}
        \frac{\pi^{\eta'(\mathsf{RC})}_{\Lambda_{n,n}' \setminus E'}(\sep_\mathsf{o})}{\pi^{\zeta(\mathsf{RC})}_{\Lambda_{n,n}' \setminus F}(\sep_\mathsf{o})} = (1 \pm \epsilon_\beta)\,.
    \end{align*}
    However, a Peierls type argument shows that if $f$ is the plaquette above $\mathsf{o} - \ez$, then
    \begin{align*}
        \pi^{\eta'(\mathsf{RC})}_{\Lambda_{n,n}' \setminus E'}(\sep_\mathsf{o}) = \pi^{\eta'(\mathsf{RC})}_{\Lambda_{n,n}' \setminus E'}(f \text{ is closed})(1\pm \epsilon_\beta) \quad \text{and} \quad \pi^{\zeta(\mathsf{RC})}_{\Lambda_{n,n}' \setminus F}(\sep_\mathsf{o}) = \pi^{\zeta(\mathsf{RC})}_{\Lambda_{n,n}' \setminus F}(f \text{ is closed})(1\pm \epsilon_\beta)\,.
    \end{align*}
    Finally, the proof concludes by observing that the probability that $f$ is closed under both measures is the same, being equal to $(qe^{-\beta})/(1-e^{-\beta} + qe^{-\beta})$.
\end{proof}

\begin{proof}[\textbf{\emph{Proof of \cref{thm:identifying-rate-with-infinite-limit}}}]
By Lemma~\ref{lem:xi-tilde-related-to-xi}, we have that $\lim_{h\to\infty} \xi_h/h = \lim_{h\to\infty}\tilde\xi_h/h=\xi$, and by Lemma~\ref{lem:xi-tilde-additive} together with Fekete's lemma, the latter limit exists so it must be what we defined to be $\xi$ in~\eqref{eq:xi}. Moreover, by Lemma~\ref{lem:xi-tilde-additive}, $|\tilde{\xi}_h - \xi h| < \epsilon_\beta$ uniformly over $h\ge 1$ (indeed, if $a_{h_1+h_2}\leq a_{h_1}+a_{h_2}+\epsilon$ for all $h_1,h_2$ then $\tilde a_h := a_h + \epsilon$ is sub-additive, thus by Fekete's Lemma, $a=\lim \tilde a_h/h=\inf \tilde a_h/h$, that is, $a_h \geq a h - \epsilon$ for all~$h$; similarly, $a_h \leq a h + \epsilon$ for all $h$ whenever we have $a_{h_1+h_2}\geq a_{h_1}+a_{h_2}-\epsilon$ for all $h_1,h_2$). The bound on $\alpha_h$ then follows from the second part of Lemma~\ref{lem:xi-tilde-related-to-xi}. 
\end{proof}

\appendix

\section{Generalized cluster expansion}\label{sec:appendix}

In this section, we provide the proof of \cref{prop:cluster-expansion}. The following two lemmas were proved by Grimmett in \cite{Grimmett_RC} (with a minor adjustment to the latter, see the proof of \cref{lem:f-term-decay} below).
\begin{lemma}[{\cite[Lem.~7.90]{Grimmett_RC}}]\label{lem:free-energy-expansion}
    Let $\cU=(V, E)$ be a connected finite graph and $\eta$ a boundary condition on $\Gamma$ where $\cU \subset  \Gamma$. Then,
    \begin{equation*}
        \log Z^\eta_\cU(p, q) = \log q + \sum_{e \in E} \f^\eta_{\cU, p, q}(e)\,,
    \end{equation*}
    where \[\f^\eta_{\cU, p, q}(e) := \int_p^1 \frac{s - \pi^\eta_{\cU, s, q}(\omega_e = 1)}{s(1-s)} \,\d s\,.\]
    Moreover, we have $0 \leq \f^\eta_{\cU, p, q}(e) \leq (1-p)(q-1)$ for any $\cU, \eta$.
\end{lemma}

\begin{lemma}[{\cite[Lem.~7.93]{Grimmett_RC}}]\label{lem:f-term-decay}
    Let $\cU, \cU' \subset  \Z^3$ be two sub-graphs. Let $e \in \cU$, $e' \in \cU'$ be two edges. Suppose that there exists $r \geq 1$ such that $B_r(e) \cap \cU \cong B_r(e') \cap \cU'$. Then, there exists a constant $K > 0$ such that for $\beta$ sufficiently large, 
    \begin{equation*}
         |\f^{\mathsf{w}}_{\cU, p, q}(e) - \f^{\mathsf{w}}_{\cU', p, q}(e')| \leq e^{-\beta r/K}\,.
     \end{equation*}
\end{lemma}
\begin{proof}
    We note that \cref{lem:f-term-decay} was originally stated with a exponential bound of $e^{-r/K}$ for all $\beta \geq \beta_0$ instead of $e^{-\beta r/K}$ as above. We show here however that the stronger bound can be concluded from their proof. As was done, we can write 
    \[\frac{s - \pi^\eta_{\cU, s, q}(\omega_e = 1)}{s(1-s)} = \frac{(q-1)(1 - \pi^\eta_{\cU, s, q}(K_e))}{s + q(1-s)}\]
    where $K_e$ is the event that the endpoints of $e$ are joined by an open path without using the edge $e$. Now let $\rho_s = \frac{s}{s+q(1-s)}$. It was then shown that for some $c_1, c_2 > 0$, we have
    \[|\pi^\eta_{\cU, s, q}(K_e) - \pi^\eta_{\cU, s, q}(K_e)| \leq c_1^r(1 - \rho_s)^{r/c_2}\,.\]
    Note that $\rho_s$ is increasing for $s < 1$, and hence is minimized at $\rho_p$ on the interval $[p, 1)$. One can compute that $1 - \rho_p \geq \tfrac{q}{2}e^{-\beta}$. Hence, we can conclude that
    \begin{equation*}
        |\f^{\mathsf{w}}_{\cU, s, q}(e) - \f^{\mathsf{w}}_{\cU', s, q}(e')| \leq c_1^r(\tfrac{q}{2})^{r/c_2}e^{-\beta r/c_2}\int_p^1 \frac{q-1}{s+q(1-s)}\,\d s\,.
    \end{equation*}
    The integral above is a constant, so it and the term $c^r(\tfrac{q}{2})^{r/c}$ can be absorbed into the term $e^{-\beta r/c_2}$ by increasing the constant $c_2$ to some sufficiently large (but $\beta$ independent) constant $K$. 
\end{proof}

\begin{proof}[\textbf{\emph{Proof of \cref{prop:cluster-expansion}}}]
    We follow the proof of \cite[Lem.~7.104]{Grimmett_RC}, with a few minor adjustments. Let $\cU \setminus \overline{F}$ denote the graph resulting by removing from $\cU$ the edge set $\{e: f_e \in \overline{F}\})$. Then, the fact that $F$ is a compatible set for $(\cU, \eta)$ means that conditioning on $\mathsf{Cl}(F)$ induces wired boundary conditions on each component of $\cU \setminus \overline{F}$. Thus, as in \cite[Eq.~7.102]{Grimmett_RC}, we can write
    \begin{equation}\label{eq:prob-as-ratio-free-energy}
        \pi^\eta_\cU(\mathsf{Cl}(F)) = \frac{Z^{\mathsf{w}}(\cU)}{Z^\eta(\cU)}\frac{Z^{\mathsf{w}}(\cU \setminus \overline{F})}{Z^{\mathsf{w}}(\cU)}(1 - e^{-\beta})^{|\partial F|}e^{-\beta|F|}q^{\kappa(F) - 1}\,.
    \end{equation}
    Applying \cref{lem:free-energy-expansion}, we get that
    \begin{equation}\label{eq:ratio-of-free-energy}
        \log\left(\frac{Z^{\mathsf{w}}(\cU \setminus \overline{F})}{Z^{\mathsf{w}}(\cU)}\right) = \sum_{e \in \cU \setminus \overline{F}} (\f^{\mathsf{w}}_{\cU \setminus \overline{F}, p, q}(e) - \f^{\mathsf{w}}_{\cU, p, q}(e)) - \sum_{e: f_e \in \overline{F}} \f^{\mathsf{w}}_{\cU, p, q}(e)\,.
    \end{equation}
    We can group the edges in the above sums according to which plaquette of $F$ they are the closest to. Let $\phi(e, F)$ be the plaquette $f \in F$ which is closest to the edge $e$, with distance $d(e, f)$ given by the $L^\infty$ distance between their centers. Ties can be decided in any arbitrary predetermined way, e.g., choosing the plaquette that is earliest in lexicographic ordering. Then, define
    \begin{equation}\label{eq:g-term}
        \g(f, F, \cU) = \sum_{\substack{e \in \cU \setminus \overline{F},\\\phi(e, F) = f}} (\f^{\mathsf{w}}_{\cU \setminus \overline{F}, p, q}(e) - \f^{\mathsf{w}}_{\cU, p, q}(e)) - \sum_{\substack{e: f_e \in \overline{F},\\ \phi(e, F) = f}} \f^{\mathsf{w}}_{\cU, p, q}(e)\,.
    \end{equation}
    We claim that $\g$ satisfies the requirements of the lemma. Combining \cref{eq:prob-as-ratio-free-energy,eq:ratio-of-free-energy,eq:g-term} immediately proves \cref{it:cluster-expansion}.

    To show \cref{it:L-inf-bound}, let $\phi(e, F) = f$ and let $r = d(e, f)$. Then, $B_{r-2}(e)$ does not contain any plaquettes of $\overline{F}$, so by \cref{lem:f-term-decay} we get that for $r \geq 3$,
    \begin{equation}
        |\f^{\mathsf{w}}_{\cU \setminus \overline{F}, p, q}(e) - \f^{\mathsf{w}}_{\cU, p, q}(e)| \leq e^{-\beta(r-2)/K}\,.
    \end{equation}
    For $r = 1, 2$, we can apply the bound $\f^{\mathsf{w}}_{\cU, p, q}(e) \leq (q-1)e^{-\beta}$ from \cref{lem:free-energy-expansion}. Thus, using the fact that the number of edges $e \in \cU \setminus \overline{F}$ such that $\phi(e, F) = f$ and $d(e, f) = r$ is at most $c^r$ for some constant $c > 0$, we can sum over the exponential tails to bound the first sum in \cref{eq:g-term} by $e^{-\beta/K'}$ for some constant $K'$. The second sum in \cref{eq:g-term} has a constant number of terms, and can thus be bounded by $Ce^{-\beta}$ for some constant $C$, which can then be absorbed for large enough $\beta$ into the constant $K'$.

    Finally, the proof of \cref{it:exp-decay} follows in the same way as in \cite[Lem.~7.104]{Grimmett_RC} with similar adjustments as above. Indeed, let $r$ be the largest radius such that $B_r(f) \cap \cU \cong B_r(f') \cap \cU'$ and $B_r(f) \cap F \cong B_r(f') \cap F'$. The original proof showed that $|\g(f, F, \cU) - \g(f', F', \cU')| \leq e^{-r/K}$ for $r > 9$. By applying the stronger, $\beta$-decaying bound in \cref{lem:f-term-decay} whenever needed, the same proof now shows that $|\g(f, F, \cU) - \g(f', F', \cU')| \leq e^{-\beta r/K}$. We can then apply \cref{it:L-inf-bound} to extend this bound to $r \leq 9$ as well, at the cost of increasing the constant $K$.
\end{proof}

\section{Proof of \cref{thm:gg-interface-height}}\label{sec:pf-of-gg-interface-height}
Recall the definitions of walls and ceilings, and their excess areas from \cref{subsec:ld-upper-bound}. Recall from~\cite{GielisGrimmett02} the notion of groups of walls (collections of nearby walls clustered by a threshold on their interactions through $\mathbf{g}$ relative to their excess areas) that are essential to being able to execute the deletion operations and establish rigidity in $\bar \pi^0_{\Lambda'_{n,\infty}}$. 

The proof of \cref{thm:gg-interface-height} with $m=\infty$, found in~\cite{GielisGrimmett02}, goes by considering the Peierls-type map that deletes a group of walls $\mathfrak W$ from the standard wall representation of $\cI_{\Full}$, and then bounds the entropy lost from the operation (controlled by enumerating the number of possible groups of walls of size $k$ nesting a vertex $x$ by $C^k$) by the energetic gain from the deletion, which is given by $\exp( - \beta \sum_{W\in \mathfrak W} \fm(W))$. 

When the domain is capped at distance $m$ away from height zero, there is the potential concern of interactions with the hard floor (the extreme case of $m=0$ of course being the same form as the interactions that cause many of the difficulties encountered in this present paper). However, when it is at distance $\Omega(n)$ away in both directions, the interaction is fairly crudely bounded by the energetic cost of reaching up to such a high height. 

To begin, notice that for any configuration $\cI_{\Full}$ having $\overline{\hgt}(\cI_{\Full})= \max_{x} \overline{\hgt}_x(\cI_{\Full})\le m/4$, say, the operation of deleting a group of walls can induce a shift, at most, of $m/4$ on its internal walls, so after the operation, the interface will retain distance at least $m/2$ to the boundary conditions at heights $\pm m$. Then, per Remark~\ref{rem:diff-domains}, the only difference in the change of weights caused by this map, relative to the same map in $\Lambda_{n,\infty}'$ is an additive $C n^2 e^{ - n/2}$ from the radius $\mathbf{r}(f, I; f',I')$ possibly being attained from faces in the boundary. This is of course negligible compared to the cost of $\beta r$ for deletion of a group of wall of excess area $r$. 

Thus, without any change to the argument that takes one from Proposition~\ref{prop:grimmett-cluster-exp} to Theorem~\ref{thm:gg-interface-height}, we get for $m\ge 4k$, for every $x$, 
\begin{align*}
\bar\pi^0_{\Lambda'_{n,m}}(\overline\hgt_x(\cI_\Full)\geq k,\, \max_y \overline{\hgt}_y(\cI_{\Full})\le m/4 ) \leq \exp(-a_p k)\,,
\end{align*}
It therefore suffices to upper bound the probability of $\max_y \overline{\hgt}_y(\cI_{\Full})\ge m/4$, for $m \ge n$, by $\exp( - a_p' m)$ as that will be smaller than $\exp( - a_p k)$ up to changing the constant $a_p$. For that purpose, consider any interface with $\max_y \overline\hgt_y(\cI_{\Full}) \ge m/4$. There must be a nested sequence of groups of walls $\overline{\mathfrak W}$  (meaning the union of the groups of walls of a nested sequence of walls $W_1,...,W_N$) whose total excess area satisfies $\fm(\overline{\mathfrak W}) \ge m/4$. 

Consider the map $\Phi$ that iteratively, 
\begin{enumerate}
    \item Picks a nested sequence of groups of walls whose total excess area exceeds $m/4$ (if one exists) and deletes it from the standard wall representation; 
    \item If no more exist, stops and outputs the resulting interface as $\Phi(\cI_{\Full})$.
\end{enumerate}
Notice that the resulting interface must be a valid interface in $\Lambda_{n,m}'$ with $\overline{\hgt}(\Phi(\cI_{\Full}))\le m/4$. 
When considering the weight gain from this map, as compared to the infinite-cylinder analysis, the only difference is the possibility that the radius of congruence $\mathbf{r}(f,\cI_{\Full}; f', \Phi(\cI_{\Full})
)$ is attained by interactions between a face $f$ interior of a wall that was deleted, and the boundary conditions at heights $\pm m$. Indeed if $\cI_{\Full}$ is any such interface, and $\mathfrak W$ is the set of walls deleted, then by~\cite{GielisGrimmett02} 
\begin{align}\label{eq:infinite-strip-energy-change-appendix}
   \Bigg| \log \frac{\bar \pi_{\Lambda'_{n,\infty}}^0 (\cI_{\Full})}{\bar \pi_{\Lambda'_{n,\infty}}^0 (\Phi(\cI_{\Full}))}  + \beta \fm(\mathfrak W)\Bigg| \le C \fm(\mathfrak W)\,,
\end{align}
while at the same time, by Proposition~\ref{prop:grimmett-cluster-exp} and Remark~\ref{rem:diff-domains}, 
\begin{align}\label{eq:finite-vs-infinite-strip-energy-change}
    \Bigg| \log \frac{\bar \pi_{\Lambda_{n,m}'}^0(\cI_{\Full})}{\bar\pi_{\Lambda_{n,m}'}^0(\Phi(\cI_{\Full}))} - \log \frac{\bar \pi_{\Lambda'_{n,\infty}}^0 (\cI_{\Full})}{\bar \pi_{\Lambda'_{n,\infty}}^0 (\Phi(\cI_{\Full}))} \Bigg|\le Kn^2 e^{ - cm/4}+ \sum_{f\in \mathsf{Ceil}(\mathfrak W): |\hgt(f)|\ge \frac{m}{2}} K e^{ - cd(f, \cL_{\pm m})}\,.
\end{align}
where $c,K$ are the constants from Proposition~\ref{prop:grimmett-cluster-exp}, and the sum runs over $f\in \cI_{\Full}$ in ceilings interior to walls of $\mathfrak W$. At this point, we wish to bound the number of summands in~\eqref{eq:finite-vs-infinite-strip-energy-change} by $\fm(\fW)$. Notice that if we partition $A=\rho(\{f\in \mathsf{Ceil}(W): |\hgt(f)|\ge m/2\})$ where $\rho(\cdot)$ is used to denote projecting it down to faces of $\cL_0$, for any $M$ and any partition of $A$ into $\mathfrak C_1,...,\mathfrak C_M$, we have 
\begin{align}\label{eq:isoperimetry-interactions-by-partition}
    |A| \le \sum_{i=1}^M \diam(\mathfrak C_i)^2 \le \sum_{i=1}^M \diam(\mathfrak C_i)\cdot n\,,
\end{align}
(because at worst, we are overcounting faces with this bound). On the other hand, for every height between $j=1,...,m/2$, there must have been a collection of wall faces in $\mathfrak W$ nesting these ceilings, and therefore at each height, there was some partition $\mathfrak C_1^{j},...,\mathfrak C_{M_j}^j$ such that those deleted wall faces at height $j$ contributed at least $\sum_{i=1}^{M_j} \diam(\mathfrak C_{M_j}^j)$. By~\eqref{eq:isoperimetry-interactions-by-partition}, for each $j$, this is at least $|A|/n$ and we sum over $j=1,...,m/2$ to get that $\fm(\mathfrak W) \ge |A| \frac{m}{2n}$. Since $m\ge n/2$, this is at least $|A|/4$, so the term in~\eqref{eq:finite-vs-infinite-strip-energy-change} the right hand side is bounded by $K n^2 e^{ - cn/8} + K\fm(\fW)$, which combined with~\eqref{eq:infinite-strip-energy-change-appendix} implies the same bound as in that equation for $\bar \pi_{\Lambda_{n,m}'}^0$ up to a change of $C$ to $C'$ for another $\beta$ independent constant. 

The only other thing left to verify is that the entropy lost by deletion of these walls is controlled by $C^{\fm(\fW)}$ for a universal constant $C$. To see this, consider the number of pre-images of a $\mathcal J = \Phi(\cI_{\Full})$, having $\fm(\cI; \cJ) = L$. We first pick a number $I\in \{1,...,L/m\}$ of iterations of step (1) of the map $\Phi$ before it terminated; then pick $I$ many vertices $v_1,...,v_I$ in $[- \frac{n}{2} , \frac{n}2]^2 \times \{0\}$, allocate the total excess area $L$ between them with each being allocated at least $m/4$ into $L_1,...,L_I$, and then enumerate over nested sequences of groups of walls containing a vertex $v_i$ by $C^{L_i}$ per~\cite[Proof of Lemma 15]{GielisGrimmett02}. In total, this bounds the entropy loss by
\begin{align*}
    \frac{L}{m} \max_{I\le L/m} n^{2 I} 2^L \prod_{i=1}^I C^{L_i} \le C_1^L
\end{align*}
where we used the fact that $m\ge n/2$. Since we established that the weight change from such maps with $\fm(\cI;\cJ)=L$ is $\exp( - (\beta- C)L)$, summing over $L\ge m/4$ is at most $\exp ( - (\beta - C')n/4)$ concluding the proof. 
\qed

\subsection*{Acknowledgments}
The authors thank the anonymous referees for their careful reading and helpful comments. 
The research of R.G. is supported in part by NSF DMS-2246780. 
The research of E.L.\ was supported by the NSF grant DMS-2054833.

\bibliographystyle{abbrv}
\bibliography{Ising3D_floor}

\end{document}